\documentclass[11pt]{amsart} 
 \usepackage{hyperref}
 \usepackage[english]{babel}
\usepackage{amsmath,amsfonts,amssymb}
\usepackage{amsbsy,amsgen,amsopn,amscd,amstext,amsxtra}
\usepackage{amsthm}
\usepackage{graphicx}
\usepackage{amsfonts}
\usepackage{color}
\usepackage{pdfsync}
\usepackage{mathtools}
 \usepackage{pb-diagram}
 \usepackage{fullpage}

\theoremstyle{plain}

\newcommand{\AdS}{\mathop{{\rm AdS}}\nolimits}
\newcommand{\Sesh}{\mathop{{\rm Sesh}}\nolimits}

\newcommand\bu{{\bf{u}}}
\newcommand\bw{{\bf{w}}}
\newcommand\bv{{\bf{v}}}

\newcommand{\mlabel}[1]{\marginpar{#1}\label{#1}}

\newcommand{\g}{{\mathfrak g}}

\newcommand{\fg}{{\mathfrak g}}
\newcommand{\fh}{{\mathfrak h}}

\newcommand{\fk}{{\mathfrak k}}

\newcommand{\fq}{{\mathfrak q}}
\newcommand{\fp}{{\mathfrak p}}

\renewcommand{\:}{\colon}
\newcommand{\1}{\mathbf{1}}

\newcommand{\cB}{\mathcal{B}}

\newcommand{\cD}{\mathcal{D}}
\newcommand{\cE}{\mathcal{E}}

\newcommand{\cG}{\mathcal{G}}
\newcommand{\cH}{\mathcal{H}}
\newcommand{\cK}{\mathcal{K}}
\newcommand{\cL}{\mathcal{L}}

\newcommand{\cN}{\mathcal{N}}
\newcommand{\cO}{\mathcal{O}}
\newcommand{\cP}{\mathcal{P}}

\newcommand{\cT}{\mathcal{T}}

\newcommand\bx{{\bf{x}}}
\newcommand\by{{\bf{y}}}
\newcommand\bz{{\bf{z}}}

\newcommand{\eset}{\emptyset}

\newcommand{\dd}{{\tt d}}

\newcommand{\subeq}{\subseteq}
\newcommand{\supeq}{\supseteq}

\newcommand{\eps}{\varepsilon}

\newcommand{\shalf}{{\textstyle{\frac{1}{2}}}}

\newcommand{\N}{{\mathbb N}}
\newcommand{\Z}{{\mathbb Z}}
\newcommand{\R}{{\mathbb R}}
\newcommand{\C}{{\mathbb C}}

\newcommand{\bP}{{\mathbb P}}

\renewcommand{\H}{{\mathbb H}}
\newcommand{\T}{{\mathbb T}}

\newcommand{\bB}{{\mathbb B}}

\newcommand{\bH}{{\mathbb H}}
\newcommand{\bS}{{\mathbb S}}

\renewcommand{\hat}{\widehat}

\renewcommand{\tilde}{\widetilde}

\newcommand{\GL}{\mathop{{\rm GL}}\nolimits}

\newcommand{\AU}{\mathop{{\rm AU}}\nolimits}

\newcommand{\SO}{\mathop{{\rm SO}}\nolimits}
\newcommand{\SU}{\mathop{{\rm SU}}\nolimits}
\newcommand{\OO}{\mathop{\rm O{}}\nolimits}

\newcommand{\U}{\mathop{\rm U{}}\nolimits}

\newcommand{\Exp}{\mathop{{\rm Exp}}\nolimits}
\newcommand{\Fix}{\mathop{{\rm Fix}}\nolimits}

\newcommand{\ad}{\mathop{{\rm ad}}\nolimits}

\renewcommand{\Re}{\mathop{{\rm Re}}\nolimits}
\renewcommand{\Im}{\mathop{{\rm Im}}\nolimits}

\newcommand{\Ext}{\mathop{{\rm Ext}}\nolimits}

\newcommand{\Isom}{\mathop{{\rm Isom}}\nolimits}

\newcommand{\bL}{{\mathbb L}}

\newcommand{\Aut}{\mathop{{\rm Aut}}\nolimits}

\newcommand{\diag}{\mathop{{\rm diag}}\nolimits}

\newcommand{\id}{\mathop{{\rm id}}\nolimits}

\renewcommand{\dim}{\mathop{{\rm dim}}\nolimits}

\newcommand{\supp}{\mathop{{\rm supp}}\nolimits}

\newcommand{\ev}{\mathop{{\rm ev}}\nolimits}

\newcommand{\Sesq}{\mathop{{\rm Sesq}}\nolimits}

\newcommand{\dS}{\mathop{{\rm dS}}\nolimits}

\renewcommand{\phi}{\varphi}

\newcommand{\nin}{\noindent} 
\newcommand{\oline}{\overline}

\newcommand{\la}{\langle}
\newcommand{\ra}{\rangle}

\newcommand{\res}{\vert}

\newcommand{\Spec}{{\rm Spec}}

\newcommand{\ssssarr}{\hbox to 15pt{\rightarrowfill}}
\newcommand{\sssarr}{\hbox to 20pt{\rightarrowfill}}
\newcommand{\ssarr}{\hbox to 30pt{\rightarrowfill}}
\newcommand{\sarr}{\hbox to 40pt{\rightarrowfill}}
\newcommand{\arr}{\hbox to 60pt{\rightarrowfill}}
\newcommand{\larr}{\hbox to 60pt{\leftarrowfill}}
\newcommand{\Arr}{\hbox to 80pt{\rightarrowfill}}

\def\theoremname{Theorem}
\def\propositionname{Proposition}
\def\corollaryname{Corollary}
\def\lemmaname{Lemma}
\def\remarkname{Remark}
\def\conjecturename{Conjecture} 

\def\definitionname{Definition}
\def\exercisename{Exercise}
\def\examplename{Example}
\def\examplesname{Examples}
\def\problemname{Problem}
\def\problemsname{Problems}

\def\@thmcounter#1{\noexpand\arabic{#1}}
\def\@thmcountersep{}
\def\@begintheorem#1#2{\it \trivlist \item[\hskip 
\labelsep{\bf #1\ #2.\quad}]}
\def\@opargbegintheorem#1#2#3{\it \trivlist
      \item[\hskip \labelsep{\bf #1\ #2.\quad{\rm #3}}]}
\makeatother
\newtheorem{theor}{\theoremname}[section]
\newtheorem{propo}[theor]{\propositionname}
\newtheorem{coro}[theor]{\corollaryname}
\newtheorem{lemm}[theor]{\lemmaname}

\newenvironment{thm}{\begin{theor}\it}{\end{theor}}
\newenvironment{theorem}{\begin{theor}\it}{\end{theor}}
\newenvironment{Theorem}{\begin{theor}\it}{\end{theor}}
\newenvironment{prop}{\begin{propo}\it}{\end{propo}}

\newenvironment{cor}{\begin{coro}\it}{\end{coro}}
\newenvironment{corollary}{\begin{coro}\it}{\end{coro}}

\newenvironment{lem}{\begin{lemm}\it}{\end{lemm}}
\newenvironment{lemma}{\begin{lemm}\it}{\end{lemm}}

\newtheorem{rema}[theor]{\remarkname}

\newenvironment{remark}{\begin{rema}\rm}{\end{rema}}
\newenvironment{rem}{\begin{rema}\rm}{\end{rema}}

\newtheorem{stepnow}[theor]{}

\newtheorem{defin}[theor]{\definitionname} %% write

\newenvironment{definition}{\begin{defin}\rm}{\end{defin}}
\newenvironment{defn}{\begin{defin}\rm}{\end{defin}}

\newtheorem{exerc}{\exercisename}[section]

\newtheorem{exa}[theor]{\examplename}

\newenvironment{ex}{\begin{exa}\rm}{\end{exa}}

\newtheorem{exas}[theor]{\examplesname}

\newtheorem{conj}[theor]{\conjecturename}

\newtheorem{pro}[theor]{\problemname}

\newtheorem{prs}[theor]{\problemsname}

\renewcommand{\mlabel}{\label}

\newcommand{\ip}[2]{\la #1,#2 \ra}

\newcommand{\lf}[2]{[ #1,#2 ]_V}

\newcommand{\hE}{\widehat{\mathcal{E}}}

\newcommand{\he}{\hat{\eta}}

\newcommand{\rO}{\mathrm{O}}

\newcommand{\rH}{\mathrm{H}}

\newcommand{\rU}{\mathrm{U}}

\newcommand{\rI}{\mathrm{I}}
\newcommand{\rS}{\mathrm{S}}
\newcommand{\Ctm}{C^\sigma_m}

\newcommand{\ctm}{\Psi_m}

\newcommand{\Oon}{\OO_{1,n}(\R)^\uparrow}
\newcommand{\Oo}{\OO_{n+1}(\R)}
\newcommand{\Tu}{T_{V_+}}
\newcommand{\wH}{\H^n_V}

\newcommand{\wC}{\widetilde{C}_m}
\newcommand{\hgf}{{}_2F_1}

\newcommand{\Lnp}{\mathbb{L}^{n}_+}
\numberwithin{equation}{section}
\renewcommand{\phi}{\varphi}
\newcommand{\wphi}{\widetilde{\varphi}}

\newcommand{\wpsi}{\widetilde{\psi}_m}

\newcommand{\sx}{\sigma_{V}}
\newcommand{\sr}{\sigma_E} % new from KH 
\newcommand{\cY}{\mathcal{Y}}

\begin{document}

%%%%%%%%%%%%%%%%%%%%%%%%%%

\title{Reflection positivity on spheres}   
\author{Karl-Hermann Neeb}
\address{Department of Mathematics, Friedrich-Alexander-University of 
Erlangen-Nuremberg, Cauerstrasse 11, 91058 Erlangen, Germany}
\email{neeb@math.fau.de}

\author{Gestur \'{O}lafsson}
\address{Department of Mathematics, Louisiana State University, Baton Rouge, LA 70803, U.S.A.}
\email{olafsson@math.lsu.edu}

\thanks{The research of K.-H. Neeb was partially
supported by DFG-grant NE 413/9-1. The research of G. \'Olafsson was partially supported by Simons grant 586106.}
 
\begin{abstract} In this article we specialize a construction of a reflection positive Hilbert space due to
Dimock and Jaffe--Ritter to the sphere $\bS^n$.  We determine the resulting Osterwalder--Schrader 
Hilbert space, a construction that can be viewed as the step from 
euclidean to relativistic quantum field theory. 
We show that this process  gives rise to an irreducible unitary spherical representation of the orthochronous Lorentz group 
$G^c = \OO_{1,n}(\R)^{\uparrow}$ and that the representations thus obtained are 
the irreducible unitary spherical representations  of this group. 
A key tool is a  certain complex domain $\Xi$, known as the crown 
of the hyperboloid, containing a half-sphere $\bS^n_+$ 
and the hyperboloid $\bH^n$ as totally real submanifolds. 
This domain provides a bridge between those two manifolds when we 
study unitary representations of $G^c$ in spaces of 
holomorphic functions on~$\Xi$. 
We connect this analysis with the boundary components which 
are the de Sitter space and a bundle over the space of future 
pointing lightlike vectors. 
\end{abstract}

\maketitle

\nin Keywords: 

\nin MSC 

%\vspace{1cm}

\section{Introduction}
\mlabel{sec:1} 
\noindent
In this article we continue our work on reflection positivity and its connection to representation theory and
abstract harmonic analysis, concerning  the passage from the compact group 
$\OO_{n+1}(\R)$ to its $c$-dual group $\OO_{1,n}(\R)^\uparrow$.

To make this more precise, recall that a {\it symmetric Lie group} 
is a pair $(G,\tau)$, consisting of a Lie group $G$ 
with an involutive automorphism $\tau$. The 
Lie algebra $\g$ of $G$ decomposes into $\tau$-eigenspaces 
$\g = \fh \oplus \fq$, where $\fh = \ker(\tau - \1)$ and 
$\fq = \ker(\tau + \1)$. A Lie group $G^c$ with 
Lie algebra $\g^c = \fh \oplus i\fq$ is called {\it the Cartan dual}, or for short {\it $c$-dual}, {\it to $G$}. 
Reflection positivity now provides a passage from certain 
unitary representations of $G$ to unitary representations of $G^c$. 
One considers representations $(U,\cE)$ of $G$ on 
reflection positive Hilbert spaces 
$(\cE,\cE_+,\theta)$, i.e., $\cE_+ \subeq \cE$ is a closed subspace 
and $\theta$ is a unitary involution for which 
$\la \xi,\xi \ra_\theta := \la \xi,\theta \xi \ra \geq 0$ for $\xi \in \cE_+$. 
We further assume that $\theta U(g) \theta = U(\tau(g))$ for 
$g \in G$. Then the Hilbert space $\hat\cE$ defined by 
$\la \cdot,\cdot \ra_\theta$ on $\cE_+$ 
is expected to carry a unitary representation $(U^c,\hat\cE)$ 
of the $c$-dual group $G^c$ (at least if it is $1$-connected). 
Then we call $(U,\cE)$ a {\it euclidean realization} of $(U^c,\hat\cE)$. 
We refer to \cite[\S\S 1,3]{NO18} for background and details.  

There is a natural source of reflection positive Hilbert spaces 
in Riemannian geometry. We start with a complete Riemannian manifold $M$ 
and an involutive isometry $\sigma \: M \to M$ which is 
dissecting in the sense that the submanifold $M^\sigma$ 
of fixed points is of codimension 
one and its complement consists of two connected components 
$M_\pm$ satisfying $\sigma(M_+) =~M_-$. Typical examples relevant in our 
context are: 
\begin{itemize}
\item[\rm(a)] $\sigma(x) = (-x_0, x_1, \ldots, x_n)$ on 
$\R^{n+1}$, where $M_+ = \{ x \in \R^{n+1} \: x_0 > 0\}$ is an open half space. 
\item[\rm(b)] $\sigma(x) = (-x_0, x_1, \ldots, x_n)$ on the sphere 
$\bS^n = \{ x \in \R^{n+1} \: \|x\| = 1\}\subeq 
\R^{n+1}$, where $\bS^n_+ = \{ x \in \bS^n \: x_0 > 0\}$ is an open half sphere.  
\item[\rm(c)] $\sigma(x) = (x_0, x_1, \ldots, x_{n-1}, -x_n)$ on $\bH^n  := \{ (x_0,\bx) \in \R^{n+1} \: x_0 > 0, x_0^2 - \bx^2 = 1\}$ 
(hyperbolic space), where $\bH^n_+ = \{ x \in \bH^n \: x_n > 0\}$. 
\end{itemize}

Let $\Delta$ be the Laplacian of $M$, considered as a 
negative selfadjoint operator on $L^2(M,\mu)$, where $\mu$ is the volume measure. 
For any $m>0$, we obtain a bounded positive symmetric operator 
$C_{m}=(m^2-\Delta)^{-1}$ on $L^2(M,\mu)$. 
It defines an inner product on $L^2(M,\mu)$  by
\[\ip{\varphi}{\psi}_{-1}=\ip{\varphi}{C_m \psi}_{L^2} 
=\int_M \overline{\varphi(m)} (C_{m}\psi) (m)\, d\mu (m).\]
The corresponding completion is the Sobolev space $\cH^{-1}(M)$ 
and $\sigma$ induces a unitary involution $\sigma_*(f) := f \circ \sigma$ 
on this space. It is shown in \cite{AFG86,JR08,An13,Di04} 
that the triple $(\cE,\cE_+,\theta) := (\cH^{-1}(M),\cH^{-1}(M_+),\sigma_*)$ 
is a reflection positive Hilbert space. 
In \cite{Di04}, the space $\widehat\cE$ is even identified 
with the subspace $\cH^{-1}_{M^\sigma} \subeq \cH^{-1}(M_+)$ 
consisting of all elements whose support, as distributions on $M$,  
is contained in~$M^\sigma$.

In this paper we study the case 
$M = \bS^n$, its isometry group $G := \OO_n(\R)$, 
and the representations of the $c$-dual group 
$G^c := \OO_{1,n}(\R)^\uparrow$ (the orthochronous Lorentz group) 
on the Hilbert spaces $\hat\cE$ corresponding 
to all values $m  > 0$. 
In particular, we shall see that 
the above construction provides a euclidean realization 
of all irreducible spherical unitary representations of~$G^c$. In addition, it leads to very natural realizations 
in spaces of holomorphic functions on a complex manifold $\Xi$ 
containing $\bS^n_+$ and $\bH^n$ as totally real submanifolds. 
Parts of our results have already been announced in~\cite{NO18}.\\ 

To obtain these results, we proceed as follows.  In Subsection~\ref{se:ReLap} we describe how the reflection 
positivity of $(\cH^{-1}(M),\cH^{-1}(M_+),\sigma_*)$ 
leads to a positive definite analytic kernel function 
$\Psi_m$ on the open subset $M_+  \subeq M$. This is best understood 
by first interpreting $\cH^{-1}(M)$ as a Hilbert space of distributions 
on $M$, defined by a positive definite distribution kernel 
$\Phi_m$ on $M \times M$. 
This kernel is analytic on the complement of the diagonal 
$\Delta_M \subeq M \times M$ because it satisfies on this domain 
the elliptic differential equation 
$\Delta \Phi_m =  m^2\Phi_m$ in both variables. As a consequence, 
the kernel $\Psi_m(x,y) := \Phi_m(x,\sigma(y))$ is analytic 
on $M_+ \times M_+$ and positive definite by reflection positivity 
(Corollary~\ref{cor:2.9}, Lemma~\ref{le:ctm}). 
We also make an effort to translate between the two different 
approaches to the reflection positivity result by 
Jaffe and Ritter \cite{JR08,An13} and Dimock \cite{Di04}. 
In Subsection~\ref{se:SymSp}, all this is specialized to 
the situation where $M \cong G/K$ is a Riemannian symmetric space. 
In this context the kernel $\Phi_m$ is represented by a 
$K$-invariant analytic eigenfunction $\phi_m$ of $\Delta$ on 
the complement of the base point (Theorem~\ref{the:ctmSp}). 
Eventually, we recall in Subsection~\ref{se:spSym} 
the symmetric space structures on the sphere 
$\bS^n \cong \OO_{n+1}(\R)/\OO_n(\R) = G/K$ and the hyperboloid 
$\bH^n \cong \OO_{1,n}(\R)^\uparrow/\OO_n(\R) = G^c/K$. 

To establish the connection between sphere and hyperbolic space 
in Section~\ref{sec:3}, we first note that the complex bilinear form on $\C^{n+1}$ 
restricts on the subspace $V:=\R e_0+i\R^n$ 
to the Lorentzian form $[(u_0,i\bu), (v_0, i\bv)]_V := u_0 v_0 - \bu\bv$. 
Accordingly, we obtain a natural realization of 
$G^c = \OO_{1,n}(\R) ^\uparrow$ in $\GL(V) \cap \OO_{n+1}(\C)$. 
Now the complex 
submanifold $\Xi :=G^c .\bS^n_+$ of the complex sphere $\bS^n_\C$ 
contains the half sphere $\bS^n_+$ and the hyperbolic space 
\[\wH =\{u\in V\: u_0 > 0, \lf{u}{u}=1\}\cong \OO_{1,n}(\R)^\uparrow/\OO_n(\R )\]
as totally real submanifolds, so that we may 
translate between $\bS^n$ and $\bH^n$ by analytic continuation. 
In the literature on representations of semisimple Lie groups, 
the manifold $\Xi$ is called the {\it crown of hyperbolic space} 
and it plays the role of a natural ``complexification'' of 
the Riemannian symmetric space $G^c/K$ (cf.\ \cite{AG90,KS04,KS04,KO08} 
and Theorem~\ref{thm:3.9}).
The boundary of $\Xi$ consists of two $G^c$-orbits. One is 
{\it de Sitter space} 
\[ \dS^n = \Xi \cap i V  = i \{ v \in V \: [v,v]_V =-1\}, \] 
and the other orbit projects onto the homogeneous space 
\[ \bL^n_+ :=\{v = (v_0, i\bv) \in V\: \lf{v}{v} =0, v_0>0\}\] 
of positive light rays in $V$. We also note that, 
for the future cone 
\[ V_+=\{u= (u_0,i\bu) \in V \: u_0 > 0, [u,u] > 0\},\] 
the corresponding tube $T_{V_+} := V_+ + i V$ intersects 
$\bS^n_\C$ precisely in $\Xi$. 

In Section~\ref{sec:4} we obtain the analytic continuation 
of the kernels $\Psi_m$ on $\bS^n_+$ to $G^c$-invariant kernels on~$\Xi$. 
We call a kernel $\Psi$ on $\Xi \times \Xi$ sesquiholomorphic 
if it is holomorphic in the first  and antiholomorphic in the second 
argument.  
Our first main result is Theorem~\ref{thm:KernSphere},  asserting 
that $G^c$-invariant sesquiholomorphic kernels on $\Xi$ are of the form 
\[ \Psi(z,w) = \alpha_\Psi([z,\sigma_V w]_V), \qquad 
\alpha_\Psi \: \C \setminus (-\infty,1] \to \C\ \mbox{holomorphic},\] 
where $\sx$ is the complex conjugation on $V_\C$ fixing $V$ pointwise. 
To obtain an analytic continuation of $\Psi_m$, we thus have 
to determine the corresponding function $\alpha_{\Psi_m}$, 
which occupies the remainder of Section~\ref{sec:4}. 
The main results are Theorems~\ref{th:Psi} and~\ref{thm:gamma}, 
expressing $\Psi_m$ by the hypergeometric function $\hgf$: 
\[\Psi_m(z,w)=\frac{\Gamma \left(\frac{n-1}{2}+\lambda\right)\Gamma \left(\frac{n-1}{2}-\lambda\right)}{\Gamma (n)}
\hgf \Big(\frac{n-1}{2}+\lambda,\frac{n-1}{2}-\lambda;\frac{n}{2};\frac{1}{2}\left(1-\lf{z}{\sx w}\right)\Big),\]
where 
 \[\lambda = \lambda_m := 
\begin{cases}
\sqrt{\left(\frac{n-1}{2}\right)^2-m^2} & \text{ for}\  0\leq m\le (n-1)/2  \\
i \sqrt{m^2- \left(\frac{n-1}{2}\right)^2} 
& \text{ for }  m\ge (n-1)/2.
\end{cases}\]  

In Section~\ref{se:RefPosRep} we eventually turn to the 
representation theoretic consequences of these results. 
Theorem~\ref{thm:irrSpRep} provides the key information by 
showing that the irreducible positive definite spherical functions 
$(\phi_m)_{m > 0}$ on $\bH^n_V$ are positive multiples of the functions 
$\Psi_m(\cdot, e_0)$. In particular, they  extend to 
sesquiholomorphic kernels 
\[ \Phi_m^c(z,w) := \frac{\Psi_m(z,w)}{\Psi_m(e_0,e_0)}, \quad m > 0, \qquad 
\Phi_0^c(z,w) := 1, \] 
on $\Xi \times \Xi$. This entails  that the 
representations $(U^c_m, \hat\cE)$ of $G^c$ that we obtain 
from the representations of $G = \OO_{n+1}(\R)$ 
on the reflection positive Hilbert spaces 
$(\cH^{-1}(M),\cH^{-1}(M_+),\sigma_*)$ are precisely 
the irreducible spherical representations. 
It also provides a natural realization of these representations 
in reproducing kernel Hilbert spaces $\cH_{\Phi^c_m} \subeq \cO(\Xi)$ 
of holomorphic functions on $\Xi$ by 
$\pi_m(g)f = g_* f$ (Corollary~\ref{cor:IntRep}). 
As the kernels $(\Phi^c_m)_{m \geq 0}$ 
are the extreme points in the convex set 
of sesquiholomorphic positive definite $G^c$-invariant
kernels $\Psi$ on $\Xi\times \Xi$  normalized by $\Psi(e_0,e_0)=1$ 
(Corollary \ref{cor:IntRep}), for all such kernels there exists a 
probability measure $\mu$ on $[0,\infty)$ with 
\begin{equation}
  \label{eq:intrepx}
 \Psi = \int_0^\infty \Phi^c_m\, d\mu(m).
\end{equation}
We apply this in Section \ref{se:perspectives} to two natural classes of examples.

As spherical representations of $G^c$ are typically realized in 
functions on the sphere $\bS^{n-1} \cong \bL_n^+/\R^\times_+$, 
we show in Theorem~\ref{thm:Pm} how this leads to an integral representation  
of the kernels $\Phi^c_m$ in terms of ``plane wave kernels'': 
\begin{equation}
  \label{eq:intrep1}
\Phi_m^c(z,w)=\int_{\bS^{n-1}} \lf{\sx (w)}{(1,u)}^{\lambda-\frac{n-1}{2}}\lf{z}{(1,u)}^{-\lambda - \frac{n-1}{2}}d\mu_{\bS^{n-1}}(u), 
\end{equation}
where $\mu_{\bS^{n-1}}$ is the $\OO_n(\R)$-invariant probability measure on 
$\bS^{n-1}$. 
This in turn  leads to a Poisson transform from the realization 
on $\bS^{n-1}$ to holomorphic functions on $\Xi$. 
Section~\ref{se:RefPosRep} is rounded off by 
Subsection~\ref{subsec:5.3} with a brief discussion 
of the relations between our kernels with 
canonical kernels on hyperbolic spaces (cf.~\cite{vDH97}). 
These kernels also extend analytically to a neighborhood of 
$\bH^n_V$ in $\Xi$, but not to all of $\Xi$. 

We conclude this paper with Section~\ref{se:perspectives}, 
where we discuss various aspects of our results 
that have not been pursued in this paper. 
In Section \ref{sec:lieball} we show that $\Xi$ is 
holomorphically equivalent to the Lie ball, the bounded 
symmetric domain $\SO_{2,n}(\R)/\rS(\OO_2(\R)\times \OO_n(\R))$ 
associated to the Lie group $\SO_{2,n}(\R)$. 
In particular, the action of $G^c$ on $\Xi$ extends to 
a transtive action of $\SO_{2,n}(\R)$. 
This observation can already be found 
in \cite[Table III, p.229]{KS05}. 
It was used in \cite{GKO03} to construct $\SO_{1,n}(\R)_0$-invariant 
distributions on de Sitter space 
$\dS^n$, realized as a $G^c$-orbit in $\partial \Xi$. 
From this perspective, the $\SO_{2,n}(\R)$-invariant positive 
definite kernels on $\Xi$ are of particular interest, 
and results on their branching behavior on $G^c$ 
are briefly described in Thorem~\ref{thm:ExMuNu} in terms of the 
integral decomposition in the sense of 
\eqref{eq:intrepx}. 

In  Subsection \ref{se:boundaryVal} we show that the 
$G^c$-representations in the Hilbert subspaces 
$\cH_{\Phi^c_m} \subeq \cO(\Xi)$ have natural boundary value maps 
into $G^c$-invariant Hilbert spaces of distributions on $\dS^n 
\subeq \partial \Xi$. 
Subsection~\ref{subsec:6.3} briefly discusses analogs 
of our results concerning spheres $\bS^n$ for $\R^{n}$ and $\bH^n$, 
endowed with their natural dissecting involutions.  
We end in Subsection~\ref{subsec:6.4} by showing that the 
spherical representations $(\pi_m, \cH_m)_{m \geq 0}$ 
of $G^c$ all extend naturally to antiunitary representations 
of the full Lorentz group~$\OO_{1,n}(\R)$ in the sense of \cite{NO17}. \\

\nin {\bf Background:} 
In \cite{NO19}, we classify all irreducible symmetric spaces with
dissecting involution. It turns out that they are quadrics 
\[ Q := \{ x \in \R^{p+q} \: \beta_{p,q}(x,x) = 1\},
\quad \mbox{ where } \quad 
\beta_{p,q}(x,y) = \sum_{j = 1}^p x_j y_j - \sum_{j = p+1}^{p+q}  x_j y_j\] 
and the dissecting involution is given either by 
\[ \sigma (x_1,\ldots ,x_n)=(-x_1,x_2,\ldots ,x_n) \quad \mbox{ or by } \quad 
\sigma (x_1,\ldots ,x_n)=(x_1,\ldots , x_{n-1},-x_n).\] 
In particular, the only irreducible $n$-dimensional semisimple Riemannian 
symmetric spaces with dissecting involutions 
are the sphere $\bS^n$ ($q = 0$) and the hyperboloid $\bH^n$ 
($p = 1$).\\

\nin {\bf Connections with physics:} 
The origins of reflection positivity lie 
in the construction of euclidean quantum field theories 
by Osterwalder and Schrader \cite{OS73,OS75} (see \cite{Ja08} for a historical discussion). 
The classical example is 
$M = \R^{n+1}$ on which $\sigma(x_0, \bx) = (-x_0,\bx)$ 
is interpreted as a time reflection 
and the passage from $\cE $ to $\widehat \cE$ corresponds to the passage 
from euclidean quantum field theories to relativistic ones. 
In this process the euclidean inner product on $\R^{n+1}$ changes to
the Lorentzian form $[x,y]=x_0y_0 -\bx \by$.

Since then various constructions of reflection positive spaces 
and their representaton theoretic context have been 
studied. We refer to \cite{Ja08}, \cite{JO98, JO00} for surveys 
and to the recent monograph \cite{NO18} for more details.

Our work is closely related to the series of articles by J.~Bros 
and his coauthors \cite{BM96,BV97,BEM02b,BM04}, although our 
perspective is different. In these papers the focus is mostly on de Sitter space 
$\dS^n$ and on analytic extensions from there to $\Xi$, whereas 
we focus on the passage from the sphere 
$\bS^n$ to $\bH^n_V \subeq \Xi$. 

We now comment briefly on the intersection points. 
In \cite{BM96} $X_n = \dS^n$ is $n$-dimensional de Sitter space 
and $\cT_\pm = G^c.\bS^n_\pm$ are called Lorentz tuboids; clearly 
$\cT_+ = \Xi$. Instead of sesquiholomorphic 
kernels on $\Xi$, the authors study holomorphic kernels on $\cT_+ \times \cT_-$. 
{\it Perikernels} are distributional solutions of the Klein--Gordon equation 
$(m^2 - \square)\Psi = 0$ on de Sitter space which extend holomorphically 
to $\cT_\pm$, hence correspond to our kernels~$\Psi_m$ 
\cite[\S 4.2]{BV96}. The K\"allen--Lehmann representation 
for generalized free fields \cite[Prop.~3.3]{BV96} is a variant 
of our integral representation 
\eqref{eq:intrepx} which applies to all $G^c$-invariant positive definite 
sesquiholomorphic kernels on~$\Xi$. It is formulated in terms 
of distributional boundary values of the $\Psi_m$ 
and requires certain growth conditions on the kernels. 
The integral representation in terms of {\it de Sitter plane waves} 
in \cite[\S\S 4.1]{BV96} is closely related to our 
formula \eqref{eq:intrep1}. In particular 
\cite[Thms.~4.3, 4.4]{BV96} relate to 
irreducible spherical representation of the de Sitter group $\SO_{1,n}(\R)$ 
and identify the corresponding Casimir eigenvalues, 
where  $\rho$ is called the {\it geometric mass}. 
These results are used in \cite{BEM02b} in the context of 
quantum field theory on de Sitter space $\dS^n$. 
Propositions~1,2,3 in \cite[\S 2.1]{BM04} appear also in our \S 3.1. 
In particular, Proposition 3 determines the set $\C_\Xi$.
Techniques based on \cite{Di04} and reflection positivity 
concerning the passage from $\bS^2$ to $\dS^2$ have recently been used 
to construct interacting quantum fields on $\dS^2$ in \cite{BJM16}. 

\tableofcontents

\subsection*{Notation}\label{se:Notat}
In this section we collect notation that will be used throughout the article.
\medskip

\nin{\bf Euclidean and Minkowski space:} 
The standard basis for the euclidean space $E := \R^{n+1}$ and $\C^{n+1}$ is denoted by $e_0,e_1,\ldots ,e_n$. Accordingly, we use the notation
\[ z = (z_0,z_1,\ldots ,z_n)= (z_0,\bz) \quad \mbox{ with } \quad 
z_j\in\C, \bz=(z_1,\ldots ,z_n)\in \C^n.\]  
Denote by $z w= \sum_{j=0}^{n} z_jw_j$ the standard $\C$-bilinear form
on $\C^{n+1}$ and write $z^2=zz$.  
For $j =0,\ldots, n$, we 
write $r_j(z) = z - 2 z_j e_j$ for the orthogonal reflections in $e_j^\bot$. 
Half-sphere and complex sphere are denoted by 
\[  \bS^n_\pm :=\{x\in \bS^n\: \pm x_0>0\}\subset \bS^n:=\{x\in \R^{n+1}\: x^2=1\}\subset \bS^n_\C:=\{z\in \C^{n+1}\: 
z^2 =1\}.\]
The Lorentzian bilinear form on $\R^{n+1}$ is denoted 
$[z,w]=z_0w_0 -\bz\bw$. We write 
$\R^{1,n} = (\R^{n+1},[\cdot,\cdot])$ 
for the $(n+1)$-dimensional Minkowski space 
and 
\[ \H^n:=\{x\in \R^{1,n}\: [x, x]=1, x_0 > 0\} \] 
for the hyperboloid model of $n$-dimensional hyperbolic space.

As we shall see below, it is convenient to realize Minkowski space 
as the subspace 
\[ V:=\iota \R^{n+1}=\R e_0\oplus i\R^n \subeq \C^n \quad \mbox{ for } \quad 
\iota(x_0,\bx) := (x_0,i\bx).\]   
Then 
\[ [x,y] = x_0y_0 -\bx  \by = (\iota x)(  \iota y) =: 
[\iota(x),\iota(y)]_V  \quad \mbox{ for }\quad x,y \in \R^{n+1}. \] 
On $\C^{n+1}$ we consider the conjugations 
\[ \sr(z_0, \ldots, z_n) := (\oline{z_0}, \ldots, \oline{z_n}) 
\quad \mbox{ and } \quad 
 \sx(z_0, \ldots, z_n) := (\oline{z_0}, -\oline{z_1}, \ldots, -\oline{z_n}) \] 
with respect to the real subspaces $E = \R^{n+1}$ and $V$, respectively. 
These  conjugations commute and the holomorphic involution $\sr\sx$ is $-r_0$. \\

\nin{\bf Groups:} For the matrix 
Lie groups that will be used in this article we use the notation
\begin{itemize}
\item $\OO_{n+1}(\C)=\{g\in \GL_{n+1}(\C )\: (\forall z,w\in\C^{n+1})\, (gz)(gw)=z w\}$.
\item
$\rO_{n+1} (\R)=\GL_{n+1}(\R)\cap \rO_{n+1}(\C)$.
\item $\OO_{1,n}(\R ) =\{g\in \GL_{n+1}(\R )\:  (\forall x,y\in\R^{n+1})\, [gx,gy]=[x,y]\}$.  
\item $\Oon :=\{g\in\OO_{1,n}(\R)\: g_{00} > 0\}.$ 
\end{itemize}
The natural action of $G:=\Oo$ on $\R^{n+1}$ 
defines a transitive action of $G$ on $\bS^n$. The stabilizer of 
$e_0$ is $K := G_{e_0} \cong \OO_n(\R)$, 
so that we obtain a $G$-equivariant diffeomorphism 
$G/K \to \bS^n, gK \mapsto g.e_0$, by the orbit map.  
Similarly, the group $\Oon$ acts transitively on $\H^n = \Oon.e_0\simeq  \Oon / K$.

The group 
\[  G^c:=\iota \OO_{1,n}(\R)^\uparrow \iota \subeq \GL(V)\]
preserves the Lorentzian product $[\cdot,\cdot]_V$. 
It also contains $K = G_{e_0} = G^c_{e_0}$.  
The involution 
$\sx$ commutes with $G^c$,  and 
$\sr g \sr = r_0 g r_0 = \tau(g)$ 
is the involution on $G^c$ whose fixed point group is 
$K = G^c_{e_0}$. With 
\[  G_\C  :=\OO_{n+1}(\C) \quad \mbox{ and }\quad K_\C = G_{\C,e_0} \cong 
\OO_n(\C), \] 
we then have $\bS^n_\C = G_\C.e_0 \cong G_\C/K_\C.$ 
 
We shall also consider the following sets: 
\begin{itemize}
\item $V_+:=\{v\in V\: \lf{v}{v}>0, v_0>0\}$ 
(open upper light cone),
\item  $\wH:= \iota \H^n=\bS^n_\C \cap V_+ = G^c.e_0 \cong G^c/K$ 
(hyperbolic space), 
\item $\bS^n_{+,\C} := \{ z \in \bS^n_\C \: \Re z_0 > 0\}$,  
\item $\Tu:= iV+V_+$ (future tube), and 
\item $\Xi:=G^c.\bS^n_+\subset \bS^n_{+,\C}$, called 
the \textit{crown} of~$\wH$. 
\end{itemize}
Note that $V$, $V_+$, $T_{V_+}$ and $\bS^n_{\C}$ are invariant under 
the action of $G^c$.  \\
 
\nin {\bf Distributions:} 
If $M$ is a manifold, then $C^\infty (M)$ denotes 
the space of smooth complex valued functions on $M$ and
$C^\infty_c(M)$ the subspace of compactly supported smooth functions. 
This space carries the locally convex topology for which it 
is the direct limit of the closed subspaces 
$C^\infty_C(M)$ of smooth functions supported in the compact subset $C\subeq M$, 
on which the topology is that of 
uniform convergence of all derivatives on compact subsets
 of chart neighborhoods. This turns $C^\infty_c(M)$ into a complete 
locally convex space. Its conjugate linear dual, i.e., the space of 
continuous antilinear functionals,  is denoted by $C^{-\infty} (M)$.
This is the space of \textit{distributions} on $M$. 
If $M$ is a Riemannian manifold, then the volume measure $\mu = \mu_M$ defines a locally finite measure 
on $M$. This  leads to 
a linear injection  $L^1_{\text{loc}}(M)\hookrightarrow C^{-\infty}(M)$ given by
\[ \Psi_f(\eta )=\int_M \overline{\eta (x)} f(x)\, d\mu (x), \quad 
f\in L^1_{\text{loc}}(M), \eta\in C_c^\infty (M).\]

If $\theta : M\to M$ is a diffeomorphism then $\theta_* :C^\infty_c (M)\to C^\infty_c(M)$ is defined by
$\theta_* \eta =\eta \circ \theta^{-1}$. One defines $\theta_* $ in a 
similar way on other function spaces and
spaces of distributions. If $\theta$ is an isometry, 
then $\mu$ is $\theta_*$ invariant, i.e., $\int_M \theta_* \eta\, d\mu
=\int_M \eta\, d\mu$ for all $\eta\in C_c^\infty (M)$.

\section{Reflection positivity related to the resolvent of the Laplacian} 
\mlabel{sec:2} 
\noindent
In this section we recall the definition of a reflection 
positive Hilbert space. 
We then introduce a construction of Hilbert 
spaces based on resolvents 
of the Laplace operator   \cite{AFG86,An13,Di04,JR08,JR07a,JR07b}.  
We then specialize the discussion to symmetric spaces and the sphere. We follow mostly the presentation
of \cite{Di04} which provides a Markov realization of $\hat\cE$ 
 and refer to \cite{NO18} for more details 
and background on reflection positivity. 

\subsection{Reflection positivity and the resolvent of the Laplacian}\label{se:ReLap}

 Most 
of the material in this subsection can be found in our
previous articles \cite{NO14,JO98,JO00} and the  monograph~\cite{NO18}. 

\begin{defn}(Reflection positive Hilbert space) 
\mlabel{def:x.1} A reflection positive Hilbert space is a
triple $(\cE,\cE_+,\theta)$, where $\cE$ is a Hilbert space, $\cE_+$ is
a closed subspace of $\cE$ and $\theta \in \rU (\cE)$ is an involution for which the 
hermitian form $\ip{\xi }{\eta}_\theta := \ip{\theta \xi}{\eta}$ is positive semi-definite 
on $\cE_+$. 
\end{defn}

{}Let $(\cE,\cE_+,\theta)$ be a reflection positive Hilbert space. 
We write 
\[\cN 
:= \{ \eta  \in \cE_+ \: \|\eta \|_\theta  = 0\} 
= \{ \eta  \in \cE_+ \: (\forall \zeta \in \cE_+)\ \ip{\zeta }{\eta }_\theta = 0\}
= \cE_+ \cap \theta(\cE_+)^\bot \] 
and $q$ or $\widehat{\,\,}$ for the quotient map $ \cE_+ \to \cE_+/\cN, \eta \mapsto \hat\eta=q(\eta)$. 
Denote by $\hE$  the Hilbert space completion of $\cE_+/\cN$ with respect to 
the norm $\|\he\|_{\theta}:= \sqrt{\la \theta \eta , \eta \ra}$. 

The passage from $\cE$ to $\hat\cE$ can be used to 
construct from certain unitary representations 
of a Lie group $G$ on $\cE$ 
a unitary representation of another Lie group $G^c$ 
on $\hE$. We refer to 
\cite{MNO14,NO14,NO15a,NO18,JO98,JO00} for details.

The geometric setup for the main theme of this article is the following.  
Let $M$ be a complete Riemannian manifold and $\Delta$ denote the negative 
Laplace--Beltrami operator on $L^2(M, \mu_M)$, considered as a selfadjoint 
operator. Here the completeness of $M$ is used for the 
essential selfadjointness (\cite[Thm.~2.4]{Str83}) 
and $\mu = \mu_M$ is the Riemannian volume 
measure. Then $\Spec(\Delta)$ is contained in $(-\infty ,0]$. If $M$ is compact then  $\mu (M)<\infty$ and the 
constants are in contained in $L^2(M)$, so that $0\in \Spec(\Delta)$. 

For each $m > 0$, we have a bounded positive operator 
$C_m=C:=(m^2-\Delta)^{-1}$ on $L^2(M)$.

\begin{defn}\label{def:PsoDef} Let $M_+\subset M$ be an open 
submanifold, let $\sigma : M\to M$ be an isometric 
involution, and $D\in C^{-\infty}(M \times M)$. Then 
\begin{itemize}
\item[(i)] $D$ is called {\it positive definite} if
$D(\varphi \otimes \oline\varphi) \geq 0$ for $\varphi\in C^\infty_c(M)$. 
We then write $\cH_D \subeq C^{-\infty}(M)$ for the Hilbert subspace obtained 
from the positive semidefinite form on $C^\infty_c(M)$ defined  by 
\[ \la \varphi, \eta\ra_D := D(\varphi \otimes \oline\eta) 
\quad \mbox{ for } \quad \varphi, \eta \in C^\infty_c(M).\] 
This space embeds naturally into $C^{-\infty}(M)$ in such a way 
that $\xi \in \cH_D$ is mapped to the  the distribution 
$\xi_D := \la \cdot, \xi \ra_D$. 

\item[(ii)] A $\sigma$-invariant distribution $D$ on $M \times M$ 
is said to be  {\it reflection positive with respect to $(M,M_+,\sigma)$} 
if $D$ is positive definite and the distribution 
$D^\sigma :=D\circ (\id ,\sigma )$ on $M_+ \times M_+$ is positive definite. 
We write $\cH_{D,+} \subeq \cH_D$ for the subspace 
generated by $\xi_D$, $\xi \in C^\infty_c(M_+)$. We denote the extension of $\theta_*$ 
 to $\cH_D$ by the same symbol. Then $D$ is reflection positive if and only if 
$(\cH_D,\cH_{D+},\theta_*)$ is reflection positive.
\item[(iii)] If $G$ is a group, $G_+\subset G$ and $\tau :G\to G$ 
is an involution, then
$\varphi : G\to \C$ is {\it positive definite}, resp., 
{\it reflection positive with respect to $G_+$ and $\tau$} if 
the kernel $D(x,y)=\varphi (y^{-1}x)$ 
is positive definite, resp., reflection positive with
respect to $(G,G_+,\tau)$.
\end{itemize}
\end{defn}

\begin{rem} If the distribution $D$ is represented by a continuous 
function $d(x,y)$ with respect to the measure $\mu$ on $M$, i.e., 
\[ D(\varphi \otimes \oline\psi)  
= \int_{M \times M} \oline{\varphi(x)} d(x,y) \psi(y)\, d\mu(x)\, 
d\mu(y) \quad \mbox{ for } \quad 
\varphi,\psi \in C^\infty_c(M),\] 
then the positive definiteness of $D$ is equivalent to the 
positive definiteness of the kernel function $d$, i.e., to 
\[\sum_{i,j=1}^n c_i\overline{c_j} d(x_i,x_j)\ge 0
\quad \mbox{ for } \quad x_1,\ldots x_n, c_1,\ldots ,c_n\in \C.\]  
\end{rem}

We write 
$\mathrm{diag}(M):=\{(x,x)\in M\times M\: x\in M\}$ for the diagonal 
in $M \times M$.

\begin{lem}\label{le:Phim} The bounded operator 
$C_m$ on $L^2(M)$ defines a distribution $\tilde C_m$ on 
$M \times M$ by 
\[\wC (\varphi\otimes \oline\eta )
:=\ip{{\varphi}}{C_m{\eta}}=\int_{M} \oline{\varphi} (x)
(C_m\eta) (x)\, d\mu_M(x) 
\quad \mbox{ for } \quad \varphi, \eta 
\in C^\infty_c(M).\]
Furthermore, the following assertions hold: 
\begin{itemize}
\item[\rm(i)] $\tilde C_m$ is positive definite, in particular 
$\wC (\varphi\otimes \oline\eta) = \oline{\wC (\eta \otimes \oline\varphi)}$ 
for $\varphi, \eta \in C^\infty_c(M)$. 
\item[\rm(ii)] We have 
$(m^2-\Delta )_y \tilde C_m(x,y)=  \delta_M(x,y)$ in sense of distributions on 
$M \times M$, where 
$\delta_M(\phi) = \int_M \phi(x,x)\, d\mu(x)$. 
\item[\rm(iii)] On $M\times M\setminus \diag(M)$, 
the distributional derivatives of $\tilde C_m$ satisfy 
\[0 = (m^2-\Delta)_y \tilde C_m(x,y)= (m^2-\Delta)_x\tilde C_m(x,y)
=(m^2-\Delta)_x(m^2-\Delta)_y\tilde C_m(x,y).\]
\item[\rm(iv)] On the open subset $M\times M\setminus \diag (M)$, the distribution 
$\tilde C_m$ is represented by an analytic function $\Phi_m$. 
\item[\rm(v)] The distribution $\tilde C_m$ and the function 
$\Phi_m$ are invariant under the isometry group $\Isom(M)$. 
\end{itemize}
\end{lem}

\begin{proof} That $\tilde C_m$ defines a distribution on $M \times M$ 
follows from the continuity of the bounded operator $C_m$ 
(\cite[Thm.~51.6]{Tr67} and a partition of unity argument to reduce to 
open subsets of $\R^n$). 

\nin (i) follows directly from the positivity of the operator $C_m$. 

\nin (ii) follows from 
\begin{align*}
\tilde C_m(\phi \otimes (m^2 - \Delta)\eta) 
&= \la \phi, C_m (m^2 - \Delta)\oline\eta\ra_{L^2(M)} 
= \la \phi, \oline\eta\ra_{L^2(M)} \\
&= \int_M \oline{\phi(x)} \oline{\eta(x)}\, d\mu_M(x) 
= \delta_M(\phi \otimes \eta).
\end{align*}

\nin (iii) and (iv): The first equality follows immediately from (ii), and the 
second one  from~(i). The third is an immediate consequence of the first two. 
As $(m^2-\Delta)_x(m^2-\Delta)_y$ is an elliptic operator on $M\times M\setminus
\diag(M)$ and annihilates $\tilde C_m$ on this open subset, it follows that 
$\tilde C_m$ is represented on the complement of $\diag(M)$ by an 
analytic function (\cite[Thm.~8.12]{Ru73}). 

\nin (v) Since $\Delta$ commutes with the action of $\Isom(M)$ on $L^2(M)$, 
the operator $C_m$ also commutes with $\Isom(M)$. This implies that the 
corresponding distribution  $\tilde C_m$ 
on $M \times M$ is invariant under $\Isom(M)$. 
\end{proof}

\begin{defn} An isometry $\sigma$ of a connected complete Riemannian 
manifold $M$ is called {\it dissecting} if the complement of the 
fixed point set $M^\sigma$ is not connected. Then $\sigma$ is an involution, 
the complement of $M^\sigma$ has two connected components 
$M_\pm$ with $\sigma(M_\pm) = M_\mp$, and each connected component of $M^\sigma$ 
is of codimension~$1$ (see \cite[Lemma~2.7]{AKLM06}). 
\end{defn}

Let $M$ be a connected complete Riemannian manifold and 
let $\sigma \: M \to M$ be a dissecting involution. For $m > 0$ and  
$r\in\R$ let $\cH^r(M)$ be the completion of $C_c^\infty(M)$ in the norm
\[ \ip{\varphi}{(-\Delta +m^2)^r\varphi}^{1/2}.\] It is easy to see that this 
space does not depend on $m$. 
For $\varphi,\psi\in C_c^\infty (M)$  
we have $|\ip{\varphi}{\psi}_{L^2}|\le \|\varphi\|_r\|\psi \|_{-r}$, 
so that the $L^2$-pairing extends to the duality pairing on 
$\cH^{r}(M) \times \cH^{-r}(M)$, 
realizing $\cH^r(M)$ as the dual of $\cH^{-r}(M)$. We
also note that $(-\Delta +m^2): \cH^1 (M)\to \cH^{-1}(M)$ is unitary. 

For a closed subset $A\subset M$, we let
$\cH^r_{A}(M)=\{\varphi\in\cH^r(M)\: \supp \varphi \subseteq A\}$ and for 
an open subset $\Omega\subseteq M$, we let $\cH_{0 }^r(\Omega )$ be the closed 
subspace of $\cH^r (M)$ generated by $C_c^\infty (\Omega)$. 

The following results can be found in \cite[Lem.~1/2, Cor.~1/2]{Di04} 
but it should be pointed out that
the reflection positivity for $(\cH^{-1}(M),\cH^{-1}_{\oline{M_+}}(M),\theta)$ was also established in~\cite{JR08, An13, AFG86}. This follows from 
Lemma~\ref{lem:2.9} below which builds a bridge to Dimock's context.  

\begin{thm}{\rm (\cite{Di04})} \label{thm:E0}  
Consider the following subspaces 
$\cE:=\cH^{-1}(M)\subset C^{-\infty}(M)$: 
\[ \cE_0=\cH^{-1}_{M^\sigma}(M), \qquad 
\cE_+=\cE_0\oplus (-\Delta +m^2)\cH_{0}^1(M_+)\quad \mbox{  and } \quad 
\cE_-=\theta \cE_+=\cE_0\oplus (-\Delta +m^2)\cH_{0}^1(M_-),\] 
where $\oplus$ stands for orthogonal direct sum.
We note that $\cE_0=\cE_+\cap \cE_-=\cE_+^{\theta}$. 
Then the  following assertions hold:
\begin{itemize}
\item[\rm (i)] $\cE_+=\cH^{-1}_{M^\sigma \cup M_+}(M) = \cH^{-1}_{\oline{M_+}}(M)$.
\item[\rm(ii)] $\cE =\big((-\Delta + m^2)\cH^1_0(M_-)\big)\oplus\cE_+$.
\item[\rm (iii)] (The Markov condition) If $\varphi \in \cE_+$, 
then $P_-\varphi = P_0\varphi$, i.e., $P_-P_+=P_0$. 
\item[\rm (iv)] $\cE_0\not=\{0\}$.
\item[\rm(v)] For $\varphi\in\cE_+$ and $\psi\in \cE_-$ we have 
$\ip{\varphi}{\psi}_{-1}=\la \varphi, (m^2 - \Delta)^{-1} \psi \ra_{H^{-1} \times H^1} = 
\ip{P_0\varphi}{P_0 \psi}_{-1}$.
\end{itemize}
\end{thm}

{}From this we immediately get 
(cf.~\cite[\S 2.3]{NO18}):
\begin{thm}{\rm (Reflection Positivity)} \label{thm:JRA}
The triple 
$(\cH^{-1}(M),\cH^{-1}_{M_+}(M),\theta)$  is a reflection
positive Hilbert space with $\cN =(-\Delta +m^2 )\cH^1_0(M_+)$
and $\widehat\cE \simeq \cE_0 = \cH_{M^\sigma}^{-1}(M)$.
\end{thm}

\begin{cor} \mlabel{cor:2.9}
The positive definite distribution $\tilde C_m$ on $M \times M$ 
is reflection positive with respect to $(M,M_+,\sigma)$, 
i.e., the analytic kernel 
\[ \Psi_m(x,y) := \Phi_m^\sigma(x,y) := \Phi_m(x,\sigma(y))\] 
on  $M_+ \times M_+$ defined in {\rm Lemma~\ref{le:Phim}} is positive definite. 
\end{cor}

The following lemma is a bridge 
between Dimock's approach and the papers by Jaffe and Ritter: 

\begin{lem} \mlabel{lem:2.9} 
The closed subspace $\cH^{-1}_{\oline{M_+}}(M) \subeq \cH^{-1}(M)$ 
is generated by the subspace $C^\infty_c(M_+)$ of test functions on $M_+$, 
considered as elements of $\cH^{-1}(M)$. 
\end{lem}

\begin{proof} Let $\cK_\pm \subeq \cH^{-1}(M)$ denote the closed subspace generated 
by $C^\infty_c(M_\pm)$. 

\nin (a) $\cK_+ + \cK_-$ is dense in $\cH^{-1}(M)$ because 
the subspace $C^\infty_c(M_\pm)$ is dense in $L^2(M_\pm)$ and 
$L^2(M)\cong L^2(M_+) \oplus L^2(M_-)$ is dense in $\cH^{-1}(M)$. 

\nin (b) $\cK_\pm \subeq \cH^{-1}_{\oline{M_\pm}}$: 
By Dimock's Theorem~\ref{thm:E0}, 
$\cH^{-1}_{\oline{M_\pm}}$ is the orthogonal complement 
of the closed subspace $(m^2 - \Delta) \cH^1_0(M_\mp)$, hence contains $\cK_\pm$ 
because $\varphi_\pm \in C^\infty(M_\pm)$ implies 
\[ 
\la \varphi_+, (m^2 - \Delta) \varphi_- \ra_{-1} 
= \la (m^2 - \Delta)\varphi_+,  \varphi_- \ra_{-1} 
= \la \varphi_+, \varphi_- \ra_{L^2} = 0.\] 

\nin (c) $(m^2 - \Delta) \cH^1_0(M_\pm)\subeq \cK_\pm$ follows from 
$(m^2 - \Delta) C^\infty_c(M_\pm) \subeq \cK_\pm$ and the density 
of $C^\infty_c(M_\pm)$ in $\cH^1_0(M_\pm)$. 

\nin (d) From (b) and (c) and Dimock's Theorem, we obtain the orthogonal decomposition 
\[ \cK_\pm = (\cK_\pm \cap \cH^{-1}_{M^\sigma}(M)) \oplus (m^2 - \Delta) \cH^1_0(M_\pm).\] 
Further $\theta(\cK_\pm) = \cK_\mp$ and 
$\cE_0 \subeq \Fix(\theta)$ imply 
$\cK_+ \cap \cH^{-1}_{M^\sigma}(M)) = \cK_- \cap \cH^{-1}_{M^\sigma}(M))$. 
Therefore 
\[\cK_+ +  \cK_- 
= (\cK_+ \cap \cH^{-1}_{M^\sigma}(M)) 
\oplus (m^2 - \Delta) \cH^1_0(M_-) 
\oplus (m^2 - \Delta) \cH^1_0(M_+), \] 
so that (a) and Dimock's Theorem show that 
$\cK_+ \cap \cH^{-1}_{M^\sigma}(M) = \cH^{-1}_{M^\sigma}(M)$.
This completes the proof. 
\end{proof}

\begin{lem}\label{le:ctm} Define $\Ctm : C_c^\infty (M_+)\to L^2(M_+)$ 
by 
$\Ctm(f)  :=(\sigma_* C_m)(f)\res_{M_+}= (C_m\sigma_*)(f)\res_{M_+}$.
Then, for all $\eta\in C_c^\infty (M_+)$, we have:
\begin{itemize}
\item[\rm(i)] $(m^2-\Delta)\Ctm \eta = 0$ on $M_+$ and $\Ctm\eta $ is analytic on $M_+$. 
\item[\rm(ii)] $(\Ctm   \eta)(x)=\int_{M_+} \Psi_m(x, y)\eta(y)\, d\mu (y)
=\int_{M_-}\Phi_m (x,y) \eta (\sigma (y))\, d\mu (y)$ 
for $x \in M_+$ and $\Phi_m$ is analytic on $M_+\times M_+$.
\end{itemize}
\end{lem}

\begin{proof}  As $\sigma (x)\in M_-$ for $x\in M_+$, (i) follows from Lemma \ref{le:Phim} and the fact that $-\Delta + m^2$ is elliptic. It now follows that, for $\eta \in C_c^\infty (M_+)$, we have
\[(-\Delta +m^2)C^\sigma_m \eta (x)=\int_{M_+} (-\Delta +m^2)\Psi_m(x,y)\eta (y)d\mu (y)=0.\]
Hence $(-\Delta +m^2)\Psi_m(x,y)=0$ for all $y \in M_+$. 
Hence $y\mapsto \Psi_m(x,y)$ is
analytic for all $x\in M_+$. As $-\Delta +m^2$ is symmetric, it follows that $\Phi_m (x,y)
=\overline{\Phi_m(y,x)}$. Hence $\Psi_m (\cdot ,x)$ is also analytic 
on~$M_+$. 
\end{proof}

\begin{rem} If $\eta\in C_c^\infty (M_+)$, then $\sigma (\supp (\eta))\subset M_-$ is compact. 
Lemma~\ref{le:ctm}(i) thus shows 
that $\Ctm \eta$ is analytic on the open 
subset $M\setminus \sigma (\supp \eta) \supeq M_- \cup M^\sigma$.
\end{rem}

 \subsection{Symmetric spaces}
\mlabel{se:SymSp}

\begin{definition} \mlabel{def:ss} (a) Let $M$ be a smooth manifold and 
$\mu \: M \times M \to M, (x,y) \mapsto x \cdot y =: s_x(y)$ 
be a smooth map with the following properties: 
each $s_x$ is an involution for which $x$ is an  isolated fixed point and 
\[ s_x(y \cdot z) = s_x(y)\cdot s_x(z) \quad \mbox{ for all } \quad x,y \in M, \quad 
\mbox{ i.e.,} \quad  s_x \in \Aut(M,\mu).\] 
Then we call $(M,\mu)$ a {\it symmetric space}. 

(b) A morphism of symmetric spaces $M$ and $N$ is a smooth
map $\varphi : M\to N$ such that $\varphi (x\cdot y)= \varphi (x)\cdot \varphi (y)$.

(c) The real line $\R$ is a symmetric space with respect to 
$s_x(y) = 2x - y$ and a {\it geodesic in $M$} is a smooth morphism 
$\gamma \: \R \to M$ of symmetric spaces.
 
In \cite{Lo69} it is shown that, for every 
$v \in T_p(M)$, there is a unique geodesic $\gamma_p^v \:\R \to M$ 
with $\gamma_p^v(0) = p$ and  $(\gamma_p^v)'(0) = v$. The corresponding map 
\[ \Exp \: TM \to M, \quad \Exp(v) := \Exp_p(v) := \gamma_p^v(1) \] 
is called the {\it exponential function of $M$}. It satisfies 
$\gamma_p^v(t) = \Exp_p(tv)$ for all $t \in \R$. 
\end{definition}

As shown in \cite{Lo69}, connected symmetric spaces are homogeneous spaces 
of Lie groups and they arise from the following construction that
goes back to \'E. Cartan, see \cite{Hel78} for detailed discussion.
Let $G$ be a Lie group and $\theta : G\to G$ be an involution. Let $K\subset G$ be a
subgroup such that $(G^\theta)_0\subseteq K\subseteq  G^\theta$. 
Then the homogeneous space $M := G/K$ is a symmetric space with respect to 
$s_{gK}(xK) := g\theta(g^{-1}x)K$.  
We write $m_0=eK \in M$ for 
the canonical base point and 
$\theta_M$ for the reflection $s_{m_0}(gK) = \theta(g)K$ 
in the base point~$m_0$. 

Denote by $d\theta : \fg \to \fg$ the derived involution and let 
\[ \fk := \{x\in \fg\: d\theta(x )= x\} \quad \mbox{ and } \quad 
\fp:= \{x\in\fg\: d\theta (x )=-x\}.\] Then $\fg = \fk\oplus \fp$, $\fk$ is the Lie algebra of $K$ and $\fp$ can be identified with 
the tangent space $T_{m_0}(M)$. The isomorphism is given
by the tangent map $T_e(q)\res_{\fp} \: \fp \to T_{m_0}(M)$, where 
$q \: G \to M, g \mapsto gK$ is the quotient map. We note that
if $x\in \fq$ and $\varphi \in C^\infty (M)$ then 
$T_e(q)(x)f=\frac{d}{dt}\big|_{t=0}f(\exp (tx).m_0)$.  

Assume that $K$ is compact. Then there is a $G$-invariant 
Riemannian structure on  $M$ given in the following way:  
Fix a $K$-invariant inner product $\la \cdot,\cdot \ra$ on $\fp\cong T_{m_0}(M)$, 
which is possible because $K$ is compact.  
Write $\ell_g: M\to M$ for the diffeomorphism
$\ell_g(x);=g.x$. We then define a $G$-invariant metric on $T(M)$ by
\[g_{g.m_0}((d\ell_g)_{m_0}(u),(d\ell_g)_{m_0}(v)):=\ip{u}{v}\, .\] 
The so-obtained Riemannian manifold is called a {\it Riemannian 
symmetric space}. 

\begin{thm}\mlabel{the:ctmSp}
Let $M=G/K$ be a Riemannian symmetric space, 
$\sigma : M\to M$ be a dissecting 
isometric involution, and 
$m_0 \in M_+$ such that there exists an 
involutive automorphism $\tau$ of $G$ with $\tau(K) = K$ 
such that 
\[ \sigma (g.x )= \tau (g).\sigma (x)\quad \mbox{ for } \quad g \in G, x \in M\]
holds for an involutive automorphism $\tau$ of $G$. 
 Let 
\[ \varphi_m(x):=\Phi_m (x,m_0), \qquad x \not=m_0 \quad\mbox{ and } \quad 
\psi_m (x):=\ctm (x,m_0) = \varphi_m(\sigma(x)), \qquad x \not= \sigma(m_0). \] 
Then the following assertions hold: 
\begin{itemize}
\item[\rm(1)] $\varphi_m$ is a $K$-invariant analytic function on $M\setminus \{m_0\}$ and satisfies
the differential equation
$\Delta \varphi_m  = m^2\varphi_m$. 
\item[\rm(2)]   $\psi_m$ is
a $K$-invariant analytic function on $M\setminus \{\sigma (m_0)\}$ satisfying the differential equation $\Delta \psi_m =
m^2\psi_m$.
\item[\rm(3)] $(\Ctm \eta)(x) 
=\int_{G} \psi_m(\tau (h)^{-1}.x)\eta (h.m_0) \, dh$ for 
$\eta \in C^\infty_c(M_+)$ and  $x\in M_+$. 
\end{itemize}
\end{thm}

\begin{proof} (1) follows from Lemma~\ref{le:Phim}(iv),(v),  
and (2) from Lemma~\ref{le:ctm}.
For (3) we observe that the function $\tilde\eta(h) := \eta (h.m_0)$ 
on $G$ is supported in $\{h\in G\: h.m_0\in M_+\}$. 
Furthermore, the singularity of
$\psi_m$ is in $\sigma(m_0)$, and if $\tau(h)^{-1}.x=\sigma (m_0)$, 
then $x=\sigma (h.m_0)\in M_-$, so that the singularity is not contained in 
the support of $\tilde\eta$. 
\end{proof}

\begin{rem} Theorem~\ref{the:ctmSp} applies in particular 
to the dissecting involution on $\bS^n$ defined by 
the reflection $\sigma = r_0$, the base point $m_0 = e_0$, 
and $\tau(g) = r_0 g r_0$ on $G = \OO_{n+1}(\R)$. 
\end{rem}

\subsection{The sphere and the hyperboloid as symmetric spaces}\label{se:spSym}

Both the sphere $\bS^n\subeq \R^{n+1}$ and the hyperboloid $\bH^n\subeq \R^{1,n}$ 
are Riemannian symmetric spaces. 

The tangent bundle of the sphere is given by $T(\bS^n)=\{(u,v)\in \bS^n\times \R^{n+1}\:
v\perp u\}$ with the $\OO_{n+1}(\R)$ action $g.(u,v)=(gu,gv)$. 
Geodesics and exponential map are given by
\begin{equation}\label{eq:Exp}
\Exp_p(v)= \cos (\|v\|)p + \sin (\|v\|)\frac{v}{\|v\|} 
\quad \mbox{ for } \quad 0 \not= v \in T_p(\bS^n) \cong p^\bot \cong \R^n
\end{equation}
and in particular 
\begin{equation}\label{eq:Exp'}
\Exp_p(tv)= \cos (t)p + \sin (t)v 
\quad \mbox{ for } \quad \|v\| = 1. 
\end{equation}

The exponential function can be dealt with more easily if we use the 
analytic functions $C, S \: \C \to \C$ defined by 
\begin{equation}
  \label{eq:CandS}
 C(z) := \sum_{k = 0}^\infty \frac{(-1)^k}{(2k)!} z^{k} \quad \mbox{ and } \quad 
 S(z) := \sum_{k = 0}^\infty \frac{(-1)^k}{(2k+1)!} z^{k}
\end{equation}
which satisfy 
\[ \cos z = C(z^2) \quad \mbox{ and } \quad \frac{\sin z}{z} = S(z^2) 
\quad \mbox{ for } \quad z \in \C^\times.\]
We thus obtain 
\begin{equation}
 \label{eq:exp-rel}
\Exp_{p}(v) = C(v^2 ) p + S(v^2) v.
\end{equation}

On the tangen bundle $T (\H^n)\simeq \{(u,v)\in \H^n\times \R^{n+1}\: [u,v]=0\}$ 
of $\bH^n$, the Riemannian structure is given by
$g_u((u,v),(u,v))=-[v,v]$. This shows that the Lorentz 
group $\OO_{1,n}(\R)^\uparrow$ acts by isometries on $\bH^n$. 
The stabilizer of $e_0$ is again the group $K$ and 
$\H^n\cong\OO_{1,n}(\R)^\uparrow/K$. 
 
\section{The complex manifold $\Xi$} 
\mlabel{sec:3} 

\noindent
In this section we take a closer look at the subset 
$\Xi = G^c.\bS^n_+$ of the complex sphere $\bS^n_\C$, 
defined in the introduction. 
This domain turns out to be open and thus inherits a complex 
manifold structure. It contains the half sphere 
$\bS^n_+$ and the hyperbolic space $\bH^n_V$ as totally real submanifolds 
$\bS^n_+=\Xi^{\sr}$ and $\wH=\Xi^{\sx}$.  
Thus if $F: \Xi \to \C$ is holomorphic and $F|_{\bS^n_+}=0$ or $F|_{\wH}=0$ then 
$F= 0$. We therefore consider $\Xi$ as a bridge between analytic functions 
on $\bS^n_+$ and $\bH^n_V$, and we shall study this connection in particular 
for functions invariant under $K = \OO_n(\R)$. 
In Subsection~\ref{se:SphCr} we explore some elementary properties 
of $\Xi$ and calculate the set $\C_\Xi$ of all values of the 
complex bilinear form $[\cdot, \cdot]_V$ on $\Xi$. This permits us later 
to obtain analytic continuations of $G^c$-invariant kernels on $\bH^n_V$ 
and from certain kernels on $\bS^n_+$. In Subsection~\ref{se:Boundary} 
we show that the boundary of $\Xi$ in $\bS^n_\C$ consists of 
two $G^c$-orbits: de Sitter space $\dS^n$ and the space $\bL^n_+$, 
the non-zero boundary elements of the positive light cone $V_+$. 
These two boundary orbits are of particular importance for realizations of 
$G^c$-representations on these spaces.

\subsection{The sphere and the crown of the hyperboloid} 
\label{se:SphCr}

The complex sphere $\bS^n_\C$ is a complex symmetric space and the 
reflections are given by the same formula as for the
 sphere. Its exponential function is given by (\ref{eq:exp-rel}). 
For $p = e_0$ and $\|v\|=1$, we obtain in particular 
\[\Exp_{e_0}(it v)=\cosh (t)e_0 + i\sinh (t)v\in \iota \H^n\subset \bS^n_\C,  
\quad \mbox{ where } \quad \iota (x_0,\bx)=(x_0,i\bx).\]
\begin{lem}\label{le:HinSn}
Let $u,v \in V$. Then $z=u+iv \in\bS^n_\C$ if and only if
\[\lf{u}{u} - \lf{v}{v}=1\quad \text{and} \quad \lf{u}{v}= 0. \]
\end{lem}

\begin{proof} A simple calculation shows that
$z^2 =\lf{u}{u}-\lf{v}{v}+2i\lf{u}{v}$ 
and the claim follows.
\end{proof}

\begin{prop}\label{prop:XiTube} The following assertions hold: 
\begin{itemize}
\item[\rm(i)] $T_{V_+} \cap \R^{n+1} = \R^{n+1}_+ := \{
(x_0,\bx)\: x_0 > 0\}$ and 
$\bS_+^n=\Tu \cap \bS^n$. 
\item[\rm(ii)] $\Xi =\Tu \cap \bS^n_{+,\C}=\Tu\cap \bS^n_\C$.
\item[\rm(iii)] We have $\sx \Xi = \sr\Xi =\Xi$ and 
$\Xi^{\sx}= \Xi \cap V = \wH$ and $\Xi^{\sr} = \Xi \cap \R^{n+1}= \bS^n_+.$
\end{itemize}
\end{prop}

\begin{proof} (i) It is clear that $z=u+iv\in \Tu\cap \R^{n+1}$ if and only if $u=re_0$ with
$r>0$ and $iv = (0,\bv)$ with $\bv\in  \R^n$. This shows that 
$T_{V_+} \cap \R^{n+1} = \R^{n+1}_+$. Intersecting with the sphere now yields 
the second assertion. 

(ii) By (i) we have $\Xi =G^c.\bS^n_+= G^c.(\Tu \cap \bS^n)
\subseteq \Tu \cap \bS^n_{\C}$.  
Conversely, let 
$z=u+iv \in \Tu\cap \bS^n_{\C}$. Then $u_0>0$ and, as $G^c$ acts transitively on all level sets $\lf{u}{u}=r>0$ in $V_+$, we
may
assume that $u=re_0$ with $r>0$. Thus $z=(r+iv_0,\bv)$ 
with $v_0 \in \R$ and $\bv \in \R^n$. As $z\in \bS^n_\C$, we have 
\[1=z^2= r^2-v_0^2+2irv_0 +\|\bv\|^2\, .\]
Hence $v_0=0$ and this implies that  $z\in \bS_+^n\subset \Xi$. Finally, we note that, if 
$z\in \Tu$, then $\Re z_0>0$ hence $\Tu\cap \bS^n_\C = \Tu \cap \bS^n_{+,\C}$.

(iii) follows from (i) and (ii). 
\end{proof}

\begin{cor} $\Xi$ is an open subset of $\bS^n_{\C}$,  
hence a complex manifold, and the 
group $G^c$ acts on $\Xi$ by holomorphic maps. 
\end{cor}

\begin{proof} The tube domain $\Tu$ is open in $\C^{n+1}$ and hence $\Xi$ is an open subset of the complex manifold
$\bS^n_{\C}$ and hence a complex submanifold.  The last statement is clear as this is the restriction of the
linear action of $\OO_{n+1}(\C)$ to $G^c$.
\end{proof}

\begin{ex} \mlabel{ex:n1} 
Let us take a closer look at the case $n=1$.
We parametrize the complex $1$-sphere with the biholomorphic map 
\[ \zeta \: \C^\times \to \bS^1_\C\subeq \C^2, \quad 
\zeta(z) = 
\Big(\frac{1}{2}\Big(z +\frac{1}{z}\Big), \frac{1}{2i}\Big(z -\frac{1}{z}\Big)\Big), 
\qquad \zeta(1) = (1,0) = e_0, \] 
whose inverse is given by $\zeta^{-1}(z_0,z_1) = z_0 + i z_1$. 
That this map is biholomorphic is most easily seen by writing 
$z = e^{iw}$ with $\cos(w) = \frac{1}{2}(z + z^{-1})$ and 
$\sin(w) = \frac{1}{2i}(z - z^{-1})$, which leads to 
$\zeta(z) = (\cos w, \sin w)$ and in particular to 
$\zeta(1) = e_0$ and  $\zeta(i) = e_1$. 
The map 
$\zeta$ intertwines the multiplication action of $\T$ on $\C^\times$ with the 
action of $G_0 = \SO_2(\R)$ on $\bS^1_\C$. 
Further, the coordinate reflections act on $\C$ by 
$r_0(z) = -z^{-1}$ and $r_1(z) = z^{-1}$. 
Accordingly, $G^c_0 \cong \R^\times_+$.  

The tube domain is given by 
\[ T_{V_+} = \{ z \in \C^2 \: |\Im z_1| < \Re z_0\}.\] 
To determine $\zeta^{-1}(\Xi)$, we observe that 
$\Re(z + z^{-1}) > 0$ if and only if $\Re z > 0$, so that 
\[ \bS^1_{+, \C} = \zeta(\C_+), \quad \mbox{ where  }\quad 
\C_+ := \{ z \in \C \: \Re z > 0\}\] 
is the open right half-plane. 
If $z = x + i y$ with $x > 0$, then 
\[ |\Im \zeta(z)_1| 
= \frac{1}{2}|\Re(z-z^{-1})| 
= \frac{1}{2} x\Big|1 - \frac{1}{x^2 + y^2}\Big| 
<  \frac{1}{2} x\Big(1 + \frac{1}{x^2 + y^2}\Big)
= \Re \zeta(z)_0 \] 
holds automatically. By Proposition~\ref{prop:XiTube}, this shows that 
\begin{equation}
  \label{eq:xi-for-n1}
\Xi = \bS^1_{+, \C} = \zeta(\C_+).
\end{equation}

Since we shall need it later, we calculate the set $\C_\Xi$ of values 
of the kernel $[\cdot,\cdot]_V$ on $\Xi$. 
We have 
\[ [\zeta(z), \zeta(w)]_V 
= \frac{1}{4}\Big(z +\frac{1}{z}\Big)\Big(w +\frac{1}{w}\Big)
- \frac{1}{4}\Big(z -\frac{1}{z}\Big)\Big(w -\frac{1}{w}\Big)
= \frac{1}{2}\Big(\frac{z}{w} + \frac{w}{z}\Big).\] 
For $z,w \in \C_+$, the values of $zw^{-1}$ are the products of elements 
in $\C_+$, which leads to the set $\C \setminus (-\infty,0]$. 
For $z \in \C^\times$, the complex number 
$\frac{1}{2}(z + z^{-1})$ is contained in $(-\infty, -1]$ 
if and only if $z \in (-\infty,0)$. This implies that 
$\C_{\Xi} = \C \setminus (-\infty,-1]$. 
\end{ex}

With the precise information on the case $n =1$, we can determine 
the set $\C_\Xi$ in general: 

\begin{lem} \mlabel{lem:xi-values}
$\C_\Xi := \{ [z,w]_V \: z,w \in \Xi\} 
= \C\setminus (-\infty , -1]$.
\end{lem} 

\begin{proof}  The case $n = 1$ has been treated in 
Example~\ref{ex:n1}. Therefore it suffices to show that 
the intersection $\C_\Xi \cap (-\infty,-1]$ is empty. 
For $z = x + i y \in T_{V_+}$ with $x,y \in V$, we have
\[ [z,z]_V = [x,x]_V - [y,y]_V + 2i [x,y]_V.\]
If $[z,z]_V \in \R$, then $[x,y]_V = 0$. 
As $[x,x]_V > 0$, this implies that $[y,y]_V \leq 0$ because the form 
is Lorentzian. 
Therefore 
\begin{equation}
  \label{eq:dag}
[z,z]_V \geq [x,x]_V > 0.
\end{equation}
For $u,v \in \Xi = T_{V_+} \cap \bS^n_\C$, we now have 
$u + v \in T_{V_+}$ and 
\[ [u+v, u+v]_V = [u,u]_V + [v,v]_V + 2 [u,v]_V 
= 2(1 + [u,v]_V).\] 
If $[u,v]_V$ is real, then so is $[u+v,u+v]_V$, and we obtain from the 
preceding paragraph that 
\[ 0 <  [u+v, u+v]_V = 2(1 + [u,v]_V),\quad \mbox{ i.e.,} \quad 
[u,v]_V > -1. \qedhere \]
\end{proof}

\begin{rem} Let $\Omega =\{v\in \R^n\: \|v\|<\pi/2\}=K. (-\pi/2,\pi/2)e_n$. Then\eqref{eq:Exp} implies that
\[\bS_+^n=\Exp_{e_0} \Omega = K.\Exp_{e_0}((-\pi/2,\pi/2)e_n)\, .\]
Thus 
\[ \Xi =G^c.\Exp_{e_0}(\Omega )=G^c.\Exp_{e_0}((-\pi/2,\pi/2)e_n).\] 
This  shows that $\Xi$ is indeed the crown domain 
\[ G_c \exp\Big(\Big\{ i x \: x\in \fp, \|\ad x \| < \frac{\pi}{2}\Big\}\Big) K_\C /K_\C 
\subeq G_\C/K_\C \] 
of  the Riemannian symmetric space $G^c.e_0 \cong G^c/K$ 
(cf.\ \cite[p.~32]{GKO03}). For
more information about the crown see
\cite{AG90,KS04,KS04,KO08}. The crown is the  maximal $G^c$-invariant 
domain for the holomorphic extension of all spherical functions on the
Riemannian symmetric space $G^c/K$. It is shown in \cite{KS05} 
that in our case the crown is the Lie ball, which can be identified with the
Riemannian symmetric space $\SO_{2,n}(\R)_0/(\SO_2(\R) \times \SO_n(\R))$; 
see also the Section \ref{sec:lieball} for a direct argument.
\end{rem}

\subsection{The boundary of $\Xi$} 
\label{se:Boundary}

In this subsection we show that the boundary of the crown consists of two types of
orbits. We then consider the projection of the orbits onto $V$ and $iV$, respectively: 
 \begin{center}
$\displaystyle\large
{\hskip -1.5cm} 
\begin{diagram}    \node[2]{ \partial \Xi} \arrow{sw,t}{p_V} \arrow{se,t}{p_{iV}}   \\ 
\node{V }  \node[2]{ iV}
\end{diagram}$
\end{center}
Both $p_V$ and $p_{iV}$ are $G^c$-equivariant maps. 
Hence the projection of a $G^c$-invariant subset is again $G^c$-invariant
in $V$, resp., $iV$. 

\begin{lem}\label{lem:3.1} The boundary of $\Xi$ in $\bS^n_\C$ is given by 
  \begin{equation}
    \label{eq:boundxi}
\partial \Xi =\{u+iv\: u,v \in V, 
\lf{u}{u}=0, u_0\ge 0, \lf{v}{v}=-1, \lf{u}{v} =0\}.
  \end{equation}
\end{lem}

\begin{proof}  Recall from Proposition \ref{prop:XiTube} that
$\Xi = T_{V_+}\cap \bS^n_\C$. In view of Lemma \ref{le:HinSn}, this means that, 
for $u,v \in V$, the element $z=u+iv$ is contained in $\Xi$ 
if and only if $\lf{u}{u}-\lf{v}{v}=1$, $\lf{u}{v}=0$, $\lf{u}{u}>0$ and $u_0>0$. This shows that 
\[\overline{\Xi}\subeq   
\{u+iv\in V_\C\: \lf{u}{u}-\lf{v}{v}=1,\, \lf{u}{v}=0,\, \lf{u}{u}\ge 0,\, \text{ and } u_0\ge 0\}.\]
To see that we actually have equality, suppose that 
$z = u + i v$ is contained in the right hand side but not in $\Xi$. 

If $u_0=0$, then $\lf{u}{u}=u_0^2-\bu^2=-\bu^2\ge 0$ implies that $u=0$. 
This shows that $z = i v$ with $[v,v]_V = -1$. Then there exists an 
element $\tilde u \in V_+$ with $[\tilde u,v]_V = 0$, so that 
\break $z_\eps := (1 + \eps^2)^{-1}(\eps \tilde u + i v) \in \Xi$ with 
$z_\eps \to iv$ for $\eps \to 0$. 

If $u_0>0$, then we must have $\lf{u}{u}=0$ and $[v,v]_V = -1$. 
Acting with a suitable element of $G^c$, we may assume that 
$v_0 = 0$. For $\eps > 0$ small enough we then have 
$z_\eps := (1 +2\eps u_0 +  \eps^2)^{-1}(z + \eps e_0) \in \Xi$ with 
$z_\eps \to z$ for $\eps \to 0$. 
\end{proof} 

If $u=0$, then \eqref{eq:boundxi} leads to a realization of \textit{de Sitter space} 
\[\dS^n:=i \{v\in V\: \lf{v}{v} =-1\} 
=\bS^n_\C \cap i V= \partial \Xi \cap i V
=p_{iV}(\partial \Xi)\, .\]
The stabilizer $G^c_{e_n}$ of the point $e_n \in \dS^n$ 
in the group $G^c$ is $H= \iota \OO_{1, n-1}(\R)^\uparrow\iota$, 
which is a non-compact symmetric subgroup of $G^c$ with respect to the
involution given by conjugation by $r_n$. Hence $\dS^n\cong G^c/H$ 
is a Lorentzian symmetric space.

Before discussing the other boundary orbits, 
some more notation is needed. 
Let 
\begin{align} M &:=\left\{m_A:=\begin{pmatrix} 1 & 0 & 0\\
0 & A & 0\\
0 & 0 & 1\end{pmatrix}\: A\in \OO_{n-1}(\R )\right\}\simeq \OO_{n-1}(\R)\\
  N&:=\left\{n_v:=\begin{pmatrix} 1+\frac{1}{2}\|v\|^2 &-iv^T & \frac{i}{2}\|v\|^2\\
iv & \rI_{n-1} & -v\\
\frac{i}{2}\|v\|^2& v^T& 1-\frac{1}{2}\|v\|^2\end{pmatrix}\: v\in\R^{n-1}\right\}
\simeq (\R^{n-1},+) \label{nv} \\
A&:=\left\{ a_t=\begin{pmatrix} \cosh (t) & 0 & -i\sinh (t)\\ 0 & \rI_{n-1} & 0\\ i\sinh (t) & 0 & \cosh (t )\end{pmatrix}\: t\in\R\right\}\simeq (\R,+)\, .
\label{at}\end{align}
The group $MN$ is the stabilizer of the element 
$\xi^0 := e_0 +i e_n$, and 
\begin{equation}\label{eq:atxb}
a_t.\xi^0=e^t\xi^0.
\end{equation} The orbit of $\xi^0$ is 
\begin{equation}
  \label{eq:conic}
\Lnp :=\{v = (v_0, i\bv) \in V\: \lf{v}{v} =0, v_0>0\} = G^c.\xi^0 \simeq G^c/MN 
\end{equation}
of non-zero forward lightlike vectors. 
To see that $\Lnp = G^c.\xi^0$, let $v=(v_0, i\bv)\in \Lnp$ and note that 
$v_0\not= 0$ and $\|\bv\|=v_0>0 $. As
$K \cong \OO_n(\R)$ acts transitively on each sphere in $\R^{n}$, we may 
assume that $\bv=v_0 e_n$, or $v= v_0 \xi_{0}$. Then 
$v = v_0(e_0+ie_n) = a_{\log v_0}.\xi^0$ by (\ref{eq:atxb}).
 
\begin{lem} \mlabel{lem:3.7} Suppose that $n\geq 2$. Let $\cO 
:= G^c. (\xi^0 +e_{n-1})$. Then the 
boundary of the crown is the union of two $G^c$-orbits 
\begin{equation}
  \label{eq:dag2}
\partial \Xi = \dS^n\dot \cup  \cO.
\end{equation}
The projection of $\cO$   onto 
$V$ is $\Lnp = G^c.\xi^0$ and the projection onto 
$iV$ is $\dS^n 
= G^c.e_{n-1}$. 
\end{lem}

\begin{proof} Assume that $z = u+iv \in \partial \Xi \setminus \dS^n$, 
so that $u\not=0$ (cf.~Lemma \ref{lem:3.1}). 
As $[u,u]_V = 0$, the $G^c$-orbit of $u$ contains $\xi^0$, 
so that we may w.l.o.g.\ assume that $u = \xi^0$. 
Acting with the subgroup $K = G^c_{e_0} \cong\OO_n(\R)$, 
we may further assume that $v \in \R e_0 - \R_+ i e_{n-1} + \R i e_n$. 
Then $[u,v]_V = 0$ implies that 
$v = v_0 \xi^0 - a i e_{n-1}$, and $[v,v]_V = -1$ yields $a = 1$. 
We thus obtain $z = (1 + i v_0) \xi^0 + e_{n-1}$. 
In the notation of \eqref{nv}, we have 
\[ n_v.(\xi^0 + e_{n-1})= 
\xi^0 - i v_{n-1} \xi^0 + e_{n-1} 
= (1-iv_{n-1}) \xi^0 + e_{n-1},\] 
so that we may further assume that $v_0 = 0$, which eventually 
shows that $z$ is $G^c$-conjugate to $\xi^0 + e_{n-1}$.
Now the claim follows from 
\[ p_{iV}(\cO) = G^c.p_{iV}(\xi^0 + e_{n-1}) = G^c.e_{n-1} = \dS^n.\qedhere\]
 \end{proof}
 
The tangent space of $\dS^n$ at $e_n$ is the $n$-dimensional Minkowski space 
\[ T_{e_n}(\dS^n) = i V\cap e_n^\perp \simeq  \R^{1,n-1}.\]
By \eqref{eq:exp-rel}, we have 
\begin{equation}
  \label{eq:exp1}
\Exp_{e_n}(z )=S(z^2) z  + C(z^2 )e_n 
\quad \mbox{ for } \quad z\in T_{e_n}(\dS^n)_\C=\C\oplus \C^{n-1}. 
\end{equation}

We now describe how one can obtain the crown by moving inward from the de Sitter 
space $\dS^n$ (see \cite{KS05} for a discussion of the general case):

\begin{thm} \mlabel{thm:3.9}
{\rm($\Xi$ from the perspective of $\dS^n$)} 
Let 
\[ \Omega_{e_n} := \{ v \in i T_{e_n}(\dS^n) = \R \oplus i \R^{n-1} \subeq V\: 
v_0 > 0, \lf{v}{v} > 0\} \] 
be the $n$-dimensional forward light cone and 
\[ \Omega_{e_n}^\pi \:= \{ v \in \Omega_{e_n} \: \lf{v}{v} < \pi^2\}.\] 
For $g \in G^c$ and $p := g.e_n \in \dS^n$, we put 
$\Omega_p := g.\Omega_{e_n}$ and 
$\Omega_p^\pi := g.\Omega_{e_n}^\pi$. 
Then we have
  \begin{align*}
\Xi 
&= G^c.\Exp_{e_n}(\Omega_{e_n}^\pi) = \bigcup_{p \in \dS^n} \Exp_p(\Omega_p^\pi). 
  \end{align*}
\end{thm}

\begin{proof} In view of the $G^c$-invariance of $\Xi$ and the equivariance 
of the exponential map of $\bS^n_\C$, it suffices to verify the first equality. 
From \eqref{eq:exp1} we obtain for $v\in \R_+\oplus i\R^{n-1}\subset iT_{e_n}(\dS^n)$ and 
$\R_+ = (0,\infty)$:
\begin{equation}\label{eq:Expen}
\Exp_{e_n}(v )=S([v,v]_V) v  + C([v,v]_V)e_n 
\end{equation}
and this is contained in $T_{V_+} \cap \bS^n_{\C}=\Xi$ if $[v,v]_V \in (0,\pi^2)$. 
Therefore $\Exp_{e_n}(\Omega_{e_n}) \subeq \Xi$. 
If, conversely, $z \in \Xi = G^c.\bS^n_+$, then there exists a 
$t\in (0,\pi)$ such that $z$ is $G^c$-conjugate to   $x = (\sin t, 0, \ldots, 0, \cos t)$. 
But then $te_0 \in \Omega_{e_n}^\pi$ and 
\eqref{eq:Expen} yields 
$x = \Exp_{e_n}(t e_0)$. This proves the claim. 
\end{proof}

\begin{rem} (a) The proof of Lemma~\ref{lem:3.7} shows that the second orbit 
in \eqref{eq:dag2} is a homogeneous space of the form 
\[ \cO \cong G^c/\OO_{n-2}(\R)\{n_v\in N\: v_{n-1}=0\},\]  
but we will not use that fact in this article. This orbit is not
a symmetric space but we have a double fibration
\begin{center}
$\large\displaystyle
{\hskip -1.5cm} 
\begin{diagram} \node[2]{\cO} \arrow{sw,t}{p_V} \arrow{se,t}{p_{iV}}   \\ 
\node{\Lnp }  \node[2]{ \dS^n}
\end{diagram}$
\end{center}
which might be interesting for harmonic analysis on $\dS^n$.

(b)  We will return to the space $\Lnp$  
in Section~\ref{se:IntRep}. 
We recall that an orbit of a subgroup $gNg^{-1}$, $g \in G^c$, 
in $\wH$ is called a
 \textit{horocycle}. The group $G^c$ 
acts transitively on the set of horocycles which is isomorphic to $G^c/MN\cong 
\Lnp$ because $M$ leaves every horocycle invariant.
As a subset of $\wH$, the basic horocycle is 
\[h_0 :=N.e_0=\{(1+\textstyle{\frac{1}{2}}\|v\|^2, iv, 
\textstyle{\frac{i}{2}}\|v\|^2)\: v\in\R^{n-1}\}\simeq \R^{n-1} .\]
As $N \cap G^c_{e_0} = \{e\}$, the horocycles are paraboloids 
diffeomorphic to $N$ and hence to~$\R^{n-1}$. 

(c) For de Sitter space, the stabilizer group $H = G^c_{e_n}$ 
is a symmetric subgroup of $G^c$ as
pointed out above. But   $H$ is not compact and
$\dS^n$ is not a Riemannian symmetric space, but a pseudo-Riemannian
symmetric space. The signature of the  metric is $(1,n-1)$. The symmetric
space $G^c/H$ is an example of a \textit{non-compactly causal symmetric space}, see \cite{HO96}. The
involution $r_n$ is dissecting and commutes with the action of $H$.
\end{rem}

\begin{ex} In the context of Example~\ref{ex:n1}, where we discuss 
the case $n = 1$ by an isomorphism $\zeta \: \C^\times \to \bS^1_\C$, 
we have $\bL^1_+ = \eset$, 
\[ \zeta(i\R^\times) = \partial \Xi = \dS^1 \subeq i \R \times \R 
\quad \mbox{ and } \quad  \zeta(\R^\times_+) = \bH^1_V \subeq \R \times i \R.\] 
\end{ex}

\section{Reflection positivity on the sphere $\bS^n$}
\mlabel{sec:4} 
\noindent
In this section we specialize the results from Section \ref{se:ReLap} to the sphere $\bS^n$ 
and the dissecting involution $r_0$. We start by discussing $G^c$-invariant 
kernels on~$\Xi$ arising by analytic extension from $G$-invariant 
kernels on $\bS^n$ by twisting with $\sigma$. 
To connect with our 
previous work on reflection positive functions on the circle $\bS^1$ \cite{NO15b}, 
we  then discuss the special case $n=1$. 
After that we turn to the general case and obtain an explicit 
expression for the kernel~$\Psi_m$ on $\Xi$ 
(Theorem~\ref{the:ctmSp}). 

\subsection{Invariant positive definite kernels} 
We say that
$\Psi \: \Xi \times \Xi \to \C$ is {\it sesquiholomorphic} if $\Psi$ is 
holomorphic in the first and antiholomorphic in the second 
argument. Note that this is equivalent to $\Psi$ being holomorphic on 
$\Xi \times \Xi^{\rm op}$, where $\Xi^{\rm op}$ carries the opposite complex structure. 
We write $\Sesh(\Xi)$ for the complex linear space 
of sesquiholomorphic $G^c$-invariant kernels on $\Xi$,  
and $\Gamma := \Gamma (\Xi)\subeq \Sesh(\Xi)$ for the 
convex cone of positive definite kernels. 
The Fr\'echet space of holomorphic functions on $\Xi$ is denoted 
by $\cO(\Xi)$. 
We note that every $\Psi \in\Gamma$ is {\it hermitian} in the
sense that $\Psi (z,w)=\overline{\Psi (w,z)}$.

\begin{lem}\label{le:PsiE0} For $\Psi \in \Sesh(\Xi)$, 
the following are equivalent: 
\begin{itemize}
\item[\rm(i)] $\Psi = 0$. 
\item[\rm(ii)] $\Psi_{e_0} = 0$. 
\item[\rm(iii)] $\Psi_{e_0} \res_{\bH_V^n} =0$.
\item[\rm(iv)] $\Psi_{e_0} \res_{\bS^n_+} =0$.
\end{itemize}
For $\Psi \in \Gamma$, these conditions are further equivalent to 
\begin{itemize}
\item[\rm(v)] $\Psi(e_0, e_0) = 0$. 
\end{itemize}
\end{lem}

\begin{proof} Obviously (i) implies (ii). 
The equivalence of (ii), (iii) and (iv) follows from the fact that 
$\Xi$ is connected and $\bS^n_+$ and $\bH_V^n$ are totally real submanifolds 
of $\Psi$. 

It is also clear that (iv) implies (v). If, conversely, 
$\Psi$ is positive definite and $\Psi(e_0,e_0) =0$, 
then 
\[ |\Psi(z, e_0)|^2 \leq \Psi(z,z)\Psi(e_0,e_0) = 0
\quad \mbox{ for } \quad z \in \Xi, \] 
i.e., $\Psi_{e_0} = 0$. 

If $\Psi_{e_0} = 0$, 
then the $G^c$-invariance of $\Psi$ leads to 
$\Psi_{g.e_0}(z) = \Psi_{e_0}(g^{-1}.z)=0$ for $g \in G^c$, $z \in \Xi$. 
Hence $\Psi_z = 0$ for every $z \in \wH = G^c.e_0$. 
Thus $\Psi (w,z)=0$ for $z\in\wH, w \in \Xi$ and since 
$\wH$ is totally real in $\Xi$, we obtain $\Psi(w,z) = 0$ 
for all $z,w \in \Xi$. 
\end{proof}
 
 \begin{cor}\label{co:UniDet} Any $\Psi\in\Sesh(\Xi)$ is uniquely 
determined by any of the $K$-invariant functions
 $\Psi_{e_0}|_{\wH}$ and $\Psi_{e_0}|_{\bS_+^n}$.
 \end{cor}

Lemma~\ref{le:PsiE0}(v) shows in particular that the convex subset 
\[ \Gamma_1:=\{\Psi\in \Gamma\: \Psi (e_0,e_0)=1\} \] 
is a base of the cone $\Gamma$, i.e., the linear functional 
$\Psi \mapsto \Psi(e_0, e_0)$ is strictly positive on 
$\Gamma \setminus \{0\}$. We write 
\[ \Gamma_e := \Ext(\Gamma_1) \] 
for the set of extreme points of $\Gamma_1$ which represent 
the extremal rays of the cone $\Gamma$. 
For $y\in \Xi$ and $\Psi\in \Gamma$, we obtain a holomorphic function 
$\Psi_y\in \cO(\Xi)$ by $\Psi_y(x):=\Psi (x,y)$ such
that
\[\Psi_y(x)=\Psi (x,y)=\overline{\Psi (y,x)}=\overline{\Psi_x(y)}\quad\text{and}\quad
\Psi_y(g.x)=\Psi_{g^{-1}.y}(x) 
\quad \mbox{ for } \quad x,y \in \Xi, g \in G^c. \]
 
Our next goal is to identify the elements in $\Gamma_e$.  
We start with the following easy geometric lemmas: 
 
\begin{lem}\label{le:3.1}
Two pairs $(x,y), (z,w) \in \bS^n \times \bS^n$ are conjugate under 
$G = \OO_{n+1}(\R)$ if and only if $x  y=z  w$. 
\end{lem}

\begin{rem} For $n > 1$, Lemma~\ref{le:3.1} remains true 
with $\SO_{n+1}(\R)$ instead of $\OO_{n+1}(\R)$. 
However, for $n = 1$, one needs an additional invariant to separate 
the $\SO_2(\R)$ orbits of pairs, namely 
the determinant $\det(x,y) := x_0 y_1 - x_1 y_0$, to determine 
the oriented angle between $x$ and $y$. 
\end{rem}

\begin{lem} \mlabel{lem:alphafun}
Let $\Omega \subeq \bS^n\times \bS^n$ be a subset invariant under 
the diagonal action of $G = \OO_{n+1}(\R)$. 
For a function $\Phi \: \Omega \to \C$, 
the following are equivalent:
\begin{itemize}
\item[\rm(i)] $\Phi$ is $G$-invariant.
\item[\rm(ii)] There  exists a $K$-biinvariant function 
$\varphi_\Phi$ on $G_\Omega := \{ g \in G \: (g.e_0, e_0)\in \Omega\}$ 
such that
$\Phi(g_1.e_0,g_2.e_0)=\varphi_\Phi(g_2^{-1}g_1)$.
\item[\rm(iii)] There exists a function $\alpha_\Phi : 
\{ xy \: (x,y) \in \Omega \} \to \C$ such that
$\Phi (x,y)=\alpha_\Phi (x  y)$ for $(x,y)~\in~\Omega$.
\end{itemize}
\end{lem}

\begin{defn}
An analytic kernel $\Psi $ on $\bS^n_+\times \bS^n_+$ 
is said to be {\it $(\fg^c, K)$-invariant} if the following 
conditions hold for all $(x,y)\in\bS^n_+$: 
\begin{itemize}
\item[\rm (i)]   For $X\in \fg^c$ we write 
\[ (\cL^1_X F)(x,y) 
= \frac{d}{dt}\Big|_{t=0} F(\exp(tX).x,y) \quad \mbox{ and } \quad 
(\cL^2_X F)(x,y) 
= \frac{d}{dt}\Big|_{t=0} F(x,\exp(tX).y)\] 
and define $\cL_X^j$ for $X \in \g_\C$ by complex linear extension. 
Then 
\begin{equation}\label{eq:Inv}
\cL^1_X\Psi (x,y)=-\cL^2_X\Psi (x,y)\quad \text{for all } X\in \fg^c, 
\end{equation} 
\item[(ii)] For $k\in K$ we have $\Psi (k.x,k.y)=\Psi (x,y)$.
\end{itemize}
\end{defn}

If $\Psi \in \Sesq(\Xi)$ is real-valued 
and symmetric on $\bS^n_+\times \bS^n_+$, then 
the kernel $\Psi^*(z,w) := \oline{\Psi(w,z)}$ coincides with 
$\Psi$ on the totally real submanifold $\bS^n_+ \times \bS^n_+$ of 
$\Xi \times \Xi^{\rm op}$, hence on $\Xi \times \Xi^{\rm op}$. 
We conclude that $\Psi$ is hermitian. 

\begin{lem}\label{le:PsiLocInv} Let $\Psi$ be a $(\g^c,K)$-invariant 
analytic kernel on $\bS^n_+\times \bS^n_+$. Then there exist an analytic 
complex-valued function $\alpha_\Psi$ on an open neighborhood 
of $(-1,1]$, such that 
\[ \Psi (x,y)=\alpha_\Psi (\lf{x}{\sx y})\quad \mbox{ for } \quad 
x,y \in \bS^n_+.\]
The kernel $\Psi$ is hermitian if and only if $\alpha_\Psi$ is real valued.
\end{lem}

\begin{proof} {\bf Step 1:} If $y \in \bS_+^n$, then $\sx y\in \bS_+^n$. 
Hence $\Phi (x,y):=\Psi (x,\sx(y))$ is defined on
$\bS^n_+\times \bS^n_+$. As $\Psi$ is analytic and $(\g^c,K)$-invariant, 
the kernel $\Phi$ is $(\fg,K)$-invariant and hence locally $G$-invariant. 
Transforming pairs of points on $\bS^n_+$ by differentiable paths in $G$ 
inside of $\bS^n_+$ into pairs lying in 
$\bS^1 \cong \bS^n \cap (\R e_0 + \R e_1)$, it follows that 
there exists a function $\alpha_\Phi : (-1,1]\to \C$ such that 
\[ \Phi (x,y)=\alpha_\Phi (x  y) \quad \mbox{ for }\quad 
x,y \in \bS^n_+.\]  

\nin {\bf Step 2:} Now we argue that $\alpha_\Phi$ extends to an analytic function 
on an open interval containing $(-1,1]$. 
The analyticity of $\alpha_\Phi$ 
on the open interval $(-1,1)$ immediately follows from the analyticity of $\Phi$. 
To see what happens in $1$, we observe that 
\[ \Phi(x,e_0) = \alpha_\Phi(x   e_0) = \alpha_\Phi(x_0). \] 
For $t$ close to $0$, this leads to 
\[ \Phi(\cos(t)e_0 + \sin(t) e_1,e_0) = \alpha_\Phi(\cos(t)) 
= \alpha_\Phi(C(t^2)) \] 
with $C$ as in \eqref{eq:CandS}. Since this function is symmetric in $t$, 
it follows that $\alpha_\Phi \circ C$ extends to a function which is 
analytic in a neighborhood of $0$. As $C(0) = 1$ and $C'(0) \not=0$, 
the Inverse Function Theorem for holomorphic functions shows that 
$\alpha_\Phi$ extends to an analytic function in a neighborhood of~$1$.

\nin {\bf Step 3:} For $x, y \in \bS^n_+$ we now have 
\[ \Psi (x,y)= \alpha_\Phi (x  \sx y)= \alpha_\Phi (\lf{x}{ \sx y}).\]
 Hence the statement holds with
$\alpha_\Psi :=\alpha_\Phi$. 
As $\sx|_{\R^n}=-r_0\in K$, we have 
\[\alpha_\Psi (x  \sx y)=\alpha_\Psi (\sx x   y)=\alpha_\Psi (y  \sx x).\]
Thus $\Psi (x,y)=\Psi (y,x)$. Hence $\Psi$ is hermitian if and only if 
$\Psi(y,x) = \Psi (x,y)=\overline{\Psi (y,x)}$ which holds if
and only if $\alpha_\Psi$ is real valued. 
\end{proof}
 
\begin{thm}\label{thm:KernSphere} A 
sesquiholomorphic kernel $\Psi$ on $\Xi\times \Xi$ 
is $G^c$-invariant if and only if there exists a holomorphic function 
\[ \alpha_\Psi : \C_\Xi = \C \setminus (-\infty,-1] \to \C 
\quad \mbox{ such that } \quad 
\Psi (z,w)=\alpha_\Psi (\lf{z}{ \sx w})\quad \mbox{ for } \quad 
z,w \in \Xi.\] 
Then $\Psi$ is hermitian if and only if $\alpha_\Psi |_{(-1,1]}$ is real valued.
\end{thm}

\begin{proof} By Lemma~\ref{lem:xi-values}, 
any holomorphic function $\alpha$ on $\C_\Xi$ defines a 
$G^c$-invariant sesquiholomorphic kernel by $\Psi(z,w) 
:=\alpha_\Psi (\lf{z}{ \sx w})$. 

Suppose, conversely, that $\Psi$ is a $G^c$-invariant sesquiholomorphic 
kernel on $\Xi$. 
We have already seen that $\Psi$ is uniquely determined 
by its restriction to  $\bS^n_+\times \bS^n_+$. This
restriction is $(\g^c,K)$-invariant if and only if $\Psi$ is $G^c$-invariant. 
In view of Lemma \ref{le:PsiLocInv}, it therefore remains to show that 
the analytic function $\alpha_\Psi$ extends to the domain $\C_\Xi$. 
Since $\bS^1_\C$ embeds naturally in $\bS^n_\C$, it suffices to verify this 
for $n =1$. 

We recall from Example~\ref{ex:n1} that 
$\bS^1_\C \cong \C^\times$, $\Xi \cong \C_+$ is the open right half-plane 
and 
\[ G^c \cong \R^\times_+ \times \{\id,\sigma\} \quad \mbox{  with } \quad 
\sigma(z) = z^{-1},\] 
where $r \in \R^\times_+$ acts by multiplication. Further, 
$\sigma_V\zeta(z) = \zeta(\oline z)$ for $z \in \Xi$. 
If $\Psi$ is an $\R^\times_+$-invariant sesquiholomorphic kernel, 
we can use the family $(\Psi_w)_{w \in \Xi}$ of holomorphic functions on 
$\C_+$ to obtain a holomorphic function $F$ defined on the domain 
$\Xi \cdot \Xi = \C \setminus (-\infty,0]$ 
such that $\Psi(z,w) = F(z\oline w^{-1})$ holds for $z,w \in \Xi$. 
The $\sigma$-invariance of $\Psi$ yields for $x,y \in \R^\times_+$: 
\[ F(xy^{-1}) = \Psi(x,y) = \Psi(x^{-1}, y^{-1}) = F(x^{-1}y),\] 
and thus $F(z^{-1}) = F(z)$. 
From Example~\ref{ex:n1}  we further recall that 
\[ [\zeta(z),\sigma_V \zeta(w)] = \frac{1}{2}\Big(\frac{z}{\oline w} + \frac{\oline w}{z}\Big) 
= R(z\oline w^{-1}) \quad \mbox{ for } \quad 
R(z) := \frac{1}{2}(z + z^{-1}).\] 
In a $1$-neighborhood the function 
$\alpha_\Psi$ thus satisfies 
\[ \alpha_\Psi(R(z)) = \Psi(z,1) = F(z).\] 
Next we observe that the 
equation $R(z) = w$ has for $w \not\in (-\infty,-1]$ the solutions 
\[ z_{1/2} = w \pm \sqrt{w^2 - 1} \in \C \setminus (-\infty, 0] 
\quad \mbox{ satisfying} \quad 
z_1 z_2 = 1,\] 
so that $F(z_1) = F(z_2)$. 
On $\C \setminus (-\infty,-1]$ we thus obtain by 
\begin{equation}  \label{eq:ana-ext}
w \mapsto F(w + \sqrt{w^2-1}) = F(w - \sqrt{w^2-1}) 
\end{equation}
a well-defined function. In a neighborhood of $1$ it coincides 
with $\alpha_\Psi$, so that it is holomorphic. 
Outside of $1$, both branches of the square root yield the same 
holomorphic function when composed with $F$, so that 
$\alpha_\Psi(w) := F(w \pm \sqrt{w^2-1})$ 
defines a holomorphic function on $\C \setminus (-\infty,-1]$. 
This completes the proof. 
\end{proof} 

\subsection{Reflection positivity on $\bS^1$}
\mlabel{subsec:4.2}

Let us recall the results from \cite{NO15b} where we considered the
case $G=M=\T_\beta =\R/\beta \Z $, $\beta > 0$, with the involution 
$r_1(z)=z^{-1}$. Assume that $\beta = 2\pi$. Then $\T=\T_\beta$ is 
identified with $\bS^1$ by  the map $t + 2\pi \Z \mapsto (\cos t, \sin t)$, 
mapping $0$ to $e_0$ and $\pi/2$ to $e_1$. 
We let $\T_{\beta,+}=\{t+\beta\Z\: 0<t<\beta/2\}$.

Here $G^c_0 :=  \R^\times_+$ act by multiplication,
$G \cong  G_0\rtimes \{\id,r_1\}$ and $G^c = G^c_0\rtimes  \{\id,r_1\}$. 
We also recall from \eqref{eq:xi-for-n1} 
that $\Xi = \C_+$ is the open right half plane and $\bH^1 = \R^\times_+$. 

Here the basic example of reflection positive functions are given by
\[f_\lambda (t)=e^{-t\lambda} +e^{-(\beta -t)\lambda }= 2e^{-\beta \lambda /2}
\cosh \left(\left(\frac{\beta}{2}-t\right)\lambda\right)\]
(Definition~\ref{def:PsoDef}(iii)). 
With $\beta=2\pi$ this becomes $f_\lambda (t)=2e^{-\pi \lambda}\cosh ((\pi -t)\lambda )$.
In general, we have (see \cite[Thm. 3.3]{KL81} or \cite[Thm. 2.4]{NO15b}):
\begin{thm} \label{thm:KL} A $2\pi $-periodic symmetric 
continuous function $\varphi : \R\to \C$ is reflection
positive with respect to $(\T, \T^+,\tau_1)$, i.e., the 
kernels 
$(\phi(t+s))_{t,s \in (0,\pi)}$ and  
$(\phi(t-s))_{t,s \in \R}$ are positive definite, 
if and only if there exists a positive measure $\mu$ on $[0,\infty)$ such that
\[\varphi (t)
=\int_0^\infty e^{-t\lambda } +e^{(\pi -t)\lambda}d\mu (\lambda)
\quad \mbox{ for } \quad  0\leq  t\leq 2\pi .\]
The measure $\mu$ is uniquely determined by $\varphi$.
\end{thm}

This fits well into our current discussion.
For the $\OO_2(\R)$-invariant kernel $\Phi_m$ on 
the complement of the diagonal in $\bS^1 \times \bS^1$ 
(Lemma~\ref{le:Phim}), we write
\[ \wphi_m(t):= \alpha_{\Phi_m}(\cos (t))=\Phi_m((\cos (t),\sin(t)),e_0) \] 
for a function $\alpha_{\Phi_m}$ in $[-1,1)$  
(Lemma~\ref{lem:alphafun}). In this notation the
differential equation in Theorem~\ref{the:ctmSp} becomes 
\[\wphi_m^{\prime\prime}(t)=m^2\wphi_m(t)\quad \mbox{ for } \quad 0 <  t < 2\pi.\]
The solutions are of the form  $\wphi_m (t)=ae^{mt} +be^{-mt}$. 
The invariance under $r_1$ leads to 
\[ \wphi_m(2\pi - t) = \wphi_m(t), \qquad 0 < t < 2\pi,\] 
and hence  to $a=e^{-2\pi m}b$.
Thus
\begin{equation}\label{eq:varpS1}
\wphi_m (t)=b( e^{-mt}+ e^{-2\pi m}e^{mt})=2b e^{-\pi m}
\cosh ((\pi -t)m), \quad 0<  t < 2\pi, 
\end{equation}
which is a multiple of the function $f_m$ from above.   

For the twisted kernel $\Psi(x,y) = \Phi(x,\sigma y)$ corresponding to an 
$\OO_2(\R)$-invariant kernel on the complement of the diagonal of $\bS^1$, 
we have for $|t|, |s| < \pi/2$: 
\begin{align*}
&\Psi((\cos t, \sin t), (\cos s, \sin s)) 
=\Phi((\cos t,\sin t),(-\cos s, \sin s)) \\
&=\Phi((\cos t,\sin t),(\cos(\pi - s), \sin(\pi -s))) 
=\tilde\varphi(t + s + \pi).
\end{align*}  
Therefore the positive definiteness of $\Psi$ is equivalent to the 
positive definiteness of the kernel $(\tilde\varphi(t + s))_{0 < t,s < \pi}$. 

As $\cosh (mt)=\int_\R e^{-\lambda t}d\mu_m (t)$ 
is the Laplace transform of the positive measure 
$\mu_m = \frac{1}{2}(\delta_m+\delta_{-m})$, 
the kernel $\cosh (m(t+s))$ is positive definite on $\R$. 
In particular, the kernel $\Psi_m$ corresponding to 
the function $\psi_m(t) := \varphi_m(\pi + t)$ 
is positive definite on $\bS^1_+ \times \bS^1_+$. 

In complex coordinates $z = e^{it}$ on $\bS^1 \cong \T$, 
we have 
\[ \Psi_m(e^{it}, e^{is}) = \tilde\phi_m(t + s + \pi) 
= 2 b e^{-\pi m} \cosh(m(t+s)) 
= b e^{-\pi m}(e^{m(t+s)} + e^{-m(t+s)}).\]
For $z,w \in \C$ with $\Re z, \Re w > 0$, we thus obtain the 
sesquiholomorphic extension 
\[ \Psi_m(z,w) 
= 2 b e^{-\pi m} \cosh(-m i\log(z/\oline w)) 
= b e^{-\pi m}((z/\oline w)^{im} + (z/\oline w)^{-im}).\] 

\subsection{The kernel  function corresponding to $(m^2-\Delta )^{-1}$}
We now discuss the general case and determine the functions 
$\psi_m$, the functions 
$\phi_m$,  as well as the corresponding  $G^c$-invariant kernel
$\Psi_m$ on $\Xi\times \Xi$. 

For a $K$-invariant function $\varphi \: \bS^n \to \C$, 
we obtain functions 
\[ \alpha \: [-1,1] \to \C \quad \mbox{ by } \quad \phi(x) = \alpha(x_0)\] 
and 
\[ \eta_\phi \: [0,\pi] \to \C \quad \mbox{ by } 
\quad \eta_\phi(t) = \alpha(\cos t)\] 
(cf.~Lemma \ref{lem:alphafun}(iii)). 
Then  \cite[Cor. 9.2.4]{Fa08} shows that 
\begin{lem} For $\varphi \in C^\infty (\bS^n )^K$, we have: 
\begin{align}
\eta_{\Delta \varphi} (t) &= \eta_\varphi ^{\prime\prime}(t)+(n-1)\cot (t )\eta_\varphi ^\prime (t)\label{eq:RadPart}
=\frac{1}{\sin^{n-1}(t)}\dfrac{d}{dt}\left(\sin^{n-1}(t)\dfrac{d}{dt}\right ) \eta_\varphi (t) \quad \mbox{ for } \quad 
0 < t < \pi\, .\nonumber
\end{align}
In particular, 
$\Delta \varphi  =m^2\varphi$ if and only if $\eta_\varphi$ satisfies the second order differential equation
\begin{equation}\label{eq:RadPart2}
\eta_\varphi^{\prime\prime}(t)+(n-1)\cot (t)\eta_\varphi ^\prime (t)
-m^2 \eta_\varphi (t)=0.
\end{equation}
\end{lem}

For $n=1$ this leads to the differential equation $ \eta_\varphi ^{\prime\prime}=m^2 \eta_\varphi $
which implies that $\eta_\varphi$ is a linear combination of
$\cosh (mt) $ and $\sinh (mt)$. The $K$-invariance leads to $\eta_\varphi$ being
even, so $\eta_\varphi$ is a multiple of $\cosh (mt)$, recovering our previous
result for $\bS^1$ (see Section~\ref{subsec:4.2}). 
We can therefore assume from now on that $n>1$.

The substitution  $s=\sin^2 (t/2)=\frac{1}{2}(1-\cos (t))\in (0,1)$ 
transforms (\ref{eq:RadPart2}) 
into the following differential equation 
for $ y(s)=\eta_\varphi (t)$, $0 < s < 1$ (see \cite[p.484]{Hel84} for
a general statement):
\begin{equation}\label{eq:HypDeEq}
s(1-s)y^{\prime\prime}(s) +\left(\frac{n}{2}-ns\right)y^\prime (s)-m^2 y (s)=0\, .
\end{equation}
This is a special case of the \textit{hypergeometric differential equation}
\[s(1-s)y^{\prime\prime}(s) + (c-(a+b+1)s)y^\prime(s) -ab y(s)=0\, .\] 
This equation has two linearly independent solutions. 
In general, one of them is singular at the origin. The other one,
the Gauss hypergeometric
function, denoted by  $\hgf (a, b; c; x)$, is regular at~$x=0$ and
normalized by  $\hgf (a,b;c;0)=1$

We recall here the definition of $\hgf$ as it will be helpful in the following. For $a \in \C$ and
$k\in \N$ let $(a)_0=1$ and 
$(a)_k :=\prod_{j=0}^{k-1}(a +j)$.
Then, for $a,b,c\in \C$ such that $c\not\in -\N_0$, we have
\begin{equation}\label{def:hgf}
\hgf (a, b ; c ; z)=\sum_{k=0}^\infty \frac{(a)_k(b)_k}{(c)_k} \frac{z^k}{k!} 
\quad \mbox{ for } \quad  |z|<1.
\end{equation}
If $a$ or $b$ is a negative integer then $\hgf (a , b; c ;z)$ is a polynomial.

\begin{rem} \mlabel{rem:x} 
The hypergeometric function $\hgf$ has
an analytic continuation to $\C\setminus [1,\infty)$ 
(see \cite[\S 14.51]{WW96} or
\cite[p. 297]{L73}). Clearly (\ref{def:hgf}) implies that $\hgf (a, b;c ; z )= \hgf (b , a;c ; z )$ and 
\[ \hgf (a , b ; c ; z)>0\quad \mbox{  for } \quad 0\le z<1\quad \mbox{ and } \quad a, b,c>0 \quad \mbox{ or } \quad b=\bar a, c>0.\] 
\end{rem}

To simplify the notation,  let 
\begin{equation}
  \label{eq:lambda}
\rho := \frac{n-1}{2}> 0\quad \mbox{  and } \quad 
\lambda := \lambda_m :=
\begin{cases} 
\sqrt{\rho^2-m^2}  & \text{ for } m^2 \leq \rho^2 \\ 
i\sqrt{m^2-\rho^2}  & \text{ for } m^2 \geq \rho^2. 
\end{cases}
\end{equation}
Note that 
\begin{equation}
  \label{eq:lambda2}
 0 \leq \Re \lambda < \rho \quad \mbox{ and } \quad 
\lambda^2 = \rho^2 - m^2 \quad \mbox{ for } \quad m > 0.
\end{equation}
Then the solution to the differential equation (\ref{eq:HypDeEq}) which
is regular at $s=0$  is, up to a constant,
\begin{equation}
  \label{eq:hypgeo1}
 \hgf (\rho +\lambda , \rho-\lambda; n/2; s) 
= \hgf \big(\rho +\lambda , \rho-\lambda; n/2; \shalf(1- \cos(t))\big).
\end{equation}

 \begin{thm}\label{th:Psi} There exists a constant $\gamma_{n,m}>0$ 
such that the $G^c$-invariant kernel 
 $\Psi_m$ on $\Xi\times \Xi$ from {\rm Theorem~\ref{thm:JRA}} 
is given by
\begin{align*}
\Psi_m(x,y)
&= \gamma_{n,m}\cdot  \hgf (\rho +\lambda , \rho -\lambda ; n/2;\shalf\left(1-\lf{x}{ \sx(y)}\right)) \\
&= \gamma_{n,m}\cdot  \hgf \Big(\frac{n-1}{2} +\lambda , \frac{n-1}{2} -\lambda ; 
\frac{n}{2};\frac{1-\lf{x}{ \sx(y)}}{2}\Big).
\end{align*}
This kernel is positive definite. Furthermore, $\Psi_m$ 
extends to a sesquiholomorphic kernel on the set 
\[ \{(z,w)\in V_\C\times V_\C\: \lf{z}{\sx w}\in \C\setminus (-\infty, -1]\}.\]
\end{thm}

Note that 
\begin{equation}
  \label{eq:hypgeo2}
\psi_m(x) = \Psi_m(x,e_0)
= \gamma_{n,m}\cdot  \hgf \Big(\frac{n-1}{2} +\lambda , \frac{n-1}{2} -\lambda ; 
\frac{n}{2};\shalf(1 - x_0)\Big).
\end{equation}

\begin{proof} We recall first that 
\begin{equation}\label{eq:psi}
\Psi_m(x,y)=\Phi_m(x,\sigma (y))=
\alpha_{\Phi_m}(x\sigma (y))
\end{equation}
for a function $\alpha_{\Phi_m}$ on the interval $[-1,1)$ 
as $\Phi_m$ is $\OO_{n+1}(\R)$-invariant 
(Lemma \ref{lem:alphafun}). We also note that $\sigma (x)=
-\sigma_V(x)$ on~$\bS_+^n$. Thus 
$\Psi_m(x,y)=\alpha_{\Phi_m}(-[x,\sigma_V(y)]_V)$, which is clearly 
$(\fg^c,K)$-invariant.
Thus
$\alpha_{\Psi_m} (s)$ from Lemma \ref{le:PsiLocInv} is given by
$\alpha_{\Phi_m}(- s)$. 
We now apply the above discussion to $\psi_m(x)=\alpha_{\Psi_m}(x_0)$ and note that,
for $u\in e_0^\perp\cap \bS^n$ and $x=\cos (t)e_0+\sin (t)u$, we have
$x_0=\cos (t)$ and 
\[\sin^2(t/2)=\frac{1}{2}\left(1-\cos (t)\right)
=\frac{1}{2}(1-x_0)
=\frac{1}{2}\left(1-[x,\sigma_V(e_0)]\right).\]
By the discussion preceding the theorem, 
there exists a constant $\gamma_{n,m}$ such that
\begin{equation}
  \label{eq:hypgeo3}
\psi_m(x)=\gamma_{n,m}\cdot\, \hgf (\rho+ \lambda, \rho-\lambda ; n/2; 
\shalf(1-x_0)) .
\end{equation}

From Lemma \ref{lem:xi-values} we know that, for $z,w \in \Xi$, we have 
$[z,\sigma_V(w)] \in \C \setminus (-\infty,-1]$, so that 
$\shalf \left(1-[x,\sigma_V(e_0)]\right) \in \C \setminus [1,\infty)$. 
Therefore Remark~\ref{rem:x}, combined with 
Lemma \ref{le:PsiE0}, implies that the right hand side 
of \eqref{eq:hypgeo3} extends uniquely to a kernel in $\Sesh (\Xi)$. Hence so
does $\Psi_m$ and the extension is given by
\begin{equation}
  \label{eq:hypgeo4}
 \Psi_m(x,y)= \gamma_{n,m}\cdot\, \hgf (\rho+ \lambda, \rho-\lambda ; n/2; 
\shalf\left(1-[x,\sigma_V(y)]\right) ).
\end{equation}
Corollary~\ref{cor:2.9} implies that the kernel 
$\Psi_m\res_{\bS^n_+ \times \bS^n_+}$ is positive definite. 
Now \cite[Thm.~A.1]{NO14} implies that the kernel $\Psi_m$ 
is positive definite on $\Xi \times \Xi$. 
As $\Psi_m\res_{\bS^n_+ \times \bS^n_+}$ is non-zero, 
Lemma \ref{le:PsiE0} implies that $\Psi_m(e_0,e_0) > 0$. 
We thus obtain 
\[0<\Psi_{m}(e_0,e_0)=\gamma_{n,m}\cdot\, \hgf (\rho+ \lambda, \rho-\lambda ; n/2; 0 )=\gamma_{n,m} .\] 
This finishes the proof.
\end{proof}

\begin{cor} The $G$-invariant reflection positive kernel $\Phi_m(x,y)$ extends to
a sesquiholomorphic kernel on $\{(x,y)\in V_\C\times V_\C\: x \sigma_E(y)
\in \C\setminus [1, \infty ) \}$ 
given by
 \begin{align}
\Phi_m(x,y)
&= \gamma_{n,m} \cdot 
\hgf (\rho+\lambda ,\rho-\lambda; n/2; \shalf(1+x \sr(y))) 
\label{eq:hygeo} \\ 
&= \gamma_{n,m}\cdot  \hgf \Big(\frac{n-1}{2} +\lambda , \frac{n-1}{2} -\lambda ; 
\frac{n}{2};\frac{1 + \lf{x}{ \sx(y)}}{2}\Big).\notag
 \end{align}
\end{cor}

\begin{proof} Let $M:=\{(x,y)\in V_\C\times V_\C\: x \sr(y)\in \C\setminus 
[1, \infty ) \}$ and observe that $M$ is open and starlike with respect to 
$(0,0)$, hence in particular connected. 
The right hand side defines an element in $\Sesh (M)$. 
By a similar argument as in the proof of Theorem \ref{thm:E0}, we only
have to show that both sides of \eqref{eq:hygeo} 
coincide on $\bS^n_+\times \{e_0\}$. But this follows
from Theorem~\ref{th:Psi} and \eqref{eq:psi} as $-e_0=\sigma (e_0)=-\sigma_V(e_0)$.
\end{proof} 

\begin{rem} \label{rem:4.4}(A simplification for $n$ odd) 
For $n=1$ we have $\rho = 0$, so that  $\lambda = i m$ and
the differential equation (\ref{eq:RadPart2}) becomes
$\eta_\varphi^{\prime\prime} =m^2\eta_\varphi$. Together 
with the fact that $\eta_\varphi$ is even, this leads
to (see also \cite[p.~3]{C03})
\begin{align}
\Psi_m ((\cos t,\sin t), e_0)
&= \gamma_{1,m}\cdot  \hgf \big(im , -im ; \shalf;\shalf(1-\cos t)\big) \nonumber\\
&= \gamma_{1,m} \cosh (m t), \quad\text{for}\quad  0 < t  < \pi.\label{eq:Rem4.4}
\end{align} 

Even if the following discussion can be made uniform for all $\lambda$, 
we separate the cases where $m>\rho$ and $0<m<\rho$. 
Induction and the identity 
$(a)_{n+1}=a(a+1)_n$ lead to  
\begin{equation}\label{eq:shift}
\dfrac{d^k}{dz^k}\ \hgf (a, b; c; z) 
=\frac{(a)_k(b)_k}{(c)_k}\hgf (k+a,k+b;k+c;z).
\end{equation} 
We use this relation starting in dimension $1$ moving up to odd dimension $n=2k +1$. Let
$z=\frac{1}{2}(1-\cos t) = \frac{1}{2}(1 - x_0)$, so that
$\frac{d}{dz} = \frac{2}{\sin t}\frac{d}{dt}$ for $0 < t < \pi$. 
We then get for  $m > \rho$ with $\lambda = i \sqrt{m^2 - \rho^2}$: 
\begin{align}
 \psi_m(x) &= \Psi_m(x, e_0) 
= \gamma_{n,m}\cdot  \hgf \Big(k +\lambda , k-\lambda ; 
k+ \shalf;\shalf(1-x_0)\Big)\nonumber \\
&= \gamma_{n,m}\cdot  \frac{(1/2)_k}{(1+\lambda)_k(1-\lambda)_k} 
\dfrac{d^k}{dz^k}\ \hgf \Big(\lambda ,  -\lambda; \shalf;\shalf(1-x_0)\Big)\nonumber\\
&= \gamma_{n,m}\cdot  \frac{(1/2)_k}{(1+\lambda)_k(1-\lambda)_k} 
\dfrac{d^k}{dz^k}\cosh (mt)\\
&= \frac{\gamma_{n,m} (n-2)(n-4) \cdots 3 \cdot 1}
{2^k \prod_{j=0}^{k-1}(j^2+m^2 - \rho^2)}
\left(\frac{1}{\sin(t)}\dfrac{d}{dt}\right)^k\cosh(mt).\label{eq:PsiOdd}
\end{align} 
For $n=3$, we get the following formula for $\bS^3$: 
\[\wpsi (t)=\frac{\gamma_{3,m}}{2} 
\frac{m}{m^2 - \rho^2} \cdot 
\frac{\sin(mt)}{\sin(t)} \, .\] 
For the case $0<m<\rho$ we have $0<\lambda <\rho$. The arguments are the same as before, the only change is that now $\lambda$ is real so that
we get  
\[\psi_m(x) = \gamma_{n,m} \frac{(n-2)\cdots 1}{2^k\prod_{j=0}^{k-1}(j^2-\lambda ^2)}\left(\frac{1}{\sin(t)}\dfrac{d}{dt}\right)^k\cosh (m t).\] 
This formula is originally due to Takahaski, \cite[p. 326]{T63}. For a more general statement see \cite[Thm.5.1 and Ex. 5.3]{OP04}.
\end{rem}

\subsection{The constant $\gamma_{n,m}$}

In this section we evaluate the constant $\gamma_{n,m}$ explicitly. 
We need the following facts:
\begin{lem}\label{le:415} {\rm(\cite[Prop. 9.1.2]{Fa08})}
Let $f\in L^1(\bS^n)$ be $K$-invariant and 
$\alpha : [-1,1]\to \C$ with $f(x)=\alpha (x_0)$. Then
\[\int_{\bS^n} f(x)d\mu (x) = \frac{\Gamma \left(\frac{n+1}{2}\right)}{\sqrt{\pi}\Gamma\left(\frac{n}{2}\right)}
\, \int_{-1}^1 \alpha (t)(1-t^2)^{\frac{n}{2}-1}dt\]
holds for the $\OO_{n+1}(\R)$-invariant probability measure $\mu$ on $\bS^n$. 
\end{lem}

\begin{lem} \label{le:416} {\rm(\cite[\S 9.2]{L73})} 
Let $a,b,c\in \C$ be such that 
$0 < \Re(b) < \Re(c)$ and $\Re (c - a -b)>0$. Then 
\[\lim_{t\to 1^-}\hgf (a , b;c; t)= \frac{\Gamma (c)\Gamma (c- a- b)}{\Gamma (c-a)\Gamma (c-b)}.\]
\end{lem}

Finally we also have (\cite[Lem. 1]{Lo67}):
 
\begin{lem} \label{le:417} For $c,d \in \C$ with $\Re c> 0$ and $\Re d>0$,   
we have for $0<t<x$:
\[\int_t^x (x-u)^{c-1}(u-t)^{d-1}\hgf \Big(a , b ; c; 1-\frac{u}{x}\Big)\, du= 
\frac{(x-t)^{c+d - 1}\Gamma (c)\Gamma (d)}{\Gamma (c+d)}
\hgf \Big(a , b; c+d ; 1- \frac{t}{x}\Big).\]
\end{lem} 

\begin{lemma} \label{le:418} 
For $n > 1$ and $\rho =(n-1)/2$, we have 
\[\frac{\Gamma \left(\frac{n+1}{2}\right)}{\sqrt{\pi}\Gamma\left(\frac{n}{2}\right)}\gamma_{n,m}\int_{-1}^1 \hgf \left(\rho+\lambda , \rho -\lambda ; \frac{n}{2} ; 
\frac{1}{2}(1-t)\right) (1-t^2)^{\frac{n}{2}-1}dt= 
\int_{\bS^n}\psi_m(x)\ d\mu (x)=\frac{1}{m^2}. \]
\end{lemma}
\begin{proof} We have $C_m 1=(m^2-\Delta)^{-1}1=\frac{1}{m^2}$
and it follows by the invariance of the measure on the sphere that 
\begin{align*} 
\frac{1}{m^2}&=\ip{1}{C_m 1}_{L^2}
=\int_{\bS^n}\int_{\bS^n}\Phi_m(x,y)\, d\mu(x)\, d\mu(y) \\
&=\int_G\int_G \Phi_m(g_1.e_0, g_2.(-e_0))\, dg_1\, dg_2 
=\int_G\int_G \Phi_m(g_2^{-1} g_1.e_0, -e_0)\, dg_1\, dg_2 \\ 
&=\int_G \Phi_m(g.e_0, -e_0)\, dg 
=\int_{\bS^n} \Psi_m(x,e_0)\, d\mu(x) 
=\int_{\bS^n} \psi_m(x)\, d\mu(x). 
\end{align*}
The claim now follows from Theorem \ref{th:Psi} and Lemma~\ref{le:415}.
\end{proof}

We now have the tools to evaluate the constant $\gamma_{n,m}$.  
 
\begin{thm}\label{thm:gamma} Let $m>0$. Then 
  \begin{equation}
    \label{eq:gamma}
\gamma_{n,m}  
=  \frac{  
 \Gamma \left(\frac{n-1}{2}+ \lambda\right)\Gamma \left(\frac{n-1}{2}-\lambda\right)}{\Gamma (n)} 
  \end{equation}
For $\lambda \in i\R$, we obtain in particular 
\begin{equation}
  \label{eq:gamma2}
\gamma_{n,m} = \frac{  | \Gamma \left(\frac{n-1}{2}+ \lambda\right)|^2 }
{\Gamma (n)  }, \quad \mbox{ and } \quad 
\gamma_{1,m}=\frac{\pi}{m\sinh (\pi m)}\quad \mbox{ for } \quad 
n=1.  
\end{equation}
\end{thm}

\begin{proof} We assume first that $n>1$ so that $\Re (\rho \pm \lambda )>0$. Put $\gamma :=\gamma_{n,m} 
 \frac{\Gamma \left(\frac{n+1}{2}\right)}{\sqrt{\pi}\Gamma\left(\frac{n}{2}\right)}$
 and $F(s ) :=\gamma \cdot \hgf(\rho+\lambda ,\rho-\lambda ;\frac{n}{2};s)$.  
By Lemma \ref{le:418} and  the change of variables $u=1+t$ we get
\begin{align*}
\frac{1}{m^2}& =  \int_{-1}^1  F((1-t)/2) (1-t)^{\frac{n}{2}-1}(1+t)^{\frac{n}{2}-1}dt\\ 
&=\int_0^2 (2-u)^{\frac{n}{2}-1}u^{\frac{n}{2}-1} F(1-u/2) du\\
&=\lim_{t\to 0^+} \int_t^2 (2-u)^{\frac{n}{2}-1}(u-t)^{\frac{n}{2}-1} F(1-u/2)du\\
&= 2^{n-1}\gamma \frac{\Gamma(n/2)^2}{\Gamma (n)}\lim_{t\to 0^+} \hgf \left(\frac{n-1}{2} + \lambda , \frac{n-1}{2}- \lambda ; 
n ; 1-\frac{t}{2}\right)\\
&= 2^{n-1}\gamma \frac{\Gamma(n/2)^2}{\Gamma (n)}\lim_{t\to 0^+} 
\hgf \Big(\rho+\lambda, \rho-\lambda; n ; 1-\frac{t}{2}\Big).
\end{align*}

With $a=\rho +\lambda$, $b=\rho -\lambda$, $\rho =(n-1)/2$ and  $c=n$,  we have
$c-a-b=1>0$ and $\Re (c-b) = \frac{n+1}{2}+\Re \lambda >0$ 
(see \eqref{eq:lambda2}). Hence 
Lemma \ref{le:416}  implies 
\begin{align*}
\lim_{t\to 0^+} \hgf \Big(\rho+\lambda, \rho-\lambda ; 
n ; 1-\frac{t}{2} \Big) 
&= \frac{\Gamma (n)}{\Gamma \left(\frac{n+1}{2}+\lambda\right)
\Gamma \left(\frac{n+1}{2}-\lambda\right)}
= \frac{\Gamma (n)}{\Gamma \left(1+\rho+\lambda\right)
\Gamma \left(1+\rho-\lambda\right)} \\
&= \frac{\Gamma (n)}{(\rho^2- \lambda^2)\Gamma \left(\rho+\lambda\right)
\Gamma \left(\rho-\lambda\right)} 
= \frac{\Gamma (n)}{m^2\Gamma \left(\rho+\lambda\right)
\Gamma \left(\rho-\lambda\right)},   
\end{align*}
where we have used that $\Re(\rho \pm \lambda) > 0$. 
This further leads to 
\begin{align*}
1&=
m^2 2^{n-1}\gamma\frac{\Gamma \left(\frac{n}{2}\right)^2}{\Gamma (n)}
 \frac{\Gamma (n)}{m^2\Gamma \left(\rho+\lambda\right)
\Gamma \left(\rho-\lambda\right)}\\
&=2^{n-1}\gamma_{n,m} \frac{ \Gamma \left(\frac{n+1}{2}\right)
\Gamma \left(\frac{n}{2}\right)^2}{\sqrt{\pi}
\Gamma\left(\frac{n}{2}\right)}
 \frac{1}{\Gamma \left(\rho+\lambda\right)
\Gamma \left(\rho-\lambda\right)}\\
&=\frac{2^{n-1}\gamma_{n,m}}{\sqrt{\pi}}
 \frac{ \Gamma \left(\frac{n+1}{2}\right)\Gamma \left(\frac{n}{2}\right)}
{\Gamma \left(\rho+\lambda\right)
\Gamma \left(\rho-\lambda\right)}.
 \end{align*} 
This, together with the identity $2^{2z-1}\Gamma (z)\Gamma (z + 1/2) =\sqrt{\pi }\Gamma (2z)$ (\cite[p. 21]{L73}) proves the theorem for $n>1$. 

For $n=1$ we first note that the left hand side of 
\eqref{eq:gamma2} specializes for $n = 1$  to
\[|\Gamma (i m )|^2 = \frac{\pi}{m \sinh (\pi m)} \quad \mbox{ for } \quad 
m > 0.\]
On the other hand 
we have by  the discussion after (\ref{eq:RadPart2}) and our normalization of the measure on the circle
\[
\frac{1}{m^2}
= \frac{\gamma_{1,m}}{2\pi} \int_{-\pi}^\pi \cosh (mt)dt = \frac{\gamma_{1,m}}
{\pi m} \sinh (\pi m) \quad \mbox{ or } \quad 
\gamma_{1,m} = \frac{\pi }{m \sinh (\pi m)}.\qedhere\]
  \end{proof}
 
\begin{remark} ($m \to 0$) As $\Delta $ annihilates the constants, 
its inverse $\Delta^{-1}$ is not densely defined. However, it is
bounded on the hyperplane $1^\perp$, where $1$ denotes the constant function. 
The formula for $\gamma_{n, m}$ shows that 
\[\lim_{m\to 0} m^2 \gamma_{1,m}= 1\quad 
\mbox{for} \quad n=1. \] 
For $n > 1$ and $m < \rho$ we have for $m \to 0$:
\[ \rho - \lambda 
= \rho\Big(1-(1 - m^2/\rho^2)^{1/2}\Big)
\sim \rho \frac{m^2}{2\rho^2} =  \frac{m^2}{2\rho}
\quad \mbox{ and } \quad 
\rho + \lambda \to 2 \rho.\] 
Therefore 
\[\lim_{m \to 0} m^2 \gamma_{n,m} 
=  \lim_{m \to 0} m^2\frac{\Gamma(\rho + \lambda)\Gamma(\rho-\lambda)}{\Gamma(n)} 
= \frac{\Gamma(2\rho)}{\Gamma(2\rho+1)} \lim_{m \to 0} \frac{m^2}{\rho-\lambda}
=\lim_{m \to 0}  m^2\frac{1}{2\rho} \frac{2\rho}{m^2} =1.\]

We also have by the power series expression of $\hgf$ for $\lambda = \rho$:  
\begin{equation}
  \label{eq:m0case}
\lim_{m\to 0}\hgf \Big(\rho+\lambda , \rho -\lambda; 
\frac{n}{2}; z\Big)= \hgf \Big(n-1, 0; \frac{n}{2}; z\Big) = 1.
\end{equation}
Thus, 
\begin{equation}
  \label{eq:m0case2}
\lim_{m\to 0} m^2\Psi_m = \lim_{m\to 0} m^2\gamma_{n,m}
\cdot \hgf \Big(\rho+\lambda,\rho-\lambda; \frac{n}{2}; \cdot \Big)
=1.
\end{equation}

Let $\cY_q$ be the space of homogeneous degree $q$ harmonic polynomials on 
$\R^{n+1}$, restricted to the sphere $\bS^n$. Then
the canonical action $\delta_q(k)f(x)=f(k^{-1}x)$ defines an irreducible representation of $G$ and
$L^2(\bS^n)\simeq \hat\bigoplus_{q=0}^\infty \cY_q$ is a Hilbert space direct sum 
(\cite[Thm.~9.3.2]{Fa08} or \cite[Ch. 7.3]{vD09}).
For $\eta \in\cY_q$ we have, see \cite[Prop. 9.3.5]{Fa08}:
\[\Delta \eta = -q(q+n-1)\eta.\]
For $m>0$ we therefore obtain with the orthogonal projection 
$p_q : L^2 (\bS^n)\to \cY_q\subeq  L^2(\bS^n)$ the representation 
\[(-\Delta +m^2)^{-1}=\frac{1}{m^2} p_0 
+ \sum_{q=1}^\infty \frac{1}{q(q+n-1) + m^2}\cdot p_q. \]
Thus
\[m^2(-\Delta +m^2)^{-1}=p_0 + \sum_{q=1}^\infty \frac{m^2}{q(q+n-1) + m^2}
\cdot p_q
\stackrel{m\to 0}{\ssarr} p_0 .\]
This fits \eqref{eq:m0case2} because $p_0(\varphi) =\int_G \varphi (u)\, d\mu (u)$.
\end{remark}

\section{Reflection positivity on the sphere and representation theory}\label{se:RefPosRep}
\noindent
In this section we explain how our results from the previous section connect to 
representation theory. In particular we show that \textit{all} 
irreducible unitary spherical representations 
of $G^c:= \OO_{1,n}(\R) ^\uparrow$ 
are obtained by 
reflection positivity. We use \cite{vD09}, in particular Section~7.5, as
a standard reference. 

\subsection{The spherical 
  unitary representations of the Lorentz group}
\mlabel{se:IntRep}
 
Recall that the group $G^c$ acts transitively on $\Lnp 
= G^c.\xi^0 \simeq G^c/MN$, where 
$\xi^0 = e_0 +i e_n$ 
(cf.~Subsection~\ref{se:Boundary}). We embed the sphere $\bS^{n-1}$ into $\bL^n_+$ by $\xi_u = (1,iu)$, so that $\xi_{e_n}$ coincides with the element $\xi^0$ from 
above. Then 
\[\Lnp= \{t\xi_u\: t>0, u\in\bS^{n-1}\} \simeq \bS^{n-1}\times \R_+^\times. \]

\begin{lem}\label{le:jgu} For $g=\begin{pmatrix} a & iv\\ iw^\top & A\end{pmatrix} \in G^c$, $a\in \R, v,w\in\R^n, 
A\in M_{n, n}(\R)$, and $u\in \bS^{n-1}$ define
\[ g.u  =\frac{w + u A^\top}{a-v  u} 
\quad \text{and}\quad j(g,u)=a - vu .\]
Then the following holds for $g,g_1,g_2\in G^c$ and $u\in\bS^{n-1}$:
\begin{itemize}
\item[\rm (i)] $g.u\in \bS^{n-1}$ and $(g,u)\mapsto g.u $ defines an action of $G^c$ on $\bS^{n-1}$.
\item[\rm (ii)] $j(g,u)>0$ and   $j(g_1g_2,u)=j(g_1,g_2.u)j(g_2,u)$.
\item[\rm (iii)] $j(g,u)= \lf{ g\xi_u }{e_0}$.
\item[\rm (iv)] For $k\in K, a_t\in A$ and $\tilde n\in N$, we have 
$j(ka_t \tilde n, e_n)=e^t$.
\item[\rm (v)] $g.(t\xi_u)=t\cdot j(g,u) \xi_{g.u}$.
\end{itemize} 
\end{lem}
\begin{proof}   A direct calculation show that 
\begin{equation}\label{eq:gxiu}
g.(t\xi_u) = tg (\xi_u^\top)^\top = t\begin{pmatrix} a - vu\\ i (Au^\top +w^\top) \end{pmatrix}^\top=
 t \cdot j(g,u)\xi_{g.u}. 
 \end{equation}
 Here we have used that $x_0>0$ for every $(x_0,\bx)\in \Lnp$ so that $j(g, u) >0$. This proves (v) which then implies
 (i) and (ii). Part (iii) follows directly from (\ref{eq:gxiu}) by taking $t=1$, 
and (iv) follows from (iii) as $K$ stabilizes $e_0$,
 $N$ stabilizes $\xi^0 = \xi_{e_n}$, and $a_t\xi^0=e^t\xi^0$. 
  \end{proof} 
 
With $\rho =\frac{n-1}{2}$ as before, 
we define 
\begin{equation}\label{eq:jl}
j_\lambda : G^c\times \bS^{n-1}\to \C^\times, \qquad 
j_\lambda (g , u):= j(g , u)^{-\lambda -\rho},
\end{equation} 
Then (iii) and (iv) above imply that
\begin{equation}\label{eq:Jacob}
j_\lambda (ka\tilde n,e_n)=e^{-(\lambda +\rho)t}
 \quad\text{and}\quad j_\lambda (k,u)=1 \quad \mbox{ for } \quad k\in K,u\in \bS^{n-1}.  
\end{equation}
Further, (ii) leads to the  cocycle relation
\[ j_\lambda (gh, u) = j_\lambda (g, h.u)j_\lambda (h,u) 
\quad \mbox{ for } \quad g,h \in G^c, u \in \bS^{n-1}.\]

\begin{lem}\label{le:GactMea} The following holds:
\begin{itemize}
\item[\rm(i)] If $\varphi \in L^1(\bS^{n-1})$, then
$\int_{\bS^{n-1}}\varphi (g.u)j_\rho (g, u)\, d\mu(u) 
= \int_{\bS^{n-1}}\varphi (u)d\mu (u)$ holds for the 
$\OO_n(\R)$-invariant probability measure $\mu$ on $\bS^{n-1}$. 
\item[\rm(ii)] We obtain a $G^c$-invariant measure on $\Lnp\cong G^c/MN$  by
\[\int_{\Lnp} \varphi (\xi )\, d\xi 
=\int_0^\infty \int_{\bS^{n-1}} \varphi (r\xi_u)r^{n-1}\, d\mu (u) \frac{dr}{r}
\quad \mbox{ for }  \quad  \varphi \in L^1(\Lnp)\, .\]
\end{itemize}
\end{lem}

\begin{proof}
 (i) follows from Lemma \ref{le:jgu}(v) 
and \cite[Prop. 7.5.8]{vD09}, and  (ii) from \cite[p.~114]{vD09}. 
\end{proof}

For $\lambda\in \C$, let $\cH_\lambda$ be the space of 
measurable functions $\phi: \Lnp \to \C$ 
such that $\phi (t\xi) = t^{-\lambda-\rho} \phi(\xi)$ for 
all $t>0$ and $\xi \in \Lnp$, endowed with the Hilbert space structure 
specified by 
\[\|\phi \|^2 := \int_{K} |\phi (k\xi^0 )|^2dk=\int_{\bS^{n-1}}|\phi (\xi_u)|^2d\mu (u) <\infty.\] 
The corresponding scalar product is 
\[\ip{\varphi}{\psi}_{L^2} 
=\int_{\bS^{n-1}}\overline{\varphi (\xi_u)} \psi (\xi_u) d\mu (u)\quad \mbox{ for } \quad \varphi,\psi\in \cH_\lambda.\]
We obtain a representation, not unitary in general, of $G^c$ on $\cH_\lambda$ by 
\[(\pi_\lambda (g)\varphi)(\xi )=\varphi (g^{-1}.\xi )\quad \mbox{ for }  
\quad \varphi \in \cH_\lambda ,g\in G^c\text{ and } \xi \in \Lnp.\]
We then have 
\[(\pi_\lambda (g)\varphi)(\xi_u)
=\varphi (g^{-1}.\xi_u)
= j(g^{-1},u)^{-\lambda-\rho} \varphi (\xi_{g^{-1}.u})
= j_\lambda(g^{-1},u) \varphi (\xi_{g^{-1}.u})\quad \mbox{ for } \quad 
u \in \bS^{n-1}.\]

\begin{lem}\label{le:InvForm} Let $\lambda \in \C$, 
$\varphi\in\cH_\lambda$, $\psi \in\cH_{-\overline{\lambda}}$, and $g\in G^c$.  
Then the following assertions hold:
\begin{itemize}
\item[\rm(i)] $\| \pi_\lambda (g)\varphi\|^2 = \int_{\bS^{n-1}} j(g,u)^{2\Re \lambda}|\varphi (\xi_u)|^2d\mu (u)$.
\item[\rm (ii)] $(\pi_\lambda ,\cH_\lambda)$ is unitary if and only if $\lambda\in i\R$, and then it is irreducible.
\item[\rm (iii)] Let $L$ be a compact subset of $G$. Then there exists a constant $C_L>0$ such that
we have $\|\pi_\lambda (g)\varphi \| \le C_L \|\varphi \|$
for all $g\in L$ and all $\varphi \in \cH_\lambda$.
\item[\rm (iv)]
$\ip{\pi_{-\overline{\lambda}} (g)\psi}
{\pi_{\lambda} (g)\varphi}_{L^2} 
=\ip{\psi}{\varphi}_{L^2}$, so that the $L^2$-inner product 
defines a $\pi_{-\oline\lambda} \times \pi_\lambda$-invariant hermitian 
pairing on $\cH_{-\oline \lambda }\times \cH_\lambda$.
\end{itemize}
\end{lem}

\begin{proof} These are standard facts 
about principal series representations 
of semisimple Lie groups, but we provide 
the main ideas of the proof to stay self-contained and avoid the structure 
theory of semisimple Lie groups (see \cite[\S 7.5]{vD09} for details).

Assume first that $\eta,\gamma \in L^2(\bS^{n-1})$ and $\lambda , \mu \in\C$. The cocylce relation and
the fact that $j(e,u)=1$ shows that $j(g , u)= j(g^{-1},g.u)^{-1}$. Hence 
\begin{align*}
\int_{\bS^{n-1}} \overline{j_\lambda(g  ,u)\eta (g. u)}j_\mu (g,u)\gamma (g.u)\, d\mu (u)
&=      \int_{\bS^{n-1}} j (g,u)^{-\overline{\lambda}-\mu} \overline{ \eta (g. u)} \gamma (g.u)j(g,u)^{-2\rho}\, d\mu (u)\\
&=      \int_{\bS^{n-1}} j (g^{-1},u)^{\overline{\lambda}+\mu} \overline{ \eta (u)} \gamma (u)\, d\mu (u), 
\end{align*}
where we used Lemma \ref{le:GactMea}(i) in the last step. 

All parts of Lemma \ref{le:InvForm}, except the irreducibility assertion 
(\cite[Cor. 7.5.12]{vD09}), follow from this calculation.
\end{proof}

\begin{lem} \mlabel{lem:4.14} 
Let $z\in \Xi$ and $\xi \in \Lnp$.
\begin{itemize}
\item[\rm(i)]  $\Re \lf{z}{ \xi }> 0$ and $\lf{z}{\xi }\in \R$ if $z\in \wH$.
\item[\rm(ii)] $\overline{\lf{z}{\xi }}=\lf{\sx(z)}{\xi}$. 
\item[\rm (iii)] $\xi \mapsto [z,\xi ]^{-\lambda - \rho}$ is in $\cH_\lambda$ for all $z\in \Xi$. 
\item[\rm(iv)] 
$\1_\lambda(\xi)=[e_0,\xi ]^{-\lambda-\rho}$ is $K$-invariant and $\cH_\lambda^K
=\C \1_\lambda$. 
\end{itemize}
\end{lem}

\begin{proof} (i) If $z=(z_0,i\bz)\in \wH$ and $\xi= (w_0,i\mathbf{w})$ then
$\lf{z}{=\xi}=z_0w_0-\bz\bw\in \R$, so $\lf{z}{\xi}$ is real. Now let
$z = u +i v\in \Xi$. Then $u\in V_+$  implies that  $\Re \lf{ z}{ \xi^0}= \lf{u}{\xi^0}= u_0 - u_n>0$. But
then $\Re \lf{z}{g.\xi^0}=\Re \lf{g^{-1}.z}{\xi^0}> 0$ 
for all $g \in G^c$ because $\Xi$ is $G^c$-invariant. Now (i) follows.

For (ii) we note that the maps $\Xi \ni z\mapsto \lf{z}{\xi }, 
\overline{\lf{\sx (z)}{\xi }}\in \C$ are holomorphic and
agree on $\wH$ by (i). Hence they agree on all of $\Xi$ as $\wH$ is totally real in $\Xi$.
(iii) is now obvious and (iv) follows from the definition of $\1_\lambda$ 
and the transitivity of the $K$-action on $\bS^{n-1}$.
\end{proof}

\begin{definition} (a) An irreducible unitary representation $(\pi,\cH)$ of $G^c$ 
is called {\it spherical} if ${\cH^K\not=\{0\}}$. In
that case $\dim \cH^K=1$ and, for every unit vector $u\in \cH^K$, the 
function
$\varphi _\pi (g)= \ip{u}{\pi (g)u}$ is the called the 
corresponding {\it spherical function}. 
It is positive definite and $K$-biinvariant. 

(b) The representations 
$(\pi_\lambda, \cH_\lambda)_{\lambda \in i \R}$ are called 
the \textit{spherical principal series representations}. 
For these parameters, we write $\la \cdot,\cdot\ra_\lambda := \la \cdot,\cdot\ra_{L^2}$ 
for the scalar product on $\cH_\lambda$. 
Note 
that $\pi_{\lambda} \cong \pi_{-\lambda}$ for $\lambda \in i \R$ (\cite[p.~119]{vD09}). 
\end{definition}

We now consider the case $0<\lambda <\rho$. Lemma \ref{le:InvForm}(iv) 
suggests the existence of a $G^c$-intertwining operator 
$A_\lambda : \cH_{\lambda}\to \cH_{-\lambda}$ such that the hermitian form
\begin{equation}
  \label{eq:compsersp}
\ip{\psi}{\varphi }_\lambda :=\int_{\bS^{n-1}} \overline{A_\lambda\psi (\xi_u)}
\varphi (\xi_u)\, d\mu (u) 
= \la A_\lambda\psi, \phi \ra_{L^2}, \qquad 
0 < \lambda < \rho, 
\end{equation}
is non-degenerate and $G^c$-invariant.

\begin{lemma}\label{le:Int} Let $x\in \Lnp$. Then $u\mapsto \lf{x}{\xi_u }^{\lambda-\rho}$ is integrable on $\bS^{n-1}$ if
and only if $\Re \lambda >0$ and in that case
\[\int_{\bS^{n-1}} \lf{x}{\xi_u}^{\lambda - \rho}d\mu (u)= 
\left(\frac{2^{\lambda +\frac{n-3}{2}}}{\sqrt \pi}
\frac{\Gamma \left(\frac{n}{2}\right)\Gamma\left(\lambda\right)}
{\Gamma (\lambda +\rho)}\right)\, \lf{e_0}{x}^{\lambda-\rho}\, .\]
\end{lemma}

For $\lambda = \rho$ and $m = 0$, we have 
\[ \frac{2^{\lambda +\frac{n-3}{2}}}{\sqrt \pi}
\frac{\Gamma \left(\frac{n}{2}\right)\Gamma\left(\lambda\right)}
{\Gamma (\lambda +\rho)} 
=  \frac{2^{n-2}}{\sqrt \pi}
\frac{\Gamma \left(\frac{n}{2}\right)\Gamma\left(\frac{n-1}{2}\right)}
{\Gamma (n-1)} = 1
\]
by the identity $2^{2z-1}\Gamma (z + 1/2)\Gamma (z) =\sqrt{\pi }\Gamma (2z)$ 
(\cite[p. 21]{L73}).

\begin{proof} Write $x=ka_t.\xi_{e_n}$ and $\lambda = s+ir$.  Then 
  \begin{equation}
    \label{eq:dag3}
\lf{x}{\xi_u} =e^t \lf{\xi_{e_n}}{\xi_{k^{-1}.u}}=\lf{e_0}{x} \lf{\xi_{e_n}}{\xi_{k^{-1}.u}}.
  \end{equation}
We can therefore assume
that $x=\xi_{e_n}$. As $u\mapsto \lf{\xi_{e_n}}{\xi_u}$ is $\OO_{n-1}(\R)$-invariant, 
 we get by Lemma~\ref{le:415}, with
$n$ replaced by $n-1$,
\begin{align*}
\int_{\bS^{n-1}} | \lf{\xi_{e_n}}{\xi_u}^{\lambda - \rho}|d\mu (u)
&= \int_{\bS^{n-1}}(1-u_{n})^{s-\rho}d\mu (u)\\
&= \frac{\Gamma\left(\frac{n}{2}\right)}{\sqrt{\pi}\Gamma\left(\frac{n-1}{2}\right)}
\int_{-1}^1(1-t)^{s-\rho}(1-t^2)^{\frac{n-3}{2}}dt\\
&=\frac{\Gamma\left(\frac{n}{2}\right)}{\sqrt{\pi}\Gamma\left(\frac{n-1}{2}\right)}
\int_{-1}^1(1-t)^{s-1}(1+t)^{\frac{n-3}{2}}dt\\
&=\frac{\Gamma\left(\frac{n}{2}\right)}{\sqrt{\pi}\Gamma\left(\frac{n-1}{2}\right)}\, 2^{s+\frac{n-3}{2}}
\int_{0}^1(1-u)^{s-1}u^{\frac{n-3}{2}}du \qquad \text{(substitute } 2u=1+t).
\end{align*}
As $n > 1$, this integral is finite if and only if $s>0$. 
The same calculation as above, and analytic continuation in $\lambda$ further show that 
\[\int_{\bS^{n-1}} \lf{\xi_{e_n}}{\xi_u}^{\lambda - \rho}\, d\mu (u)
=\frac{\Gamma\left(\frac{n}{2}\right)}{\sqrt{\pi}
\Gamma\left(\frac{n-1}{2}\right)}\, 2^{\lambda+\frac{n-3}{2}}
\int_{0}^1(1-u)^{\lambda-1}u^{\frac{n-3}{2}}du 
= \frac{2^{\lambda +\frac{n-3}{2}}}{\sqrt \pi} 
\frac{\Gamma \left(\frac{n}{2}\right)\Gamma\left(\lambda\right)}
{\Gamma (\lambda +\rho)},\]
where we use \cite[\S 12.41]{WW96} for the last equality.
\end{proof} 

\begin{lemma}\label{lem:IEx}
 Let $\mu,\lambda\in \C$, $\varphi \in \cH_\lambda$ and $\emptyset\not= L\subset \Xi$ compact.
 Then there exists a constant $C$ such that for all $z\in L$ and $u\in \bS^{n-1}$ we have
\[| \lf{z}{\xi_u}^{\mu -\rho}\varphi (\xi_u)|\le C|\varphi (\xi_u)| .\]
In particular, $u\mapsto \lf{z}{\xi_u}^{\mu -\rho}\varphi (\xi_u)$ is integrable and
$z\mapsto \int_{\bS^{n-1}}\lf{z}{\xi_u}^{\mu -\rho}\varphi (\xi_u)d\mu (u)$ 
is holomorphic on $\Xi$.
\end{lemma}

\begin{proof} The continuity 
of the function $(z,u) \mapsto \lf{z}{\xi_u}^{\mu -\rho}$ on 
$\Xi \times \bS^{n-1}$ implies that it is bounded on the compact subset 
$L \times \bS^{n-1}$, and this implies the assertion.
\end{proof}

\begin{lem}\label{lem:Alambda} For $\Re \lambda >0$ and $\varphi \in
\cH_\lambda\cap C (\Lnp)$, define 
\[ (A_\lambda \varphi) (x)
:=\frac{\sqrt{\pi}\Gamma (\lambda +\rho)}{2^{\lambda +\frac{n-3}{2}}\Gamma \left(\frac{n}{2}\right)\Gamma\left( \lambda\right)}  
\int_{\bS^{n-1}} [x,\xi_u]^{ \lambda -\rho}\varphi (\xi_u)d\mu (u ) \quad \mbox{ for } \quad 
x \in \bL_n^+.\]
Then 
\begin{itemize}
\item[\rm (i)] $A_\lambda\varphi \in \cH_{- \lambda}$
\item[\rm (ii)] $A_{\lambda } \pi_{\lambda }(g) = \pi_{- \lambda} (g)A_\lambda$ 
on $\cH_\lambda\cap C(\Lnp )$ 
\item[\rm (iii)]   $A_\lambda \1_\lambda=
\1_{- \lambda}$.
\end{itemize}
\end{lem}

\begin{proof}  
By Lemma \ref{le:Int}, the integral defining $A_\lambda \varphi (x)$ exists for 
$x \in \bL_n^+$ and the homogeneity requirement follows directly from 
the definition. The continuity of the function $A_\lambda \varphi$ follows from 
\begin{align*}
\int_{\bS^{n-1}} [k.\xi_{e_n},\xi_u]^{ \lambda -\rho}\varphi (\xi_u)\, d\mu (u )
&= \int_{\bS^{n-1}} [\xi_{e_n},\xi_{k^{-1}.u}]^{ \lambda -\rho}\varphi (\xi_u)\, d\mu (u )\\
&= \int_{\bS^{n-1}} [\xi_{e_n},\xi_u]^{ \lambda -\rho}\varphi (k.\xi_u)\, d\mu (u ) 
\end{align*}
and the fact that the map $k \mapsto \varphi(k.\cdot), K \to C(\bS^{n-1})$ 
is continuous. This shows that $A_\lambda\varphi \in\cH_{-\lambda}$. The proof
 of the intertwining relation (ii) is the same
as the proof Lemma \ref{le:InvForm}(iv). It uses 
Lemma~\ref{le:GactMea}(i), and  that $j  (g^{-1},u) = j(g,g^{-1}.u)^{-1}$. 
For the constant $c > 0$ in the definition of $A_\lambda$ 
in Lemma~\ref{lem:Alambda}, we have 
\begin{align*}
(A_\lambda \pi_\lambda (g)\varphi)  (x)
&=c \int_{\bS^{n-1}} \lf{x}{\xi_u}^{\lambda - \rho}\varphi (g^{-1}.\xi_u)d\mu (u)\\
&=c\int_{\bS^{n-1}} \lf{x}{\xi_u}^{\lambda - \rho}j_\lambda (g^{-1},u)\varphi (\xi_{g^{-1}.u})d\mu (u)\\
&= c\int_{\bS^{n-1}} \lf{x}{\xi_{g.u}}^{\lambda - \rho}j_{-\lambda}(g,u)\varphi (\xi_u)d\mu (u)\\
&=c\int_{\bS^{n-1}} \lf{x}{j(g,u)\xi_{g.u}}^{\lambda - \rho} \varphi (\xi_u)d\mu (u)\\
&=c\int_{\bS^{n-1}} \lf{x}{g.\xi_u}^{\lambda - \rho}\varphi (\xi_u)d\mu (u)\\
&=c\int_{\bS^{n-1}} \lf{g^{-1}.x}{\xi_u}^{\lambda - \rho}\varphi (\xi_u)d\mu (u)\\
&=\pi_{-\lambda }(g)(A_\lambda  \varphi) (x).
\end{align*}

(iii) follows from
$\lf{x}{\xi_u}^{\lambda -\rho }= 
\1_{ \lambda}(x)\lf{\xi_{e_n}}{\xi_{k^{-1}.u}}^{\lambda -\rho}$ 
(see \eqref{eq:dag3}).
\end{proof}  

\begin{defn}
With $m>0$ and $\lambda = \lambda_m$ as before,  we write 
$(\pi_m,\cH_m)$ and $\ip{\cdot}{\cdot}_m$, 
instead of using $\lambda$ as an index. Here 
$(\pi_m, \cH_m) = (\pi_{\lambda_m}, \cH_{\lambda_m})$ for $m \geq \rho$ 
and $\lambda_m \in i \R$, and for 
$m < \rho$, $\cH_m$ is the Hilbert space obtained from the scalar product 
\eqref{eq:compsersp} for $\lambda = \lambda_m$. \\
We write $(\pi_0,\cH_0)$ for the trivial representation of $G^c$. 
\end{defn}

\begin{theorem}\label{thm:irrSpRep} 
The representations $(\pi_m, \cH_m)_{m \geq 0}$  
are irreducible unitary
spherical representations of $G^c$, and any 
irreducible unitary spherical representations is equivalent 
to exactly one of these. 
The corresponding spherical functions are given for 
$\lambda = \lambda_m$, $m \geq 0$, by
\begin{align} \label{eq:spfunc}
\varphi_m (x) & = \int_{\bS^{n-1}} [x,\xi_u]^{-\lambda - \rho}\, d\mu (u)
= \hgf (\rho +\lambda ,\rho -\lambda; 
\textstyle{\frac{n}{2}}; -\sinh^2(t/2)) \\
&= \frac{(n-1)!}{\Gamma\left(\frac{n-1}{2}+\lambda\right)\Gamma\left(\frac{n-1}{2}-\lambda\right)}\Psi_m(x,e_0)  ,\qquad x=ka_t.e_0\in \wH .\notag
\end{align}
\end{theorem}

\begin{proof} The irreducibility of the representation $(\pi_m,\cH_m)$ 
is \cite[Cor. 7.5.12]{vD09}. That those
are all the irreducible unitary spherical representations is 
\cite[p. 119]{vD09} and \cite[Thm.~7.5.9]{vD09}. 
The statements about the spherical functions
can be found in \cite[p. 111 and  p. 126]{vD09}.  To translate 
between our setting and \cite{vD09}, we recall from 
\cite[p.~298]{WW96} the relation 
\[ \hgf(2a,2b,a+b+ \shalf, z) = \hgf(a,b,a+b+ \shalf, 4z(1-z)) \] 
for the hypergeometric functions. 
For $a = \frac{\rho + \lambda}{2}$ and $b = \frac{\rho - \lambda}{2}$ 
with $a + b + \frac{1}{2} = \frac{n}{2}$, it leads to 
\[   \hgf \big(\rho +\lambda ,\rho -\lambda; 
{\textstyle \frac{n}{2}}; -\sinh^2(t/2)\big)
=\hgf \Big(\frac{\rho +\lambda}{2} ,\frac{\rho -\lambda}{2}; 
{\frac{n}{2}}; -\sinh^2(t)\Big)  \] 
because $z = -\sinh^2(t/2)$ implies 
\[ 4z(1-z) = - 4 \sinh^2(t/2)\cosh^2(t/2) = - \sinh^2(t).\] 
The last equality follows from Theorem~\ref{th:Psi} and the fact that
\[\frac{1}{2}\left(1-\lf{ka_t.e_0}{e_0}\right)=\frac{1}{2}\left(1-\cosh (t)\right)=-\sinh^2(t/2).\qedhere \]
\end{proof} 
 
From now on we write 
\[ \Phi_m^c(z,w) := \frac{\Psi_m(z,w)}{\Psi_m(e_0,e_0)}, \quad m > 0, \] 
for the normalization of the kernel $\Psi_m$, so that 
$\phi_m(x)=\Phi^c_m(x,e_0)$ is the spherical function on 
$\bH_V^n$ corresponding to the spherical representation~$(\pi_m,\cH_m)$. 
For $m = 0$ we put $\Phi_0^c = 1$. For $m\geq 0$, we write $\cO_m(\Xi) := \cH_{\Phi_m^c} \subeq \cO(\Xi)$ 
for the corresponding reproducing kernel Hilbert space. Then 
left translation
$(\rho_m(g)F)(z)  := F(g^{-1}.z)$ 
defines a unitary representation $(\rho_m, \cH_m)$ of $G^c$. 
By Theorem~\ref{thm:irrSpRep}, 
the representation $(\pi_m,\cO_m(\Xi))$ is irreducible and 
isomorphic to $(\pi_m, \cH_m)$.

\begin{corollary}\label{cor:IntRep} Let the notation be as above. Then the following holds:
\begin{itemize}
\item[\rm (i)] The representations $(\rho_m,\cO_m(\Xi))$, $m\geq 0$
are unitary and irreducible and every irreducible spherical
representation of $G^c$ is unitarily equivalent to $(\pi_m,\cO_m(\Xi))$ for some 
$m\geq 0$. 
\item[\rm (ii)] Every irreducible unitary spherical representation of $G^c$ can be constructed via reflection positivity.
\item[\rm (iii)] $\Gamma_e=\{\phi_m\: m\geq 0\}$.
\item[\rm (iv)] For $\Psi\in \Gamma$, there exists a unique positive Radon 
measure $\mu_\Psi$ on $[0,\infty)$ such that
\[\Psi (z,w)=\int_{[0,\infty)} \Phi_m^c (z,w)\, d\mu_\Psi (m)
\quad \mbox{ for }\quad  z,w\in\Xi.\] 
The integral converges uniformly on compact subsets of $\Xi\times \Xi$, resp., 
as a vector-valued integral in the Fr\'echet space $\Sesh(\Xi)$.
\end{itemize}
\end{corollary}

\begin{proof} Part (iv) follows from (iii) and 
\cite[Thm.5.1]{KS05}, see also \cite[Thm.1]{FT99}. 
As pointed out above, everything else follows from Theorem~\ref{thm:irrSpRep}. 
 \end{proof}

\subsection{An integral representation of  $\phi_m$}
\mlabel{subsec:5.2}

In this subsection we  obtain an integral representation of 
the kernel $\Phi_m^c$ and to use the Poisson transform to construct a concrete 
unitary intertwining operator $\cP_m : \cH_m\to \cO_m (\Xi)$. 

For $\lambda \in \C$, we consider the {\it Poisson kernel} 
\[ P_\lambda \: \Xi \times \Lnp \to \C, \quad 
P_\lambda (z,\xi)=\lf{z}{\xi}^{-\lambda -\rho}.\] 
This kernel is defined by Lemma \ref{lem:4.14}, the functions 
$P_\lambda(\cdot, \xi)$ are holomorphic on $\Xi$, 
$P_{\lambda,z} := P_\lambda(z,\cdot) \in \cH_{\lambda}$, 
and $P_{\lambda, e_0} =\1_\lambda\in\cH_{\lambda}^K$. We define the 
{\it Poisson transform} by 
\[(\cP_\lambda \varphi) (z) 
= \int_{\bS^{n-1}} P_\lambda (z,\xi_u)\varphi (\xi_u)d\mu (u)
\quad \mbox{ for }  \quad \varphi \in \cH_\lambda .\]
The existence of the integral follows from Lemma \ref{lem:IEx}, 
and $\cP_\lambda \phi$ is holomorphic on $\Xi$. 
For $\lambda =\lambda_m$, $m \geq 0$,
we will also use the notation $P_m:=P_{\lambda_m}$ and $\cP_m=\cP_{\lambda_m}$. 

\begin{lemma}\label{lem:APl}
If $\lambda \in \C$ with $\Re \lambda>0$, then $A_\lambda P_{\lambda,z} 
= P_{-\lambda,z}$ for $z \in \Xi$. 
 \end{lemma}

 \begin{proof}
 By Lemma \ref{lem:Alambda}(iii), we
have $A_\lambda P_{\lambda, e_0}= A_\lambda \1_\lambda = \1_{-\lambda}  =P_{-\lambda,e_0}$.
For $g\in G^c$, we thus obtain 
\begin{align*}
A_\lambda P_{\lambda,g.e_0}(\xi)&=A_\lambda (\pi_\lambda (g)P_{\lambda, e_0}) (\xi )  
= \pi_{-\lambda}(g)A_\lambda P_{\lambda, e_0}(\xi ) 
=P_{-\lambda,e_0}(g^{-1}.\xi)= P_{-\lambda, g.e_0}(\xi).
\end{align*}
Hence the maps $w\mapsto A_\lambda P_{\lambda, w}(\xi)$ 
and $w\mapsto P_{-\lambda, w}(\xi)$ which are both holomorphic on $\Xi$, 
coincide on~$\wH$. This implies that 
$A_\lambda P_{\lambda, w}=P_{-\lambda, w}$ for all $w\in\Xi$. 
\end{proof}

\begin{theorem}\label{thm:Pm} Let $m>0$ and $\varphi \in\cH_{\lambda_m}$.  Then
$\cP_m (\varphi)\in \cO_m(\Xi)$ and 
$\cP_m : \cH_{\lambda_m}\to \cO_m(\Xi)$ is unitary and $G^c$-equivariant. 
We further have 
\[
\Phi_m^c(z,w ) =\ip{P_{m, w}}{P_{m,z}}_{\lambda_m}=
\int_{\bS^{n-1}} \lf{\sx (w)}{\xi_u}^{\lambda_m -\rho}\lf{z}{\xi_u}^{-\lambda_m - \rho}d\mu (u)
\quad\mbox{ for }\quad z,w \in \Xi.
\]
\end{theorem}

\begin{proof} It follows from Lemma \ref{lem:IEx} that all the integrals 
in question exist and that $\cP_m(\varphi )\in \cO(\Xi )$. The
same argument shows that the kernels 
\begin{align*}
(z,w) &\mapsto \ip{P_{m, w}}{P_{m,z}}_\lambda, \\ 
(z,w)&\mapsto \int_{\bS^{n-1}} \lf{\sx (w)}{\xi_u}^{\lambda -\rho}\lf{z}{\xi_u}^{-\lambda - \rho}d\mu (u) 
\end{align*}
are sesquiholomorphic. 
We now show that they coincide. For $m\ge \rho$ we have $\lambda\in i\R$ and
\begin{align*}
\ip{P_{m, w}}{P_{m,z}}_\lambda&=\int_{\bS^{n-1}} \overline{\lf{w}{\xi_u}^{-\lambda -\rho}}\lf{z}{\xi_u}^{-\lambda-\rho}d\mu (u)
=\int_{\bS^{n-1}} \lf{\sx w}{\xi_u}^{\lambda -\rho}\lf{z}{\xi_u}^{-\lambda-\rho}\, d\mu (u).
\end{align*} \\
For $0<m<\rho$ we have
\begin{align*}
\ip{P_{m, w}}{P_{m,z}}_\lambda&=\int_{\bS^{n-1}} \overline{(A_\lambda P_m)(w,\xi_u)} \lf{z}{\xi_u}^{-\lambda-\rho}d\mu (u)
=\int_{\bS^{n-1}}\overline{\lf{w}{\xi_u}^{\lambda-\rho} }\lf{z}{\xi_u}^{-\lambda-\rho}d\mu (u)\\
&=\int_{\bS^{n-1}}\lf{\sx w}{\xi_u}^{\lambda-\rho}\lf{z}{\xi_u}^{-\lambda-\rho}d\mu (u)
\end{align*}
by Lemmas \ref{lem:4.14} and \ref{lem:APl} because $\lambda$ is real.
This proves the asserted equality.

For $\lambda\in\C$ and  $\varphi\in\cH_\lambda$, we find with 
Lemma \ref{le:InvForm} 
\begin{align} \label{eq:dagx}
\cP_\lambda(\phi)(z) =& 
\int_{\bS^{n-1}} \lf{z}{\xi_u}^{\lambda -\rho}\varphi (g^{-1}.\xi_u )\, d\mu (u) 
= 
\int_{\bS^{n-1}}\overline{\lf{\sx z}{\xi_u}^{\oline\lambda -\rho}}\varphi (g^{-1}.\xi_u)\, 
d\mu (u) \notag \\
=&\int_{\bS^{n-1}}\overline{\lf{\sx z}{g.\xi_u}^{\oline\lambda -\rho}}\varphi (\xi_u)\, 
d\mu (u)
=\int_{\bS^{n-1}} \lf{g^{-1}.z}{\xi_u}^{\lambda -\rho}\varphi (\xi_u)\, d\mu (u) .
\end{align}
Hence $\cP_\lambda$ is an intertwining operator.

We now consider the sesquiholomorphic kernel 
\[ \Lambda_m (z,w)=\ip{P_{m, w}}{P_{m, z}}_\lambda=
\int_{\bS^{n-1}} 
\lf{\sx w}{\xi_u}^{\lambda -\rho}\lf{z}{\xi_u}^{-\lambda-\rho}d\mu (u).\]
This kernel is hermitian because 
\[\overline{\Lambda_m(w,z)}=\overline{\ip{P_{m,z}}{P_{m,w}}_\lambda}=\ip{P_{m,w}}{P_{m,z}}_\lambda =\Lambda_m (z,w).\]
By taking $\varphi =P_{m,z}$ and replacing $z$ in \eqref{eq:dagx}
by $\sx w$, it follows that $\Lambda_m$ is $G^c$-invariant.

Thus $\Lambda_m\in\Sesh(\Xi )$ and it is positive definite.
By Corollary \ref{co:UniDet} the kernel $\Lambda_m$ is determined
by the function $\Lambda_{m, e_0}|_{\wH}$. But
\[\Lambda_m (z,e_0)=\oline{\Lambda_m (e_0,z)}
=\int_{\bS^{n-1}} [z,e_0]^{-\lambda -\rho}d\mu (u)=\varphi_m(z)\]
by Theorem~\ref{thm:irrSpRep}. Hence $\Lambda_m (z,w)=\Phi_m^c(z,w)$.

As $(\pi_m, \cH_m)$ is irreducible and $\phi_m \in \cO_m(\Xi)$ is 
cyclic, $\cP_m$ is a $G^c$-isomorphism. We also have $\cP_m\1_\lambda = \varphi_m$ and
\[\ip{\cP_m\1_\lambda}{\cP_m\1_\lambda}=
\ip{\varphi_m}{\varphi_m}=\Phi_m^c(e_0,e_0)=1.\]
As $\ip{\1_\lambda}{\1_\lambda}_\lambda=1$, it follows that $\cP_m$ is unitary.
\end{proof}

\subsection{Canonical kernels} 
\mlabel{subsec:5.3}

In this section we discuss the relation between our setting and canonical 
kernels on hyperboloids. To this end, we identify (only in this section) 
$V$ with $\R^{n+1}$ and use the standard notation for $\bH^n$ etc. The conjugation
$\sx$ is then the complex conjugation $z\mapsto \oline z$ and $G^c\subset \mathrm{GL}_n(\R)$ is
the standard realization of $\OO_{1,n}^\uparrow(\R)$.  

The  {\it Berezin kernel} on the 
open unit ball $\bB^n := \{ x \in \R^n \: x^2 <1\}$ is defined by 
\[ B_\lambda(\bx,\by) = \Big(\frac{(1- \bx^2)(1 - \by^2)}{(1-\bx\by)^2} 
\Big)^\lambda \quad \mbox{ for }  \quad 
\lambda > 0, \quad \bx^2, \by^2 < 1\] 
(see \cite[\S 2]{vDH97}). 
We consider the diffeomorphism 
\[ \Gamma \: \bH^n \to \bB^n, \quad 
\Gamma(x_0,\bx) := x_0^{-1} \bx.\] 
Then 
$\Gamma^{-1}(\bx) = (1 - \bx^2)^{-1/2} (1,\bx)$ 
leads to 
\[ C_\lambda(x,y) := B_\lambda(\Gamma(x), \Gamma(y)) 
= \Big(\frac{(x_0^2- \bx^2)(y_0^2 - \by^2)}{(x_0 y_0-\bx\by)^2}\Big)^\lambda 
= [x,y]^{-2\lambda}\] 
(cf.\ \cite[\S 3]{vDH97}). 
This is a $G^c$-invariant kernel on $\bH^n$. 
It extends to a sesquiholomorphic kernel on a neighborhood 
of $\bH^n_V$ in $\Xi$ by the formula
\[  C_\lambda(z,w) :=  [z, \oline w]^{-2\lambda} .\] 
This kernel is defined on the $G^c$-invariant open subset 
$\{(z,w)\in \Xi\times \Xi\: \Re [z, \oline w]>0\}$.  
As 
$\{[z,w]\: z,w\in \Xi\}= \C \setminus (-\infty, -1]$ 
(Lemma \ref{lem:xi-values}), $C_\lambda$ can not be extended to all of $\Xi$.  
To cope with this situation, we shrink $\Xi$ to a suitable 
$G^c$-invariant domain to which $C_\lambda$ extends: 

\begin{prop} Consider the $G^c$-invariant open submanifold 
\[ \Xi' := \{ z \in \Xi \: \beta(z) > 0\} \quad \mbox{ for } \quad 
\beta(z) := [z,\oline z ] = |z_0|^2 - \|\bz\|^2.\] 
Then $C_\lambda$ defines a $G^c$-invariant sesquiholomorphic kernel on $\Xi'$. 
\end{prop}

\begin{proof}   
As the continuous function $\beta$ on $V_\C$ is $G^c$-invariant, 
$\Xi'$ is an open $G^c$-invariant submanifold of $\Xi$. 
Since $\beta(z) =1$ for $z \in \bH^n_V$, it contains $\bH^n_V$. 
With $z_t=\cos (t)e_0 + \sin(t)ie_n\in \bS^n_+$ we have 
\[ \Xi 
= G^c.\{ z_t \: |t| < \pi/2\} \supeq \Xi' = G^c.\{ z_t \: |t| < \pi/4\}.\] 

For $\bz, \bw \in \C^{n}\simeq V_\C$, we write 
$\la \bz,\bw \ra = \bz\oline\bw = \sum_{j=1}^n z_j \oline{w_j}$. 
Then $\beta(z) = |z_0|^2 - \la \bz,\bz\ra  > 0$ implies 
$z_0 \not=0$ and $\tilde \bz := z_0^{-1} \bz$ satisfies $\|\tilde \bz\| < 1$. 
For $\beta(z), \beta(w) > 0$, 
we thus  obtain $|\la \tilde \bz, \tilde \bw \ra| < 1$. 
For $z = (z_0,\bz),w =(w_0,\bw)\in \Xi'$, this leads to 
\[ [z,\oline w]
= z_0 \oline{w_0} - \la \bz, \bw\ra 
= z_0 \oline{w_0}\underbrace{\big(1  - \la \tilde\bz, \tilde \bw \ra\big)}_{\Re > 0}.\] 
As $\Re z_0 > 0$ and $\Re w_0 > 0$ follows from $\Xi' \subeq \Xi \subeq 
T_{V_+}$ (Proposition~\ref{prop:XiTube}), we see that $C_\lambda$ defines a $G^c$-invariant sesquiholomorphic 
kernel on $\Xi'$. 
\end{proof}

On $\bH^n$, the corresponding $K$-invariant function is given by 
\[ \psi_\lambda(x) = C_\lambda(x,e_0) = x_0^{-2\lambda},
\quad \mbox{ resp., } \quad 
 \psi_\lambda(\cosh(t) e_0 +  \sinh(t) e_n)  = \cosh(t)^{-2\lambda}. \] 
By \cite[Thm.~1]{vDH97}, the kernel $C_\lambda \in \Sesh(\Xi')$ 
has an integral representation 
\[ C_\lambda = \int_0^\infty \Psi_m\, d\mu_\lambda(m), \] 
where $\mu_\lambda$ is a measure on $(0,\infty)$, which is 
on the interval $[\rho,\infty)$ ($\rho = \frac{n-1}{2}$) 
equivalent to Lebesgue measure. 
For $\lambda < \frac{\rho}{2}$, the measure $\mu_\lambda$ has an 
additional singular part, given by point measures in the points 
\[ s_j(\lambda) := \rho - 2 \lambda - 2 j \quad \mbox{ for }\quad 
j \in \N_0 \quad \mbox{ with } \quad 
s_j(\lambda) > 0.\] 

For the corresponding unitary representation of $G^c$ 
on the reproducing kernel Hilbert space $\cH_{C_\lambda} \subeq 
\cO(\Xi')$, this means that it decomposes into a direct integral 
of all spherical principal series representations 
(corresponding to $\lambda > \rho$) and a direct sum of finitely 
many spherical complementary series representations, 
corresponding to the values $s_j(\lambda)$. 
We refer to \cite{vDH97} and \cite[Prop.~2.7.3, Cor.~4.2.2]{Hi99} 
for more details, 
where these kernels are considered as real 
analytic kernels on $\bH^n$, resp., $\bB^n$. 
These results extend to restrictions of minimal holomorphic 
representations of $\SU_{n,m}(\C)$ to $\SO_{n,m}(\R)$ (\cite{Se07})  
and, more generally, to matrix balls (\cite{Ner99}) and Makarevich spaces
(\cite{FP05}). 

\section{Perspectives} \label{se:perspectives}

\subsection{Identification of $\Xi$ with a Lie ball}
\mlabel{sec:lieball}

In this subsection we explain how to identify the crown domain 
$\Xi$ with the $n$-dimensional Lie ball, 
i.e., the bounded symmetric domain whose isometry group is 
locally isomorphic to $\SO_{2,n}(\R)$. 

On the $(n+1)$-dimensional tube domain $T_{V_+} = V_+ + i V$, we have a 
natural transitive action of the group $\OO_{2,n+1}(\R)^\uparrow$ 
which is obtained by extending the action of this group on 
the Minkowski space $iV \cong \R^{1,n}$ by rational maps 
to an action by holomorphic automorphisms of the 
tube domain $T_{V_+}$ whose Shilov boundary is $iV$ 
(\cite[\S X.5]{FK94}).
The identity $\Xi = \bS^n_\C \cap T_{V_+}$ 
identifies $\Xi$ with a hypersurface defined 
by the equation  $z^2=1$ 
in the tube domain $T_{V_+}$. 

The function $\Delta(z) := z ^2$ on $V_\C \cong \C^{n}$ 
can be interpreted as the determinant of the complex Jordan algebra 
\[  V_\C \cong \C \oplus \C^{n-1}, 
\qquad (t,\bz)(t',\bz') := (tt' - \bz   \bz', 
t\bz' + t'\bz),\quad t,t' \in \C, 
\bz,\bz'\in \C^{n-1} \] 
whose determinant function is given on $V = \R \oplus i \R^n$ by 
\begin{equation}
  \label{eq:delta-jor}
 \Delta(z_0, i \bz) = z_0^2 - z_1^2 - \cdots - z_{n-1}^2  
= (\iota z)^2  \quad \mbox{ for } \quad z = (z_0, \bz) \in \R^{n+1}, 
\iota(z_0, \bz) = (z_0, i \bz) 
\end{equation}
(\cite[p.~31]{FK94}). 
On $V_\C \cong \C^{n+1}$ we have an involutive rational map 
\[ r(z) := \Delta(z)^{-1} z = \frac{1}{z^2} z
\quad \mbox{ with } \quad (\C^{n+1} \setminus \Delta^{-1}(0))^r = \bS^n_\C, \] 
called {\it ray inversion}, defined in the complement 
of the hypersurface $\Delta = 0$. If $\Delta(z) = 1$, then 
$r$ maps $\C^\times z$ into itself and $\lambda z$ to $\lambda^{-1} z$ 
for $\lambda \in \C^\times$. Next we observe that 
\[ r(z) = \alpha(z^{-1}), \quad \mbox{ where } \quad 
\alpha(z_0, \bz) = (z_0, - \bz) \quad \mbox{ and } \quad (z_0,\bz)^{-1} = 
\Delta(z_0,\bz)^{-1}(z_0,-\bz) \] 
is Jordan inversion. 
Since $T_{V_+}$ is invariant under Jordan inversion (\cite[Thm.~X.1.1]{FK94}) 
and $\alpha$, 
it is also invariant under the holomorphic involution~$r$ and we thus obtain 
\begin{equation}
  \label{eq:xir}
(T_{V_+})^r = T_{V_+} \cap \bS^n_\C = \Xi.
\end{equation}

The Cayley transform $C(z) := (z-e)(z+ e)^{-1}$ maps the 
tube domain $T_{V_+}$ biholomorphically onto the Lie ball 
\[ \cD := \{ u + i v \in V_\C = \C^{n+1} \: 
\|u\|^2 + \|v\|^2 + 2 \sqrt{\|u\|^2\|v\|^2 - (u,v)} < 1 \},\] 
where we identify $V$ with the euclidean space $\R^{n+1}$, 
so that $\|\cdot\|$ and $(u,v) = \sum_j u_jv_j$ refers to the euclidean 
scalar product 
(\cite[\S X.2]{FK94}). From $C(r(z)) = - \alpha(z) = (-z_0,\bz)$ 
it follows that 
\[ C(\Xi) = \cD^n := \cD \cap (\{0\} \times \C^n),\] 
which is an $n$-dimensional Lie ball. 
Since we also have $C \circ \sigma_V = \sigma_V \circ C$, it follows 
that 
\[ C(\bH^n_V) = \cD^n_\R := \cD^n \cap V = \{ (0,i\bx) \: \bx^2 < 1 \} \] 
is an open unit ball in an $n$-dimensional euclidean space. 

Writing $T_{V_+}$ as $G_1/K_1$ for 
$G_1 \cong \SO_{2,n+1}(\R)_0$ and the stabilizer 
$K_1 \cong \SO_2(\R) \times \SO_{2n+1}(\R)$ 
of the base point $e_0$, we obtain from $r(e_0) = e_0$ 
an involution $\tau_r$ on $G_1$ defined by 
$r(g.z) = \tau_r(g)r(z)$ for $z \in T_{V_+}$. 
For the subgroups $G_1^r \subeq G_1$ and $K_1^r \subeq K_1$, we then have 
\[ \Xi = (G_1.e_0)^r = G_1^r.e_0 \cong G_1^r/K_1^r.\] 
In particular, $\Xi$ is a Riemannian homogeneous space of the group 
$G_1^r$. 

We now determine the groups $G_1$ and $G_1^r$ more explicitly. 
On 
\[ \tilde V := i V \oplus \R^2 = (i \R \oplus \R^n) \oplus \R^2, \] 
we consider the symmetric bilinear form given by 
\[ \beta((v,s,t), (v',s',t')) := v   v' - s^2 + t^2 \] 
and the projective quadric 
\[ Q := \{ [\tilde v] \in \bP(\tilde V) \: \tilde v \in \tilde V, 
\beta(\tilde v,\tilde v)= 0 \}.\] 
The map 
\[ \eta \: iV \to Q, \quad 
\eta(v) := \big[(v, \shalf(1 + v   v),\shalf(1 - v   v)\big] \] 
is the {\it conformal completion of $iV$}. It 
 is a diffeomorphism onto an open dense subset of $Q$. 
The natural action of the orthogonal 
group $\OO(\tilde V,\beta)  \cong \OO_{2,n+1}(\R)$ on $Q$ corresponds to 
the action of the conformal group on the Shilov boundary $iV$ of $T_{V_+}$. 

Next we observe that 
\[ \eta(r(v)) 
= \Big[ \big((v  v)^{-1} v, 
\shalf(1 + (v   v)^{-1}),\shalf(1 - (v   v)^{-1} \big)\Big] 
= \Big[( v, \shalf(1 + v   v),\shalf((v   v)-1))\Big],\] 
so that 
\[ \eta(r(v)) = \tilde r \eta(v) 
\quad \mbox{ for } \quad 
\tilde r(i x_0, \bx, s,t) =  (i x_0, \bx, s,-t).\] 
Therefore the involution on $\OO(\tilde V,\beta)$ corresponding to
$\tau_r$ corresponds to $\tau_{\tilde r}(g) := \tilde r g \tilde r$, 
and thus 
\[ \OO(\tilde V, \beta)^{\tilde r} \cong \OO_{2,n}(\R) \times \OO_2(\R) \] 
because the form $\beta$ on the subspace $\tilde V^{\tilde r}$ 
has signature $(n,2)$. This shows that 
\begin{equation}
  \label{eq:lieball}
\Xi \cong \SO_{2,n}(\R)_0/(\SO_2(\R) \times \SO_n(\R)) 
\end{equation}
is the Riemannian symmetric space associated to $\SO_{2,n}(\R)_0$. 
In particular, the action of $G^c 
\cong \OO_{1,n}(\R)^\uparrow$ on $\Xi$ extends to a transitive 
action of the group $\SO_{2,n}(\R)_0$. 

\subsubsection*{Connection to highest weight representations and corresponding kernels} 

On the tube domain $T_{V_+}$, there  exists a 
natural family of sesquiholomorphic kernels, given in terms 
of the Jordan determinant 
$\Delta(z) = \lf{z}{z}$ (see \eqref{eq:delta-jor}) 
by $\Delta\left(\frac{z + \sx w}{2}\right)^{-\nu}$. 
Concretely, we have  
\[ \left[\frac{z + \sx (w)}{2}, \frac{z + \sx (w)}{2}
\right]^{-\nu}_V \] 
(\cite[\S XIII.1]{FK94}). 
These sesequiholomorphic kernels are obviously invariant 
under $G^c$, hence restrict to $G^c$-invariant kernels 
on the complex submanifold $\Xi \subeq T_{V_+}$ 
given by 
\[ Q_\nu(z,w) := \left( \frac{1 + [z, \sx (w)]}{2}\right)^{-\nu}. \] 
On $\H^n_V$, these kernels correspond to the $K$-invariant function 
\[ q_\nu(x) = Q_\nu(x,e_0) = 
\Big( \frac{1 + x_0}{2}\Big)^{-\nu},\]
respectively  
\[q_\nu(\cosh(t) e_0 + \sinh(t) ie_n) = 
\left( \frac{1 + \cosh(t)}{2}\right)^{-\nu}=\cosh(t/2)^{-2\nu}.\] 
The kernel $Q_\nu$ corresponds to the reproducing kernel of a highest weight representation of $\SO_{2,n}(\R)_0$.
Its restriction to the real symmetric bounded domain corresponding to the subgroup 
 $\SO_{1,n}(\R)_0$ has been studied for instance in \cite{Hi99,OO96,O00}. 

Since $\Xi$ is biholomorphic to an $n$-dimensional Lie ball, 
hence also to an $n$-dimensional tube domain, it follows 
from \cite[Thm.~XIII.2.7]{FK94} that the kernel 
$Q_\nu$ is positive definite if and only if either 
$\nu = 0$ or $\nu \geq \frac{d}{2} = \frac{n-2}{2}$ 
(note that $r = 2$ in our case). 
We thus obtain elements $Q_\nu \in \Sesh(\Xi)$ 
for $\nu \geq \frac{n-2}{2}$ and, in view of Corollary \ref{cor:IntRep},  
it is a natural problem to determine the measure $\mu_\nu = \mu_{Q_\nu}$. 

This problem has been solved by  H.~Sepp\"anen in \cite{Se07b},   
see also \cite{OO96,O00} for part (i). This case and a different connection to reflection 
 positivity has been discussed
in \cite{NO14,JO98,JO00}. By \cite[\S 5.3]{Se07b}
we have: 

\begin{theorem}\label{thm:ExMuNu}
There exists a measure $\mu_\nu$ on $[0,\infty)$ such that:
\begin{itemize} 
\item[\rm (i)] For $\nu \geq  \rho = \frac{n-1}{2}$, the 
measure $\mu_\nu$ is absolutely 
continuous with respect to Lebesgue measure on 
the open interval $(\rho, \infty)$.  The representations
$\pi_m$ are the unitary principal series representations.
\item[\rm (ii)]  For $\frac{n-2}{2} < \nu < \frac{n-1}{2}$, we have an additional 
point mass in a point $m_\nu = \sqrt{\rho^2 - \lambda_\nu^2}$, 
where $\lambda_\nu = \rho -\nu$. The corresponding representations corresponds to the complementary
series representations. 
\item[\rm (iii)]  For $\nu = \frac{n-2}{2}$, the measure $\mu_\nu$ is a point mass 
in $m_\mu$.  
\item[\rm (iv)] For the minimal positive value $\nu_{\rm min} = \frac{n-2}{2}$, 
the corresponding representation of 
$\SO_{2,n}(\R)_0$ actually restricts to an irreducible representation 
of $G^c$ belonging to the complementary series. 
\end{itemize}
\end{theorem}

\subsection{Boundary values on the de Sitter space}\label{se:boundaryVal}

In the last section we discussed the role of the homogeneous space 
$\Lnp$ for the identification of the positive definite kernels $\Phi_m^c$ and the
corresponding unitary representations.   Here we briefly 
relate our work to analysis on the other boundary orbit $\dS^n$ 
(de Sitter space).

Let $\cB (\dS^n)$ be the space of holomorphic functions $F$ on $\Xi$ 
extending to continuous functions on $\Xi \cup \dS^n$. 
We will also write $\cB ( \dS^n\times \Xi )$ for the space of 
sesquiholomorphic kernels $\Phi \: \Xi\times \Xi \to \C$ 
extending to a continuous function on 
$(\dS^n \cup \Xi) \times \Xi$. We
then write  $\beta (\varphi )(y,w) =\beta (\varphi (\cdot ,w))(y)$. Similarly we define $\cB (\Xi \times \bS^n)$ as the space 
of kernels extending continuously to 
$\Xi \times (\Xi \cup \dS^n)$. 
Then $\beta (\varphi )(z,y)$, $(z,y)\in \Xi \times \dS^n$, is well defined.

We use the notation from previous sections:
\[a_w=\begin{pmatrix} \cosh (w) & 0 & -i\sinh (w)\\ 0 & \rI_{n-1} & 0\\
i\sinh (w) & 0 &\cosh (w)\end{pmatrix},\quad z_w=a_w.e_0=\cosh (w)e_0+i\sinh (w)e_n,\quad w\in \C .\]
Let $w= t+ir\in \R +i(-\pi /2 , \pi /2)$. Then   
\[\cosh (t- ir)=\frac{1}{2}\left(e^{t}e^{-ir} + e^{-t}e^{ir}\right) 
\stackrel{r\to \pi/2}{\sarr} -i\sinh (t)\]
and 
\[ \sinh (t-ir)=\frac{1}{2}\left(e^{t}e^{-ir} 
- e^{-t}e^{ir}\right)\stackrel{r\to \pi/2}{\sarr}-i\cosh(t). \]
Thus

\[\lim_{r\nearrow \pi/2} a_{t-ir}.e_0 = -i\sinh (t)e_0 + \cosh (t) e_n\in \dS^n\, .\]
Taking $t=0$ gives:
 
\begin{lem} Let $F\in \cB (\dS^n)$ and $g\in G^c$. Then
$\beta (F) (ge_n)=\lim_{r\nearrow \pi/2} F(ga_{-ir}.e_0)$.  
\end{lem}  
 
 The following is well known in general but we give a simple proof suitable for our special situation:
 
\begin{lemma} For the stabilizer $H := G^c_{e_n}$, we have $G^c = HAK=KAH$.
\end{lemma}

\begin{proof} That $HAK=KAH$ follows by taking inverses. 
Thus we only have to prove that
$G^c=KAH$. This assertion is equivalent to $KA.e_n=\dS^n$. Let
$v=(iv_0,\bv)\in\dS^n$. Then $-v_0^2+\|\bv\|^2=1$. Let $t\in \R$ be such that
$v_0=\sinh (t)$ and $\|\bv\|=\cosh (t)$ and let $k\in K$ be such that 
$k.e_n=\frac{1}{\cosh (t)} \bv$. Then, as $a_{-t}.e_n=i\sinh (t)e_0+\cosh (t)e_n$, we get
$ka_{-t}.e_n=v$.
\end{proof}

\begin{lemma}\label{le:ImP2} We have $\{\lf{z}{x}\: z\in \Xi, x\in\dS^n\}\cap \R =(-1,1)$. In
particular, if $z\in\Xi$ and $x\in \dS^n$, then $ [z,x]\in \C\setminus ((-\infty ,-1]\cup [1,\infty))$. 
\end{lemma}
\begin{proof}
As $\dS^n=G^c.e_n$ and $\lf{\cdot}{\cdot }$ is
$G^c$ invariant, we can assume that $x=e_n$. 
Write $z=ga_{-is}.e_0=g(\cos (s)e_0+\sin (s)e_n)$ with $|s| < \pi/2$, 
and $g=ha_tk$ with $h\in H$, $a_t\in A$ and $k\in K$. Note that 
\[k.(\cos (s )e_0 +\sin (s)e_n)= \cos (s)e_0+\sin (s)u\quad \text{ for some} \quad  u\in e_0^\perp \cap \bS^n.\]
As $e_n$ is $H$-invariant, we get
\begin{align*}
 \lf{z}{e_n}&= \lf{ha_t(\cos (s)e_0+\sin (s)u)}{e_n}\\
&=  \lf{a_t(\cos (s)e_0+\sin (s)u)}{e_n}
=  \lf{\cos (s)e_0+\sin (s)u}{a_{-t} e_n}\\
&=  \lf{\cos (s)e_0+\sin (s)u}{i \sinh(t)e_0 + \cosh(t) e_n}\\
&=u_n\cosh (t)\sin (s) +i \sinh (t) \cos (s).
\end{align*}

Thus $\lf{z}{e_n}\in \R$ implies $t=0$ because $\cos(s)>0$
for $|s|<\pi/2$. 
Then $\cosh (t)=1$ and $\lf{z}{e_n}=u_n\sin s\in (-1,1)$ 
because $|\sin (s)|<1$ and $|u_n|\le 1$. That $u_n\sin (s)$ can 
take any value in $(-1,1)$ is clear. This proves the lemma.
\end{proof}
 
\begin{Theorem} Let $z,w\in \Xi$ and $x,y\in \dS^n$. Then the following 
assertions hold:
\begin{itemize}
\item[\rm (i)] $\Phi_m^c \in \cB(\dS^n\times \Xi)$ and 
\[ \beta (\Phi_m^c) (x, w)=  \hgf \left(\rho +\lambda ,\rho -\lambda; n/2 ; 
(1-\lf{x}{ \sx w})/2\right) .\] 
\item[\rm (ii)]  $\Phi_m^c \in \cB (\Xi\times \dS^n)$ and  
\[\beta (\Phi_m^c) (z, y) = \hgf \left(\rho +\lambda ,\rho -\lambda; n/2 ; 
(1+\lf{z}{ y})/2\right). \]
\item[\rm (iii)] If $\nu \ge \frac{n-2}{2}$, then $Q_\nu \in \cB (\dS^n\times \Xi)\cap \cB (\Xi\times \dS^n)$ and
\[\beta (Q_\nu)(z,y)= \left(\frac{1-\lf{z}{y}}{2}\right)^{-\nu}
\quad\text{and}\quad \beta(Q_\nu)(x,w)= \left(\frac{1+\lf{x}{\sx w}}{2}\right)^{-\nu}.\]
\end{itemize}
\end{Theorem}

\begin{proof} The kernel 
$\Phi_m^c$ extends to a real analytic function on an open subset 
of $V_\C\times V_\C$ 
containing 
$(\Xi \cup \dS^n) \times \Xi$ (Theorem~\ref{th:Psi} and Lemma \ref{le:ImP2}). 
It follows that the boundary value exists and is
given by the value at the point.  

\nin (ii) follows in the same way noting that
$\dS^n\subset iV$ so that $\sx|_{\dS^n}= -\id$. 

\nin (iii) follows also in the same way as $Q_\nu$ is continuous on $\Xi\times \overline{\Xi}$
and $\overline{\Xi}\times \Xi$.
\end{proof}

Recall the reproducing kernel Hilbert space $\cO_m(\Xi)\subeq \cO(\Xi)$ 
with the kernel $\Phi_m^c$.  
Then $\cO_m (\Xi)=\cP_\lambda\cH_\lambda$ (Theorem~\ref{thm:Pm}). In particular
 $\cO_m(\Xi )$, with left translation as representation, is isomorphic to $(\pi_m,\cH_m)$.
For $y \in \dS^n$, the function $\eta_y:= \beta (\Phi_{m}^c)(\cdot,y)$ is 
holomorphic on $\Xi$. We claim that it does not belong
to $\cO_m(\Xi)$. To this end, we first observe that 
it is invariant under the non-compact stabilizer group $G_y$,  
but for irreducible unitary representations of $G^c$, stabilizer 
subgroups are compact by the Howe--Moore Theorem~\cite[Thm. 5.1]{HM79}.

We claim that $\eta_y$ defines a distribution vector, i.e., 
an  element of $\cO_m(\Xi )^{-\infty }$. 
Let $\cO_m(\Xi)^c \subeq \cO_m(\Xi)$ denote the linear subspace 
of all functions extending continuously to $\Xi \cup \dS^n$. 
It is dense because it contains all elements 
$\Phi^c_m(\cdot, w)$, $w \in \Xi$. 
In particular,  it contains $\varphi_m = \Phi^c_m(\cdot, e_0)$. 
Since each function $\Phi^c_m(\cdot, w)$ extends to a smooth function 
on an open subset containing $\Xi \cup \dS^n$, 
the subspace $\cO_m(\Xi)^c$ contains also 
the subspace $L_{U(\fg )}\varphi_m$ of $K$-finite functions in $\cO_m(\Xi)$. 
The Automatic Continuity Theorem \cite[Thm.~1]{BD92} then implies that 
$\ev_y \: \cO_m(\Xi)^c \to \C, f \mapsto \oline{f(y)}$ 
extends to a $G_y$-invariant distribution vector. The corresponding 
holomorphic function on $\Xi$ is given by 
\[ \Xi \to \C, \quad z \mapsto \oline{\ev_y(\Phi^c_m(\cdot, z))} 
= \oline{\Phi^c_m(y, z)} = \Phi^c_m(z,y) = \beta(\Phi^c_m)(z,y) 
= \eta_y(z).\]

According to \cite{GKO03, GKO04}, 
we can define the Hardy space $\rH^2(\dS^n)$ in the following way. 
The action of $\SO_{2,n}(\R)$ on $\Xi$ can be extended to the open semigroup 
\begin{equation}
  \label{eq:defS}
 S:=\{\gamma\in \SO_{2+n}(\C)\: \gamma^{-1}\overline{\Xi } 
 \subset \Xi\}\not= \emptyset.
\end{equation}
We define
\[\rH^2(\dS^n)=\Big\{\psi \in \cO (\Xi )\: \sup_{\gamma \in S}\int_{\dS^n}|\psi (\gamma^{-1}.x)|^2 dx<\infty\Big\}.\]
Then $\rH^2(\dS^n)$ is a Hilbert space with norm 
\[\left(\sup_{\gamma \in S}\int_{\dS^n}|\psi (\gamma^{-1}.x)|^2 dx\right)^{1/2}\]
(\cite[Cor.~5.3]{GKO03}). As a unitary 
representation of $G^c$ the Hardy space $\rH^2(\dS^n)$ is isomorphic
to 
\[L^2(\wH )\simeq \int_{m>\rho}^\oplus  (L,\cO_m(\Xi))\, d\mu_H(m) 
\simeq \int_{m>\rho}(\pi_m,\cH_m)\, d\mu_H (m)\]
where the measure $\mu_H$ is equivalent to Lebesgue measure on 
$(\rho,\infty)$.

The evaluation maps $\ev_z(\phi) := \phi (z)$, 
$z \in \Xi$, are continuous on $\rH^2(\dS^n)$ and hence given by a positive definite $\SO_{2,n}(\R)$-invariant
kernel on $\Xi \times \Xi$. This kernel is, 
up to multiplication with  a positive function, 
the kernel $Q_{n/2}$ (\cite[Thm. C]{GKO03}).

\begin{thm} Let $\nu > \frac{n-2}{2}$ and let  $\mu_\nu$ be the measure from
{\rm Theorem \ref{thm:ExMuNu}}. Then
\[Q_{\nu}(z,y)=\int_0^\infty \Phi_m^c(z,y)\, d\mu_\nu (m)\quad 
\text{and}\quad Q_{\nu} (y,z)=\int_0^\infty \Phi_m^c(y,z)\, d\mu_\nu (m)
\quad \mbox{ for } \quad (z,y) \in \Xi \times \dS^n.\]
\end{thm}

\begin{proof} For $g\in \SO_{n+2}(\C)$, let $g^\sharp=\overline{g}^{-1}$, 
where $\oline g$ denotes complex conjugation with respect to 
$\SO_{2,n}(\R)$. Then the semigroup $S$ from \eqref{eq:defS} 
is $\sharp$-invariant. Furthermore, analytic continuation implies that 
\[ Q_\nu (\gamma .z,w)=Q_\nu (z,\gamma^\sharp.w), \quad 
\Phi_m^c (\gamma.z,w)=\Phi_m^c(z,\gamma^\sharp.w)\quad \mbox{  for } \quad 
\gamma\in S^{-1}, (z,w)\in\Xi\times \oline{\Xi}.\]
For $z\in\Xi$ and $y\in \dS^n$, 
there exists an element $\gamma \in S^{-1}$ such that $\gamma^{-1} . z\in \Xi$ and
$\gamma^*.y\in \Xi$. Hence
\begin{align*}
Q_\nu (z,y) &= Q_\nu (\gamma \gamma^{-1}.z,y)
= Q_\nu (\gamma^{-1}.z,\gamma^\sharp.y)
=\int_{\R_+} \Phi_m^c(\gamma^{-1}.z,\gamma^\sharp.y)\, d\mu_\nu (m)\\
&=\int_{\R_+} \Phi_m^c (z,y)\, d\mu_\nu (m).
\qedhere
\end{align*}
\end{proof}

\subsection{Further examples}
\mlabel{subsec:6.3} 

As the classification shows, there are interesting examples of dissecting involution on non-Riemannian symmetric spaces. 
We discuss here very briefly some examples. 
For a classification we refer to \cite{NO19}. 

\begin{ex} (Euclidean space) (cf.~\cite{NO15a})   
Let $E := \R^n$ denote euclidean $n$-space and consider the 
euclidean motion group 
\[ G :=E \rtimes \OO_n(\R) \quad \mbox{ and } \quad 
K = \OO_n(\R). \] 
Then $E\cong G/K$ is a flat Riemannian symmetric space corresponding to the involution $(x,g)\mapsto (-x,g)$.  

The reflection $\sigma(x_0, \bx) = (-x_0,\bx)$ 
defines a dissecting reflection  such that $\sigma(g.x)=\tau(g).\sigma(x)$ 
holds for $\tau(x,g):=(\sigma(x), \sigma g \sigma)$. 
We then have 
\[ E_+:=\{(x_0,\bx)\:  x_0>0\}\quad \mbox{ and }
\quad E^\sigma =\{(0,\bx)\: \bx\in\R^{n-1}\}\simeq \R^{n-1}.\] 
For $m> 0$, the distribution $(m^2 - \Delta)^{-1} \delta_0$ 
is represented by  a 
rotation invariant analytic function $\phi_m$ on $\R^n \setminus \{0\}$ 
with a singularity in $0$ (for $n > 1$) (\cite{NO15a}). 
The corresponding distribution kernel 
$\Phi_m(x,y) = \phi_m(x-y)$ is singular on the diagonal, and the flipped kernel 
$\Psi_m(x,y) = \phi_m(x - \sigma(y))$ is analytic on 
$E_+ \times E_+$. 

The subspace $E^c := \R i e_0 \oplus \R^{n-1}$ carries the Lorentzian form 
$[(i x_0, \bx), (i y_0, \by)] = x_0 y_0 - \bx \by$ with the open 
forward light cone 
\[ E^c_+ := \{ (x_0,i\bx) \: x_0 > 0, x_0^2 - \bx^2 > 0\}.\] 
The $c$-dual group 
$G^c := E^c \rtimes \SO_{1,n-1}(\R)_0$ is the identity component of the 
corresponding isometry group (the Poincar\'e group). 
The corresponding tube domain is 
\[ \Xi := G^c.E_+ = G^c.(0,\infty)i e_0 = E^c + i E^c_+ =: T_{E^c_+}. \] 

One can show that also in this context, 
the kernel $\Psi_m$ extends to a sesquiholomorphic $G^c$-invariant kernel 
on $\Xi$. The boundary values on $E^c$ of the function 
$\psi_m(z) := \Psi_m(z,0)$ is a  distribution 
$D_m$, satisfying the {\it Klein--Gordon equation} 
\[ (m^2 - \square) D_m= 0.\] 
It is the Fourier transform 
of the $\SO_{1,n-1}(\R)_0$-invariant measure on the hyperboloid 
\[ \cO_m = \big\{ (x_0, \ldots, x_{n-1}) = (x_0, \bx) \: x_0^2 - \bx^2 = m^2, 
x_0 > 0\}.\] 
The corresponding $L^2$-space 
carries an irreducible unitary representation of~$G^c$, cf.\ 
Remark~6.11 in \cite{NO15a}. 
\end{ex}

\begin{ex} Besides $\R^n$, the preceding discussion also applies to 
quotients of $\R^n$ by discrete $\sigma$-invariant subgroups $\Gamma$. 
A particularly interesting case is the torus 
$\T^n$ with $\sigma(z_1, \ldots, z_n) = (\oline{z_1}, z_2, \ldots, z_n)$ 
(cf.\ \cite[\S VIII]{Ja08}). 
\end{ex}

\begin{ex} (Hyperbolic space) 
In the $n+1$-dimensional Minkowski space $V := \R^{1,n}$, we consider 
the hyperbolic space 
\[ \bH^n = \{ (x_0, \bx) \: x_0^2 - \bx^2 = 1, x_0 > 0\} \] 
on which the group $G := \OO_{1,n}(\R)^\uparrow$ acts, 
and the dissecting involutive automorphisms $\sigma$ of $\bH^n$,  
defined by the reflection $r_1 \in G$. Then 
$\bH^n_0 := (\bH^n)^\sigma \cong \bH^{n-1}$, and we put 
\[ \bH^n_\pm := \{ x \in \bH^n \: \pm x_1 > 0\}.\] 

We also write 
\[ [z,w] := z_0 w_0 - \bz \bw \] 
for the complex bilinear extension to $V_\C = \C^{n+1}$, so that 
\[  \bH^n_\C := \{z \in V_\C \: [z,z] = 1\} \cong \bS^n_\C \] 
is the complex sphere. 
On the dual space 
\[ V^c := \R e_0 \oplus \R i e_1 \oplus \R^{n-1} \subeq V_\C 
= \C^{n+1}, \] 
we have 
\[ [(x_0, i x_1, x_2, \ldots,x_n), (y_0, i y_1, y_2, \ldots,y_n)] 
= x_0 y_0 + x_1 y_1 - x_2 y_2- \cdots - x_n y_n,\] 
and this form is invariant under the action of the connected group 
\[ G^c := \SO_{2,n-1}(\R)_0.\] 
The stabilizer of $e_1$ in $G^c$ is the subgroup
$H := G^c_{e_1} \cong \SO_{1,n-1}(\R)_0$, in particular it is connected. 
Since it acts transitively on $\bH^n_0$, we have 
\[ \bH^n_\pm = H.\Exp_{e_0}(\pm (0,\infty)e_1).\] 
Accordingly, we obtain two $G^c$-invariant subsets 
\[ \Xi_\pm := G^c.\bH^n_\pm = G^c.\Exp_{e_0}(\pm (0,\infty)e_1) 
\subeq \bH^n_\C.\] 

To the two non-convex open cones 
$\Omega_\pm := \{ v \in V^c \: \pm [v,v] > 0\}$ 
we associate non-convex tube domains 
\[ T_{\Omega_\pm} := V^c \oplus i \Omega_\pm.\] 

We claim that 
\[ T_{\Omega_+} \cap \bH^n_\C = \Xi_+ \dot\cup \Xi_- = G^c.\Exp_{e_0}(\R^\times e_1)\]  
(cf.\ \cite[Lemma~3.2]{BEM02}).
In fact, for $t \not=0$, we have 
\[ \cosh(t) e_0 + \sinh(t) e_1 
= \cosh(t) e_0 + i \sinh(t) (-i e_1) \in T_{\Omega_+} \] 
because $[-i e_1, -i e_1] = - [e_1, e_1]  =1$, 
so that 
$G^c.\Exp_{e_0}(\R^\times e_1) \subeq T_{\Omega_+}$ follows from the 
$G^c$-invariance of $T_{\Omega_+}$. For the converse, 
let $z = u + i v \in T_{\Omega_+} \cap \bH^n_\C$. 
Then 
\[ 1 = [z,z] = [u,u] - [v,v] + 2 i [u,v] \] 
is equivalent to 
\[ [u,v] = 0 \quad \mbox{ and } \quad 
[u,u] = [v,v] + 1 > 1.\] 
Therfore $(u,v)$ is the basis of a positive $2$-plane in $V^c$. 
Hence the $G^c$-orbit of this pair contains an element of the form 
$(\lambda e_0, \mu i e_1)$ with  $\lambda  = \sqrt{1 + \mu^2}$ 
and $\mu \not=0$. Depending on the orientation of the basis $(u,v)$, 
we have $\mu > 0$ or $\mu < 0$. 
For $\mp\mu > 0$ we get $u + i v \in \Xi_\pm$, and this proves 
our claim. 

We thus obtain a situation very analogous to what we have 
seen in Subsection~\ref{se:SphCr} of the sphere. 
In the physics literature, the domains $\Xi_\pm$ are called 
{\it chiral tubes}. They carry an action of $G^c$ by  holomorphic maps, 
and both contain the {\it anti de Sitter space} 
\[ \AdS^n :=  \{ x \in V^c \: [x,x] = 1 \} =  G^c.e_0  
\cong \SO_{2,n-1}(\R)_0/\SO_{1,n-1}(\R)_0\] 
in their boundary. 

However, Theorem~\ref{the:ctmSp} does not apply here because 
our base point $m_0  = e_0 \in \bH^n$ is fixed by  the involution 
$\sigma$. Hence there is no Riemannian symmetric space 
in $\Xi_\pm$, such as the hyperbolic space for the sphere. 
Here one translates directly between $\bH^n$ and its Lorentzian dual 
$\AdS^n$ by first extending a positive definite analytic kernel 
$\Psi_m$ on $\bH^n_+ \times \bH^n_+$ to a sesquiholomorphic 
kernel on $\Xi_+$ and then taking boundary values on $\AdS^n$. 
Although from a slightly different perspective, this program 
is carried out to some extent in \cite{BEM02}. 
\end{ex} 

\subsection{Extension to anti-unitary representations}
\mlabel{subsec:6.4} 

A {\it graded group} $(\cG,\eps)$ is a pair of a group $\cG$ and a 
homomorphism $\eps \: \cG \to \{\pm 1\}$. A~typical example is the full Lorentz group $\OO_{1,n}(\R)$ 
with $\eps$ defined by $g V_+ = \eps(g) V_+$, for which 
$\ker \eps = \OO_{1,n}(\R)^\uparrow$. Another important example 
is the group $\AU(\cH)$ of all unitary and antiunitary operator 
on a complex Hilbert space $\cH$ with $\ker \eps = \U(\cH)$. 
We call a homomorphism 
$U \: \cG \to \AU(\cH)$ an {\it antiunitary representation} 
if $\eps_{\AU(\cH)}(U(g)) = \eps_\cG(g)$ for $g \in \cG$ 
(see \cite{NO17} for more background on these concepts).

In this section we discuss the extension of the 
representation $(\pi_m, \cH_m)_{m \geq 0}$ of 
$G^c$ to antiunitary representations of $\OO_{1,n}(\R)$.  
Let $\Psi\in \Gamma$ and denote the corresponding reproducing kernel Hilbert space 
by  $\cH_\Psi\subeq \cO(\Xi)$. 

\begin{lemma}\label{le:Conju} Let $\Psi \in \Gamma$ 
and $\sigma  \: \Xi \to \Xi$ be an  antiholomorphic 
involution extending an isometry of $\bH^n_V$ fixing~$e_0$.
 We further assume the existence 
of an involution $\sigma^G$ on $G^c$ with 
\begin{equation}\label{eq:covar}
 \sigma(g.z) = \sigma^G(g).\sigma(z) \quad \mbox{ for } \quad 
g \in G^c, z \in \Xi.
\end{equation}
Then 
\[ \Psi (z,w)=\Psi (\sigma(w),\sigma(z)) \quad \mbox{ for } \quad 
z,w \in \Xi.\] 
\end{lemma}

\begin{proof} We consider the sesquiholomorphic kernel 
\[ \tilde\Psi (z,w):=\Psi (\sigma(w),\sigma(z)) \quad \mbox{ for } \quad 
z,w \in \Xi.\] 
The relation \eqref{eq:covar} implies that 
the kernel $\tilde\Psi$ is $G^c$-invariant. Furthermore $\tilde\Psi$
is holomorphic in the first and antiholomorphic in the second argument. 
Next we observe that 
\[  \tilde\Psi^*(z,w) 
= \oline{\Psi (\sigma(z),\sigma(w))}
= \Psi (\sigma(w),\sigma(z)) = \tilde\Psi(z,w),\] 
so that $\tilde\Psi$ is hermitian, and thus $\tilde\Psi \in \Gamma$. 

To show that $\tilde\Psi = \Psi$, by Corollary~\ref{co:UniDet}, it suffices 
to show that $\tilde\Psi_{e_0}|_{\wH}= \Psi_{e_0}|_{\wH}$, i.e., that 
$\psi := \Psi_{e_0}$ is real-valued on $\bH^n_V$. 
As $\psi(\sigma(x)) = \oline{\psi(x)}$ and $\sigma \in K \cong \OO_n(\R)$, 
this follows from the $K$-invariance of $\psi$. 
This completes the proof.
\end{proof}

It follows that we can define a conjugation on $\cH_\Psi$ by 
\[J\Big(\sum_j c_j \Psi_{w_j}\Big):=\sum_j \oline{c_j} 
\oline{\Psi_{\sigma(w_j)}}, \quad 
c_j \in \C, w_j \in \Xi,\]
resp., 
\[ (Jf)(z) := \oline{f(\sigma(z))}, \qquad z \in \Xi.\]

\begin{lemma}\label{le:UniqJ} The following assertions hold:
\begin{itemize}
\item[\rm(i)] $\cH_\Psi^J=\{\phi\in\cH_\Psi \: J\phi =\phi\}$ is the 
closure of the real linear span of $\Psi_y$, $y\in \bH^n_V$.
\item[\rm(ii)] $J\pi_\Psi (g)=\pi_\Psi (\sigma^G(g))J$ for $g\in G^c$.
\item[\rm(iii)] If $\Psi =\Phi^c_m\in \Gamma_e$ and 
$J_1 :\cH_m\to \cH_m$ is a conjugation satisfying {\rm(ii)}, 
then $J_1=\mu  J$ for some $\mu \in\T$.
\end{itemize} 

\end{lemma}
\begin{proof} (i) and (ii) follow directly from the definition and Lemma \ref{le:Conju}. 

For (iii) we recall first that the representation of $G^c$ on $\cH_m$ is irreducible.
If $J_1$ satisfies~(ii), 
then $J_1J:\cH_m\to \cH_m$ is a unitary intertwining operator. By 
Schur's Lemma there exists $\lambda\in \C$ such that $J_1J=\lambda \id$ and now 
$|\lambda| = 1$ by unitarity. 
\end{proof}

\begin{ex} The assumptions of Lemma~\ref{le:Conju} are in particular 
satisfied with $\sigma^G = \id_G$ and $\sigma = \sigma_V\res_{\Xi}$. 
By Lemma \ref{le:UniqJ} and the fact that $-\1$ commutes with $G^c$,
we can extend $\pi_\Psi$ to an antiunitary representation of 
$\OO_{1,n}(\R) \cong G^c \rtimes \{\1,\sigma^G\}$ by
\[\pi_\Psi (-\1):= J, \qquad (Jf)(z) := \oline{f(\sigma_V(z)}.\]
For every $X\in \fg^c$ we then have $J\dd \pi_\Psi (X)=\dd \pi_\Psi (X)J$. In particular
\[J\pi_\Psi (\exp tX)=\pi_\Psi(\exp tX)J \quad \mbox{ for } \quad t \in \R.\]
\end{ex}

We now assume that $\Psi = \Psi_m$. Let $X\in \fg^c$ be so that $\theta(X)=-X$ 
(so $X$ is hyperbolic (a boost)) and $\ad X$ has eigenvalues $\pm \lambda $ 
and $0$ with $0<\lambda < 1/4$. 
Consider the homomorphism 
\[ \gamma_X \: \R^\times \to 
\OO_{1,n}(\R) \cong G^c \rtimes \{\1,\sigma^G\}, \qquad 
\gamma_X(e^t) := \exp(t X), \quad \gamma_X(-1) :=-\1 \]  
of graded groups. Then $\pi_\Psi \circ \gamma_X \: \R^\times \to \AU(\cH)$ 
is a strongly continuous antiunitary one-parameter group taking on $-1$ the 
value~$J$. The corresponding modular operator is given by 
\[ \Delta_X := e^{2\pi i\partial\pi_\Psi(X)},\] 
so that the corresponding standard subspace is 
\[ V_X :=  \{ v \in \cD(\Delta_X^{1/2}) = \cD(e^{i \pi\partial\pi_\Psi(X)}) \: 
Jv = e^{i \pi\partial\pi_\Psi(X)}\}.\] 
If $\|\ad X\| < \frac{1}{4}$, then the elements 
$\Psi_z$, $z \in \bH^n_V$, are contained in 
the domain of $\Delta_X$ by \cite{KS05} with 
\[ \Delta_X^{1/4} \Psi_z = \Psi_{\exp(-i\pi X/2).z} 
= J \Psi_{\exp(i\pi X/2).z} 
= \Delta_X^{1/2} \Psi_{\exp(i\pi X/2).z},\] 
so that the correspond standard subspace $V_X$ 
is generated by the elements $\Psi_{\exp(-i\pi X/2).z}$, $z \in \bH^n_V$.
In particular, $V_0 = \cH_\Psi^J$ is generated by $\Psi_z$, $z \in \bH^n_V$. 
For the most general statement about analytic extension of
coefficient functions we refer to \cite{LP18,KS05} and the references therein.

\end{document}